\documentclass[oneside]{amsart}
 
\usepackage{hyperref}
\usepackage{cleveref}
\usepackage{tikz-cd}

\usepackage{chngcntr}
 
\usepackage{lscape}

\usepackage[all,cmtip]{xy}
\usepackage{amsmath}
\usepackage{amscd} 
 
\usepackage{amssymb}
\usepackage{mathtools}
\usepackage[small,nohug,heads=vee]{diagrams} 
\newcommand{\mycolim}[1]{\mathbin{\operatorname*{colim}_{#1}^{}}}
\newcommand{\myhocolim}[1]{\mathbin{\operatorname*{hocolim}_{#1}^{}}}

\newcommand{\myotimes}[1]{\mathbin{\operatorname*{\otimes}_{#1}^{}}}

\usepackage{lscape} 
\makeindex
\usepackage{tikz}

\makeatletter  
\renewcommand\Huge{\@setfontsize\Huge{19pt}{18}}
\renewcommand\huge{\@setfontsize\huge{30pt}{18}}

\usepackage{caption}

\usepackage{graphicx}
 
\usepackage{mathrsfs}
\usepackage[all,cmtip]{xy}  
\usepackage{tikz-cd}
\usepackage{hyperref}
\usepackage{bookmark}
\usepackage{comment} 
\usepackage{setspace}
\usepackage{amsthm}
    \theoremstyle{plain}
\usepackage{epstopdf}
\usepackage{stmaryrd}
\usepackage{enumerate}

\usepackage{tikz}
\usepackage{ amssymb }

\DeclareMathOperator{\Conf}{Conf}

\DeclareMathOperator{\Tor}{Tor}
\DeclareMathOperator{\Alg}{Alg}

\DeclareMathOperator{\tr}{tr}
\DeclareMathOperator{\Sym}{Sym}
\DeclareMathOperator{\fib}{fib}
\DeclareMathOperator{\Tot}{Tot}
\DeclareMathOperator{\Sp}{Sp}

\DeclareMathOperator{\f}{f}  
\DeclareMathOperator{\odd}{odd}  
\DeclareMathOperator{\even}{even}  
\DeclareMathOperator{\fff}{ff}

\DeclareMathOperator{\pff}{pff}
\DeclareMathOperator{\LAA}{\mathbf{L}}

\DeclareMathOperator{\triv}{T}

\usepackage{chngcntr}

\DeclareMathOperator{\Comm}{Comm}

\DeclareMathOperator{\Susp}{Susp}
\DeclareMathOperator{\GL}{GL}

\newcommand{\HH}{H}   
\usepackage[small,nohug,heads=vee]{diagrams}
\diagramstyle[labelstyle=\scriptstyle]
\usepackage{graphicx}
\usepackage{float}
\usepackage{enumerate}
\usepackage{hhline}
\usepackage{amsthm}
\usepackage{appendix}
\newarrow {Openinto}{triangle}--->
\newarrow {Onto}----{>>}
\newarrow {Equal} ===={}

\usepackage{stackrel,amssymb}

\usepackage[small,nohug,heads=vee]{diagrams}
\diagramstyle[labelstyle=\scriptstyle]
\usepackage{graphicx}
\usepackage{float}
\usepackage{diagrams}
\usepackage{enumerate}
\usepackage{hhline}
\usepackage{appendix}

\restylefloat{figure}
\usepackage{wrapfig}
\usepackage{amsmath}
\usepackage{amssymb}
\usepackage{amsthm}
\usepackage{hyperref}
\usepackage{pinlabel}
\usepackage{array}
\usepackage{epstopdf}

\usepackage{amsthm}

    \theoremstyle{plain}

    \newtheorem{theorem}{Theorem}[section]
    \newtheorem{lemma}[theorem]{Lemma}
\newtheorem{proposition}[theorem]{Proposition} 
\newtheorem{corollary}[theorem]{Corollary}

    \newtheoremstyle{TheoremNum}
        {}{}              
        {\itshape}                      
        {}                              
        {\bfseries}                     
        {.}                             
        { }                             
        {\thmname{#1}\thmnote{ \bfseries #3}}
    \theoremstyle{TheoremNum}
    \newtheorem{theoremn}{Theorem}
     \newtheoremstyle{DefinitionNum}
        {}{}              
        {\itshape}                      
        {}                              
        {\bfseries}                     
        {.}                             
        { }                             
        {\thmname{#1}\thmnote{ \bfseries #3}}
    \theoremstyle{DefinitionNum}
    
    \newtheoremstyle{LemmaNum}
        {}{}              
        {\itshape}                      
        {}                              
        {\bfseries}                     
        {.}                             
        { }                             
        {\thmname{#1}\thmnote{ \bfseries #3}}

  \newtheoremstyle{CorollaryNum}
        {}{}              
        {\itshape}                      
        {}                              
        {\bfseries}                     
        {.}                             
        { }                             
        {\thmname{#1}\thmnote{ \bfseries #3}}

  \newtheoremstyle{PropositionNum}
        {}{}              
        {\itshape}                      
        {}                              
        {\bfseries}                     
        {.}                             
        { }                             
        {\thmname{#1}\thmnote{ \bfseries #3}}

\newarrow {Openinto}{triangle}--->
\newarrow {Onto}----{>>}
\newarrow {Equal} ===={}

\restylefloat{figure}
\usepackage{wrapfig}
\usepackage{amsmath}
\usepackage{amssymb}
\usepackage{amsthm}
\usepackage{hyperref}
\usepackage{pinlabel}
\usepackage{array}
\usepackage{epstopdf}

\usepackage{enumitem}

\newtheorem{warning}{Warning}
\newtheorem*{Inequality Lemma}{Inequality Lemma}
\newtheorem*{Compactness Lemma}{Compactness Lemma}

\newtheorem*{Lower Bound Theorem}{Lower Bound Theorem}
\newtheorem*{Upper Bound Theorem}{Upper Bound Theorem}
\newtheorem*{Non-Classifiability Theorem}{Non-Classifiability Theorem}
\newtheorem*{Essential Membership Theorem}{Essential Membership Theorem}
\newtheorem*{Generic Word Theorem}{Generic Word Theorem}

\newtheorem*{Complexity Theorem}{Complexity Theorem}
\theoremstyle{definition}
\newtheorem{definition}[theorem]{Definition}
\newtheorem{remark}[theorem]{Remark}
\newtheorem{example}[theorem]{Example}
\newtheorem{construction}[theorem]{Construction}
\newtheorem{notation}[theorem]{Notation}
\theoremstyle{remark}
 
\newcommand{\CE}{\mathrm{CE}}
\usepackage{stmaryrd}

\DeclareMathOperator{\Hom}{Hom}

\DeclareMathOperator{\sMod}{sMod}

\usepackage{caption}

\DeclareMathOperator{\Lie}{Lie} 
\DeclareMathOperator{\C}{C}
\DeclareMathOperator{\D}{D}
\DeclareMathOperator{\Q}{Q}
\DeclareMathOperator{\U}{U}
\DeclareMathOperator{\LQ}{\mathbb{L}Q}

\DeclareMathOperator{\sLie}{sLie}

\newcommand{\NN}{\mathbb{N}}
\newcommand{\AAA}{\mathbf{A}^{\mathcal{H}_u}}

\newcommand{\ZZ}{\mathbb{Z}}
\newcommand{\EE}{\mathbb{E}}
\newcommand{\QQ}{\mathbb{Q}}
\newcommand{\RR}{\mathbb{R}}
\newcommand{\FF}{\mathbb{F}}

\newcommand{\GG}{\mathbb{G}}

\newcommand{\DD}{\mathbb{D}}
\newcommand{\CC}{\mathbb{C}}

\newcommand{\LL}{\LAA^{\mathcal{H}_u}}
\newcommand{\TT}{\mathbf{T}}

\DeclareMathOperator{\Fun}{Fun}
\DeclareMathOperator{\AR}{AR}

\DeclareMathOperator{\Free}{Free}

\usepackage{wrapfig}
\DeclareMathOperator{\Ch}{Ch}

\DeclareMathOperator{\Ext}{Ext}

\DeclareMathOperator{\id}{id}
\DeclareMathOperator{\Map}{Map}

\DeclareMathOperator{\Mod}{Mod}
\DeclareMathOperator{\grMod}{grMod}

\DeclareMathOperator{\Barr}{Bar}

\newcommand{\HLie}{H^{\Lie}}
\newcommand{\HHLie}{H^{\Lie_{\mathcal{H}_u}}}

\usepackage{wrapfig} 
\title{The Lubin--Tate Theory of Configuration Spaces: I }

\date{} 
\begin{document} 
\author[D. Lukas B. Brantner]{D. Lukas B. Brantner }
\address{Lukas Brantner, Oxford University,   Universit\'{e} Paris--Saclay (CNRS)} 
\email{lukas.brantner@maths.ox.ac.uk, lukas.brantner@universite-paris-saclay.fr}

\author[Jeremy Hahn]{Jeremy Hahn }
\address{Jeremy Hahn, Massachusetts Institute of Technology}
\email{jhahn01@mit.edu}

\author[Ben Knudsen]{Ben Knudsen }
\address{Ben Knudsen, Northeastern University}
\email{b.knudsen@northeastern.edu}

\maketitle 
\begin{abstract} We construct a spectral sequence converging to the
Lubin--Tate theory, i.e. Morava  $E$-theory, of unordered configuration spaces and identify its ${\mathrm{E}^2}$-page as the homology of a Chevalley--Eilenberg-like complex for Hecke Lie algebras.
Based on this, we compute the $E$-theory of the weight $p$ summands of iterated loop spaces of spheres (parametrising the weight $p$  operations on $\EE_n$-algebras), as well as the $E$-theory 
 of the configuration spaces of $p$ points on  a punctured  surface.  We  read off the corresponding Morava $K$-theory groups, which appear in a  conjecture by Ravenel. Finally, we compute the $\FF_p$-homology of the space of unordered configurations of $p$ particles  on a punctured surface.
\end{abstract}\ \vspace{145pt} 
\tableofcontents

\newpage
\section{Introduction} \label{sec:Introduction}
Configuration spaces govern several  objects in mathematics and theoretical physics, ranging from $\EE_n$-algebras in topology to  Hurwitz spaces in  geometry and phase spaces in mechanics. 

The {ordered configuration space} of $k$ non-colliding particles in a manifold  $M$ is given by  \[\Conf_k(M)=\left\{(x_1,\ldots,x_k)\in M^k\ | \ x_i \neq x_j \mbox{ for all } i\neq j\right\},\] and the corresponding  {unordered} configuration space is defined as  $B_k(M)=\Conf_k(M)/\Sigma_k$. Here the symmetric group $\Sigma_k$ acts by permuting the particles.

It is a classical challenge to compute the homology groups of configuration spaces for various \mbox{manifolds $M$.}
When $M=\RR^n$ is $n$-dimensional Euclidean space, these groups are of particular significance as they classify Dyer--Lashof  {operations acting on the homology of $\EE_n$-algebras}, which are a valuable tool in computations \cite{kudo1956topology, browder1960homology,dyer1962homology, CohenLadaMay:HILS}. Through the Snaith splitting and its generalisations, this problem is closely  connected to the equally classical study of the homology of iterated loop spaces and other mapping spaces (cf. \cite{snaith1974stable,CohenMayTaylor:SCSC}). 

The rational homology groups of configuration spaces are well-understood in many cases of interest (cf. \cite{BoedigheimerCohenTaylor:OHCS, FelixThomas:RBNCS, MR1404924}). The mod $p$ homology groups are very computable 
when either $p=2$ or $M$ is an odd-dimensional   manifold (cf. \cite{BoedigheimerCohenTaylor:OHCS, MR975092, MR1240884}), but they remain mysterious in general, in particular when $M$ is a surface.

Knowledge is also scarce for \textit{generalised} homology theories. The simplest example of such a theory is given by complex $K$-theory, a classical invariant measuring the ``twistedness'' of a space through its complex vector bundles.
Vector bundles on configuration spaces are of particular interest in theoretical physics, where they connect to the Knizhnik--Zamolodchikov equations in conformal field theory (cf. \cite{etingof1998lectures}).

Chromatic homotopy theory provides, for every natural number $h$ and every prime $p$, two distinct generalisations of complex $K$-theory. The first is known as \textit{Morava $K$-theory} $K(h)$. 
Its value on a point is given by $K(h)_\ast \cong \FF_p[u^{\pm 1}]$, and it behaves in several ways like a field interpolating between $\QQ$ and $\FF_p$.
The second is the more refined \textit{Morava $E$-theory} $E=E_h$, also known as Lubin--Tate theory, which is a highly structured analogue of Lubin--Tate space in number theory satisfying $E_\ast  \cong W({\FF}_p)[[u_1,\ldots,u_{h-1}]][\beta^{\pm 1}]$ for $|\beta| = 2$. For $h=1$, Morava $E$-theory recovers $p$-completed complex $K$-theory, whereas Morava $K$-theory is a variant of  \mbox{complex $K$-theory mod $p$.}

\begin{wrapfigure}{r}{0.5\textwidth}
  \begin{center} \vspace{-0pt} \hspace{-10pt}
    \includegraphics[width=0.5\textwidth]{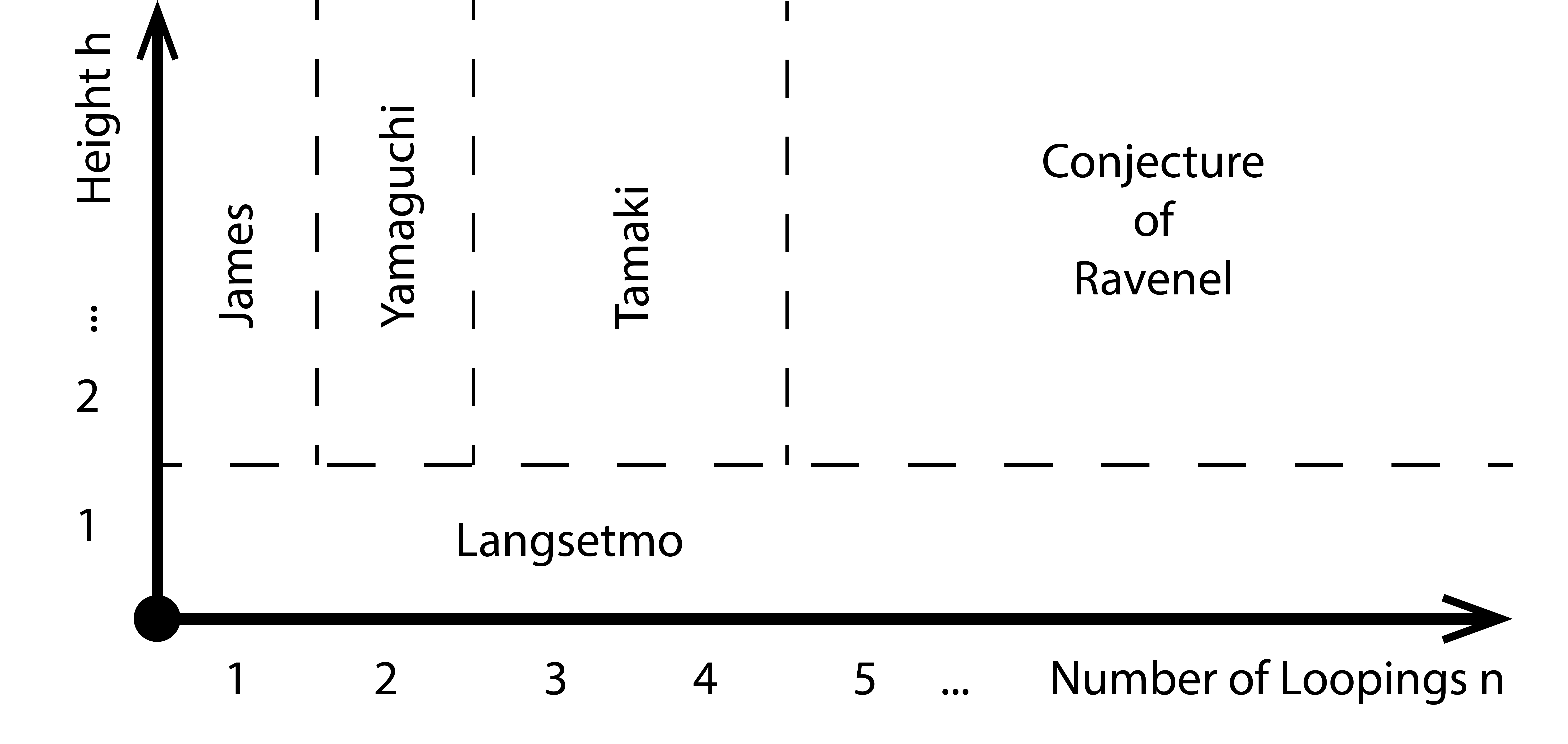}
  \end{center} 
\end{wrapfigure}The Morava $K$-theory of 2-fold, 3-fold, and even 4-fold loop spaces of spheres at arbitrary height $h$ has been computed by  Yamaguchi \cite{yamaguchi1988moravak} and Tamaki \cite{tamaki2002fiber}.  

\mbox{At height  $h$\hspace{2pt}=\hspace{2pt}$1$, the  Morava $K$-theory of $n$-} fold loop spaces of spheres 
has \mbox{been determined} by Langsetmo \cite{langsetmo1993k}, who combined an equivalence of Mahowald--Thompson  \cite{MR1153241} with McClure's computation of the \mbox{$K$-theory}  \mbox{of spaces $QX$ \cite[Chapter 9]{MR836132}.}

\mbox{For general $n$ and $h$, there is a conjectural description due to Ravenel \cite[Conjecture 3]{ravenel1998we}.}

As for $E$-theory, Langsetmo has essentially solved the height $h=1$ case \cite{MR1397734}, i.e.\ computed the \mbox{$p$-completed} complex $K$-theory of $n$-fold loop spaces of spheres for all $n$. Ravenel has formulated a conjecture concerning the $E$-homology of $\Omega^2 S^3$ \cite[Conjecture 3.4]{MR1199020} at general heights, which Goerss has linked to the theory of Dieudonn\'{e} modules (cf. \cite{MR1718079}).

In this work, we will introduce a new method for calculating the Morava $E$-theory of configuration spaces,  and apply it to perform several new computations.  Setting the height $h$ to $1$, we obtain new results about the $p$-adic $K$-theory of configuration spaces.  By taking the limit as $h$ tends to $\infty$, we additionally deduce new results about the classical, mod $p$ homology of configuration spaces of surfaces.

\subsection{Statement of Results}
The main tool used in this work is a convergent spectral sequence, together with an algebraic identification of its $\mathrm{E}^2$-page as the homology of an explicit complex. The identification of this chain complex can be viewed as our main new theoretical result. \vspace{-1pt}

The spectral sequence arises from a connection between configuration spaces and Lie algebras explored by the third author in \cite{knudsen2016higher}. Motivated by the work of Beilinson--Drinfeld on chiral algebras \cite{BeilinsonDrinfeld:CA}, as generalised and reinterpreted by Francis--Gaitsgory in \cite{FrancisGaitsgory:CKD}, this connection takes the form of an adjunction between $\mathbb{E}_n$-algebras and  {spectral Lie algebras} (in the sense of Salvatore \cite{salvatore1998configuration} and Ching \cite{ching2005bar}). The existence of this adjunction allows one to interpret the configuration spaces of $\mathbb{R}^n$ as a kind of universal enveloping algebra. Combining a version of the Poincar\'{e}--Birkhoff--Witt theorem with the theory of factorization homology \cite{AyalaFrancis:FHTM}, one obtains  a formula expressing the stable homotopy types of configuration spaces of manifolds $M$ in terms of the  {Lie algebra homology} $C^{\mathcal{L}}$ of related spectral Lie algebras.

When $M$ is  a framed manifold and $X$ is any spectrum, this   equivalence takes the form\vspace{0pt}
$$B(M;X) = \bigoplus_k B_k(M;X) \  \xrightarrow{ \ \ \simeq \ \ }  \ C^{\mathcal{L}}(\Free_{\Lie}(\Sigma^{n-1} X)^{M^+})\vspace{-2pt}, $$
where  $B_k(M;X):=\Sigma_+^\infty \Conf_k(M)\otimes_{\Sigma_k}X^{\otimes k}$.
Here $\Free_{\Lie}(\Sigma^{n-1} X)^{M^+}$ denotes   the  spectral Lie algebra of maps from the one-point compactification of $M$ to the free spectral Lie algebra  $\Free_{\Lie}(\Sigma^{n-1} X)$ on $\Sigma^{n-1} X$  (cf. Theorem \ref{thm:conf bar} below). 
As the functor $C^{\mathcal{L}}$ can be computed by a simplicial spectrum, we obtain a  spectral sequence converging to the $E$-theory of configuration spaces.
In good cases, we can identify its $\mathrm{E}^2$-page  with the derived abelianisation 
$\HH^{\Lie^{\mathcal{H}_u}}\hspace{-1pt}\left(\mathfrak{g}(M;X)\right)$
of the    {(unshifted)} Hecke Lie algebra\footnote{We will deviate in our grading conventions from \cite{brantnerthesis} and consider an unshifted variant of Hecke Lie algebras. To remind the reader of this minor difference, we will use the letter $\mathcal{H}_u$ instead of $\mathcal{H}$ throughout.} \vspace{-2pt} $$\mathfrak{g}(M;X) := E^\wedge_\ast(\Free_{\Lie}(\Sigma^{n-1} X)^{M^+}).\vspace{-2pt}$$

\textit{Hecke Lie algebra}s are  purely algebraic objects, which were introduced by the first author in order to describe the natural operations acting on the  $E$-theory of  spectral Lie algebras (cf. \cite[Theorem 4.4.4]{brantnerthesis}). Roughly speaking, Hecke Lie algebras are  Lie algebras in $E_*$-modules, equipped with an additive action by the cohomology $\Ext^\ast_{\Gamma}(\overline{E}_0, \overline{E}_0)$ of Rezk's ring $\Gamma$, which is in turn   closely related to the Hecke algebra of $\GL_n(\ZZ_p)$, cf. \cite{rezk2009congruence}, \cite[Section 14]{rezk2006units}. At the prime $p=2$, there are additional non-additive operations   witnessing certain  congruences.
To make this definition precise, one must keep careful track of  the way in which  operations compose, which is somewhat subtle as they lower homological degree. We refer to  \cite[Definition 4.4.2]{brantnerthesis} for a precise definition.\vspace{-1pt}

When working at an odd prime $p$, which we fix throughout this paper, the definition of Hecke Lie algebras simplifies significantly, and this simplification is recorded for the reader's convenience as  Definition \ref{HLA} below.  In this case, we construct an analogue $\CE_{\mathcal{H}_u}$ of the classical Chevalley--Eilenberg complex (cf. Definition \ref{def:hecke chevalley--eilenberg})  by first killing the additive operations in a derived fashion, and then taking the derived abelianisation of the resulting Lie algebra. In good cases, this complex computes Hecke Lie algebra homology---Theorem \ref{thm:hecke chevalley--eilenberg works} below. 

Combining these  observations, we arrive at the following result:\vspace{-2pt}
  
\begin{theoremn}[\ref{thm:main}] {\normalfont (Hecke spectral sequence)}
Let $M$ be a framed $n$-manifold and $X$ a spectrum, and suppose that the  Hecke Lie algebra $\mathfrak{g}(M;X):=E_*^{\wedge} (\Free_{\Lie}(\Sigma^{n-1}X)^{M^+} )$ is a finite and free $E_\ast$-module in each weight.  
There is a convergent weighted spectral sequence 
\[\mathrm{E}^2_{s,t} \cong H_{s+1}(\CE_{{\mathcal{H}_u}}\left(\mathfrak{g}(M;X))\right)_{t-1} \implies \bigoplus_{k\geq0} E_{s+t}^{\wedge}(B_k(M;X)).\]
\end{theoremn}

\begin{remark}
We will construct this spectral sequence for any  {form} of Morava $E$-theory, meaning any Morava $E$-theory associated to a formal group over a perfect field of \mbox{characteristic $p>2$.}
\end{remark}

In \Cref{Euclid}, we apply this result to compute the completed $E$-homology of the  $p^{th}$ Snaith summand of  $\Omega^nS^k$ for all $n, k$ at arbitrary chromatic height $h$, and establish the \mbox{following result:}

\begin{theoremn}[\ref{prop:general calculations}] \normalfont{($E$-theory, Euclidean case).}
Let $E$ be a Morava $E$-theory  at  an odd  \mbox{prime $p$.}
For any positive integer $n$ and integer $k$, the $E_*$-module $E_*^\wedge \left(\Conf_p(\RR^n)_+ \otimes_{h\Sigma_p} (S^k)^{\otimes p} \right)$ is  given by one of the following $E_\ast$-modules:
$$ E^\wedge_\ast( B_p(\RR^n; S^k))\cong    \begin{cases}
  \ \ \ \      \Sigma^{kp}E_\ast   \oplus \Sigma^{pk+n-1} E_*  \oplus \Sigma^{k-1} E^\ast(B\Sigma_p)/( \tr, e^{\frac{n}{2}-1} )    & \mbox{for } n \mbox{   even, } k \mbox{   even}\\
  \ \ \ \  \Sigma^{k-1} E^\ast(B\Sigma_p)/( \tr, e^{ \frac{n}{2}  }  )& \mbox{for } n \mbox{ even, } k \mbox{  odd}\\
  \ \ \ \  \Sigma^{kp}E_\ast  \oplus \Sigma^{k-1} E^\ast(B\Sigma_p)/(\tr, e^{ \frac{n-1}{2} })& \mbox{for } n \mbox{  odd, } k \mbox{  even}\\
  \ \ \ \  \Sigma^{k+ (2k+n-1)(\frac{p-1}{2})} E_\ast  \oplus  \Sigma^{k-1} E^\ast(B\Sigma_p)/(\tr, e^{  \frac{n-1}{2} })& \mbox{for } n \mbox{  odd, } k \mbox{  odd}
\end{cases} $$ 
  
Here $(\tr)$ denotes the transfer ideal associated to the inclusion of the trivial group, whereas $e\in E^0(B\Sigma_p)$ is the Euler class of the reduced complex standard representation.
\end{theoremn}

The differentials in our spectral sequence exhibit  intriguing behaviour familiar from other spectral sequences (cf.  \cite{hunter1996homology}):  there is a divided power class $\gamma_p(x)$ which affords a nontrivial differential $d^{p-1}$ landing on a $p$-torsion element; hence $x^p$ survives, while \mbox{$\gamma_p(x)$ does not.}    

\begin{remark}We can interpret \Cref{prop:general calculations} as a description of the weight $p$ power operations acting on the (completed) $E$-homology of $\EE_n$-algebras. The torsion-free classes are related to expressions $x^p   , \   \ x^{p-2}\cdot[x,x]$, \  and  \ $x\cdot [x,x]^{\frac{p-1}{2}} $ coming from the Poisson structure; the torsion classes are a new chromatic phenomenon.
\end{remark}

\begin{remark}At height $1$, our result is particularly simple and stated as \Cref{warmup} below. In this case, it can (with some care) also be read off from the work of Langsetmo (cf. \cite{langsetmo1993k,MR1397734}), who obtains this computation by entirely different means.
\end{remark}

\begin{remark} Combining \Cref{prop:general calculations} with the work of Zhu \cite{zhu20power}, we can give very concrete formulae at height $2$. For example, at $p=3$, we have  
\begin{eqnarray*}E^\wedge_*(B_3(\mathbb{R}^{11})) &\cong& E_* \oplus \Sigma^{-1} E_*[\alpha]/(\alpha^4-6\alpha^2+(h-9)\alpha-3,\alpha^5) 
\end{eqnarray*} 
We refer to \Cref{sec:height 2} for more detailed computations. 
\end{remark}
\begin{remark}
The analogue of \Cref{prop:general calculations} at $p=2$ is much easier, since the configuration spaces can be expressed in terms of real projective spaces (compare \Cref{2iseasy} below).
\end{remark}
\begin{remark}
It is not difficult to read off  the Morava $K$-theory groups $K(h)_\ast( B_p(\RR^n; S^k))$ from \Cref{prop:general calculations}, and we will record the resulting
groups in \Cref{Ktheoryatheightp} below. In \Cref{K-theory section}, we  will also  outline the connection between our results and the 
 classical works of Langsetmo, Yamaguchi, and Tamaki.
\end{remark}

The $E$-cohomology of the configuration space of $p$ points in a general punctured surface can be computed along similar lines, and we obtain the following result (again at arbitrary height):

\begin{theoremn}[\ref{thm:open surfaces}]\normalfont{($E$-theory, surface case).}  Let $E$ denote a Morava $E$-theory at an odd prime $p$. 
\begin{enumerate}[wide, labelwidth=!, labelindent=0pt] 
\item \mbox{The $E$-cohomology of the space  $B_p(\dot{T})$ of $p$ unordered points in the punctured torus satisfies}
$$E^\ast(B_p(\dot{T}))) \cong \bigg(\bigoplus_{0 \leq i<p} \Sigma^i E_\ast^{\oplus \lfloor \frac{3i+2}{2}\rfloor}\bigg)
\oplus
 \Sigma^p E_\ast^{\oplus (p+1)} . $$
\item The $E$-cohomology of the unordered configuration space of $p$ points in  a punctured orientable genus $g$ surface $\mathcal{S}_{g,1}$ is  given by 
$$E^\ast(B_p(\mathcal{S}_{g,1})) \cong   \bigoplus_{0 \leq i \leq p} \Sigma^i E_\ast^{\oplus \beta_i},  $$
where $\beta_0,\ldots, \beta_p$ are integers specified on p.\pageref{formulabetti} of the main text.  
\end{enumerate}
\end{theoremn}
\begin{remark} At height $1$, \Cref{thm:open surfaces}
 is a statement  about the $p$-adic $K$-theory of $B_p(\mathcal{S}_{g,1})$, i.e. about  vector bundles on configuration spaces of punctured surfaces. \end{remark}
\begin{remark}
The above  computation makes use of the corresponding computation over the rationals, which is originally due to B\"{o}digheimer--Cohen \cite{BoedigheimerCohen:RCCSS}, and has been revisited by the third author and Drummand-Cole in \cite{Knudsen:BNSCSVFH} and  \cite{DrummondColeKnudsen:BNCSS}.\end{remark}
\begin{remark}
The absence of torsion in \Cref{thm:open surfaces} may be thought of as a reflection of the fact that $E^\wedge_\ast(B_p(\RR^2))$  is a free $E_\ast$-module by \Cref{prop:general calculations}.
\end{remark}
\Cref{prop:general calculations} and \Cref{thm:open surfaces} illustrate how the general method introduced in \Cref{thm:main} can be used  to generate new chromatic information about specific labelled configuration \vspace{3pt} spaces.  While both statements appear to be new, particularly \Cref{thm:open surfaces} seems difficult to prove by more direct means. Preliminary calculations by the authors and Adela Zhang suggest that the main obstruction to extending our results to higher weights is the algebraic calculation of the $\mathrm{E}^2$-page, which can in principle be achieved in any finite range by computer, rather than the calculation of differentials.
Our final result is somewhat more surprising, as it gives new information about the 
$p$-primary part of the ordinary homology of configuration spaces (in the hard case where $p$ is odd and the manifold $M$ is even-dimensional): 

\begin{theorem}\label{thm:no torsion} For any odd prime $p$ and any genus $g$, the integral (co)homology of $B_p(\mathcal{S}_{g,1})$ has no $p$-power torsion.
\end{theorem}
\begin{corollary}
The $\mathbb{F}_p$-Betti numbers of $B_p(\mathcal{S}_{g,1})$ coincide with the rational Betti numbers. Hence $\dim_{\FF_p}(H_i(B_p(\mathcal{S}_{g,1});\FF_p) = \beta_i$, where the numbers $\beta_i$ are specified in \Cref{thm:open surfaces}.  
\end{corollary}

\subsection{Future directions} Our work points to  further questions, to which we \mbox{hope to return.}
\begin{enumerate} 
\item \emph{Higher weights}. We have restricted  attention to computations in Snaith weight $p$. A more organised approach will lead to similar computations in higher weights, and thereby perhaps even give  a proof of a form of  Ravenel's conjecture \cite[Conjecture 3]{ravenel1998we}   at all weights. The coherently cocommutative coalgebra structure  on stabilised configuration spaces (defined geometrically by  splitting configurations) will  be a helpful tool.  
\item \emph{Manifolds.} After incorporating certain actions of tangential structure groups,  our methods  extend to   configuration spaces of non-parallelisable manifolds. One could therefore     emulate the rational computation of \cite{DrummondColeKnudsen:BNCSS} and attempt to treat all surfaces.
\item \emph{Vector bundles}. We hope that our techniques will yield information about the complex $K$-theory of configuration spaces which is helpful for  the classification of vector bundles.
\item \emph{Coefficients}. We see this paper as a model, in the case of Lubin--Tate theory, for a program that is valid for any homology theory. 
We have seen that our methods give new information for the mod $p$ homology of configuration spaces of manifolds $M$,  about which little is known unless $\dim M$ is odd or $p=2$ (\cite{BoedigheimerCohenTaylor:OHCS}, $M=\mathbb{R}^n$ \cite[III]{CohenLadaMay:HILS}, or $M=S^2$ \cite{Schiessl:ICCSS}).

\end{enumerate}

\subsection{Linear outline} Following the introduction, Section \ref{sec:Preliminaries} recalls the basic (mostly linear) algebra underlying our work, and Section \ref{sec:SimplicialLieAlgebras} develops the theory of Lie algebra homology and the Chevalley--Eilenberg complex at a level of generality suitable for our purposes. Section \ref{sec:Hecke} goes on to produce an analogue of the Chevalley--Eilenberg complex for Hecke Lie algebras, and Section \ref{sec:SpectralSequenceInput} is concerned with the construction of the spectral sequence of interest and the identification of its second page in algebraic terms. In the final two sections, we carry out our computations.

\subsection{Acknowledgements}
The authors wish to thank Greg Arone, Paul Goerss,  Mike Hopkins, Jacob Lurie, Oscar Randal-Williams,  Neil Strickland, and  Ulrike Tillmann for helpful conversations related to this paper. 

The first author would also like to thank the Max Planck Institute in Bonn and Merton College, Oxford, for their support. Moreover, he is grateful to  the Mathematical Sciences Research Institute  in Berkeley, California, where his  work was  supported by the National Science Foundation under Grant No. DMS-1440140 during the  Spring 2019 semester.

The second author was supported by NSF Grant DMS-1803273.

The third author wishes to thank the Isaac Newton Institute for Mathematical Sciences for support and hospitality during the programme Homotopy Harnessing Higher Structures, where he benefited from the support of EPSRC grants EP/K032208/1 and EP/R014604/1. He was supported by NSF awards DMS-1606422 and DMS-190674 during the writing of this paper.

\section{Preliminaries }  \label{sec:Preliminaries}
 Let $R$ be a graded-commutative ring, where the grading is parametrised by the integers.
 In this section, we will briefly recall the basic  homological algebra of graded $R$-modules in the context of interest to us, and fix some notation for the remainder of this paper.\vspace{-2pt}
 
\subsection{Weighted graded $R$-modules} 
We start with the abelian category $\Mod_{R}$ of graded $R$-modules and grading-preserving maps between them. It  carries a symmetric monoidal structure given by the graded relative  tensor product $\otimes$. The symmetry isomorphism of $\otimes$ incorporates the  usual Koszul sign rule, which means that the isomorphism $M  \otimes N \cong N \otimes M $ sends  an element $m\otimes n $ to \mbox{$(-1)^{|m| |n|} n \otimes m$.}

The free graded $R$-algebra on a graded $R$-module $M$ takes the form $\Sym_R(M)= \bigoplus_w M^{\otimes w}_{\Sigma_w} $ and hence naturally splits as an infinite direct sum of ``weighted pieces" indexed by the naturals. In our later spectral sequences, it will be important to effectively access  weighted pieces of this kind. We therefore introduce an additional  grading:
\begin{definition}[Weighted graded  modules]
The category $\Mod^{\NN}_{R}$ of \textit{weighted graded \mbox{$R$-modules}} is given by the  category
of functors from the discrete category $\NN$ of nonnegative integers \mbox{to  $\Mod_R$.} Day convolution equips $\Mod^{\NN}_{R}$ with a 
 symmetric monoidal structure, which we will  denote \mbox{by $\otimes$.}

Concretely, an object $M\in \Mod^{\NN}_{R}$  is  simply an $\NN$-indexed collection of $\ZZ$-graded $R$-modules
$$M(0)\ , \ M(1) \ , \ M(2)  \ , \ M(3)  \ , \  \ldots .$$
Given $M, N \in \Mod^{\NN}_{R}$, the weight $w$ component of $M\otimes N$ is given by $\bigoplus_{u+v=w}M(u) \otimes N(v)$. \vspace{-2pt}
\end{definition}
We  shall call the $\ZZ$-grading the \textit{internal grading} and the $\NN$-grading the \textit{weight grading}. This means that elements  in internal degree $i$ and weight $w$   belong to the group  $M(w)_i$.
Note that  given $M, N \in  \Mod^{\NN}_{R}$, the symmetry isomorphism $M\otimes N \cong N \otimes M$ only implements the Koszul sign rule with respect to the internal grading: this means that if $m\in M(u)_i$ and $n \in N(v)_j$, then
$m\otimes n \in M\otimes N$ is sent to $(-1)^{i+j} n \otimes m$ in $ N\otimes M$.

The category $\Mod^{\NN}_{R}$ of weighted graded $R$-modules admits a canonical endofunctor:
\begin{definition}[Suspension]\label{suspension}
Given  $n\in \ZZ$, the \emph{$n^{th}$ suspension} of  $M\in \Mod^{\NN}_{R}$ is the unique weighted graded $R $-module $\Sigma^n M$ with $(\Sigma^n M)_i=M_{i-n}$ and with $R $-action inherited from $M$. 
\end{definition}
\begin{notation} Given a nonnegative integer $w\in \NN$ and 
a graded $R$-module $M\in \Mod_R$, we write $M\{w\}\in \Mod^{\NN}_{R}$ for the weighted graded $R$-module which is $M$ concentrated in weight $w$.\end{notation} 
\begin{definition}\label{freedef}
A weighted graded $R $-module  $M \in \Mod^{\NN}_{R}$
is said to be 
\begin{enumerate} 
\item  \emph{finitely generated} if the underlying ordinary $R$-module has this property;
\item  \emph{free} if it is   a direct sum of modules  {$\Sigma^rR{\{w\}} $ with $r\in \mathbb{Z}, w\in \NN$};
\item \textit{projective} if it is projective as an object in the abelian category $\Mod^{\NN}_{R}$; 
\item \textit{flat} if the functor $M\otimes -$ preserves short exact sequences.
\end{enumerate}

By the usual argument, projective modules are precisely the summands of free modules. Write $\Mod^{\NN}_{R ,\f}$ $(\Mod^{\NN}_{R ,\fff})$ for the full subcategory spanned by all (finite) free modules.

The following classical result will be indispensable \mbox{(cf. \cite{Lazard}, and \cite{FossumFosbury:CGM} for the graded case):}
\begin{theorem}[Lazard]\label{thm:lazard}
A  module $M\in \Mod^{\NN}_{R}$ is flat if and only if it is a filtered colimit of finite  free  modules.
\end{theorem}

\end{definition}

\subsection{The derived category of $\Mod_R^\NN$.} In order to compute derived functors, we will need to work in the (nonnegative) derived $\infty$-category $\mathcal{D}_{\geq 0}(\Mod^{\NN}_{R})$ of  $\Mod^{\NN}_{R}$ (cf. \cite[Section 1.3.2]{Lurie:HA}).
 As an $\infty$-category, this can be constructed by   freely adding geometric \mbox{realisations to $\Mod^{\NN}_{R ,f}$.} \\
For us, it will however be more  helpful to model $\mathcal{D}_{\geq 0}(\Mod^{\NN}_{R})$ by two concrete model categories.

\subsubsection*{Chain complexes} The first such model is given by the category $\Ch_{\geq 0}(\Mod^{\NN}_{R})$
of (homologically) nonnegatively graded chain complexes 
of weighted graded $R$-modules
$$ \ldots  \xrightarrow{\ } M_{2 }  \xrightarrow{d_{2 }} M_{1 } \xrightarrow{d_{1 }}  M_{0 } \xrightarrow{\ }  0 \xrightarrow{\ }  \ldots . $$
Chain maps, homology modules, and quasi-isomorphisms are defined in the usual way.

\begin{notation}
A chain complex   $M\in \Ch_{\geq 0}(\Mod^{\NN}_{R})$ is   equipped with 
  \mbox{three} different gradings.  The  \textit{homological grading} measures in which piece of the chain complex we are in; it is  indexed by $\NN$. The second 
and third grading use  that all  pieces $M_{n }$
are weighted graded $R $-modules; as before, \mbox{we will call them the \textit{internal grading} and \textit{weight grading}, respectively.} Hence elements of homological degree $n$,  internal degree $i$, and weight $w$ belong to the  $R_0$-module $M_{n}(w)_i$.
\end{notation}

The category of chain complexes carries a second natural endofunctor:
\begin{definition}[Shift]\label{shift}
Given an integer $n\in \ZZ$, the \emph{$n^{th}$ (homological) shift} of a chain complex \mbox{$M=(M_{i  }, d_i)$}  is given by 
the chain complex with $(M[n])_{ i }:=M_{i-n }$   with differential $(-1)^n d_i$. 
Here we use the convention that $M_i=0$ for all $i<0$.
\end{definition}
\begin{warning} The shift operator $[n]$
should not be confused with the
suspension operator $\Sigma^n$ (obtained by applying \Cref{suspension} in each component).
\end{warning}

\begin{definition}
A  chain complex in $ \Ch_{\geq 0}(\Mod^{\NN}_{R})$ is  \textit{levelwise free} if it is a free weighted graded $R $-module  in each homological degree (in the sense of Definition \ref{freedef}).
The terms \textit{levelwise finitely generated},  \textit{levelwise projective},  and \textit{levelwise flat} are defined in a similar way.
\end{definition} 

The category $ \Ch_{\geq 0}(\Mod^{\NN}_{R})$ is equipped with the (cofibrantly generated) projective model structure.
Its weak equivalences are the quasi-isomorphisms, its fibrations are the levelwise surjections, and its cofibrations are the levelwise injections with levelwise projective \mbox{cokernel.} 

Moreover,  $ \Ch_{\geq 0}(\Mod^{\NN}_{R})$ carries a natural symmetric monoidal structures given by the  usual tensor product of chain complexes. Given chain complexes $M, N \in  \Mod^{\NN}_{R}$, the symmetry isomorphism $M\otimes N \cong N \otimes M$ implements the Koszul sign rule with respect to the internal  and homological gradings, but not the weight grading: this means that if $m\in M_a(u)_i$ and $n \in N_b(v)_j$, then
$m\otimes n \in M\otimes N$ is sent to $(-1)^{i+j+a+b} n \otimes m$ in $ N\otimes M$.

\subsubsection*{Simplicial modules} \label{simplicialmodules}
 The second model for the (nonnegative) derived category is $\sMod_R^\NN$, 
the category of     simplicial objects in $\Mod^{\NN}_{R }$. These are contravariant functors from  the category $\Delta$ of nonempty finite linearly ordered sets to 
$\Mod^{\NN}_{R }$. Once more, its objects  
 carry three  different gradings, which we will  again refer to as the \textit{internal}, \textit{homological}, and \textit{weight} gradings.

\begin{definition}
An object in   $\sMod_R^\NN$    is said to be \textit{levelwise free} if it is a free weighted graded $R $-module  in each homological (i.e. simplicial) degree (cf. Definition \ref{freedef}).
The terms  \textit{levelwise finitely generated}, \textit{levelwise projective},  and \textit{levelwise flat} are defined analogously.
\end{definition}

The category $\sMod^\NN_{R }$ carries a (cofibrantly generated) model structure as well: its weak equivalences and fibrations are precisely those maps which induce weak equivalences or fibrations on underlying simplicial sets. For a well-known, more explicit  description of fibrations and cofibrations, we refer to \cite[Lemma 4.3]{schlemmer}.
The levelwise tensor product gives $\sMod^\NN_{R }$ a symmetric monoidal structure.

One pleasant feature of  $\sMod_R^\NN$ is that filtered colimits are automatically homotopy colimits:

\begin{lemma}\label{lem:filtered hocolim}
If $M:I\to \sMod_R^\NN$ is a  functor with $I$   filtered, then    $\displaystyle\myhocolim{i\in I}\ M_i\to \mycolim{i\in I}\ M_i$ is a weak equivalence.
\end{lemma}
\begin{proof}
As quasi-isomorphisms of chain complexes (hence also weak equivalences in $\sMod_R^\NN$) are detected by compact objects,  filtered colimits of weak equivalences are  weak equivalences.
\end{proof}

\subsubsection*{The Dold--Kan correspondence}
The two model categories $ \Ch_{\geq 0}(\Mod^{\NN}_{R})$ and $\sMod_R^\NN$ have the same 
underlying $\infty$-category. This fact is witnessed by a well-known Quillen equivalence arising as a composite  of adjunctions
\begin{diagram}
\Ch_{\geq 0}(\Mod^{\NN}_{R}) \  \ \  & & 
\pile{ \rTo\\ \lTo   } & & \  \ \   (\Mod^{\NN}_{R })^{\Delta^{op}_\mathrm{inj}}\  \ \   & & 
\pile{ \rTo^{\ \  \iota_! \ \ }    \\  \lTo_{ \  \ \iota^*\ \ } }&  & \  \ \    (\Mod^{\NN}_{R })^{\Delta^{op} }= \sMod^\NN_{R }.
\end{diagram}
Here $\iota:\Delta^{op}_\mathrm{inj}\rightarrow  \Delta^{op}$ denotes the inclusion of the wide subcategory on    order-preserving injections. The upper arrow on the left identifies chain complexes with semisimplicial objects $M_{  \bullet}$ for which $d_i$ vanishes on $M_{  n}$ whenever $i<n$. Its right adjoint sends a semisimplicial object $M_{  \bullet}$ to the chain complex whose $n^{th}$ term is given by
$\bigcap_{i=0}^{n-1} \ker(d_i:M_{  n}\to M_{  {n-1}})$,  {with differential $(-1)^nd_n$.} 
We write $\Gamma$ and $N$ 
for the composite left and right adjoint,  respectively, and  observe:
\begin{lemma}\label{lem:flat dold--kan}
The functors $N$ and $\Gamma$ each preserve levelwise projectivity, levelwise flatness, and levelwise finite generation. The functor $\Gamma$ also preserves levelwise freeness.
\end{lemma}
\begin{proof}
The claims concerning the right adjoint $N$ all follow from the   fact \mbox{that the inclusion} \mbox{$N(M)_n\subseteq M_n$} is split with complement the span of the images of the degeneracies with \mbox{target $M_n$.}
The assertions about the functor $\Gamma$ follow from the well-known formula $\Gamma(V)_n=\bigoplus_{\pi:[n]\twoheadrightarrow[k]} V_k,$ where the sum is indexed over the set of surjective maps in $\Delta$ with domain $[n]$.
\end{proof}
The adjunction $(\Gamma \dashv N)$ can be used to identify the two models for the derived category:
\begin{theorem}[Dold--Kan correspondence]\label{thm:dold--kan}
The pair $\Gamma: \Ch_{\geq 0}(\Mod^{\NN}_{R})\leftrightarrows \sMod^\NN_{R }: N$ forms an adjoint equivalence of categories, and in fact a Quillen equivalence of model categories.
\end{theorem}

While the Dold--Kan  correspondence is an equivalence, it does \textit{not} respect the symmetric monoidal structure. The categories of commutative algebra objects in $\Ch_{\geq 0}(\grMod_{R})$  and $\sMod^\NN_{R }$  are therefore  {not} equivalent. 
In fact, only the latter category  has a well-behaved homotopy theory for general rings $R$. 

This relies on the symmetric power functor preserving weak equivalences between cofibrant (i.e. levelwise projective) simplicial modules.
In fact, we will show in \Cref{cor:flatsym invariance} that the symmetric power functor also preserves equivalences between the slightly larger class of levelwise flat simplicial modules. This may be thought of as a non-additive variant of the flat resolution lemma in homological algebra. 

Our argument will rely on the following simplicial variant of Lazard's theorem:

\begin{lemma}[Simplicial Lazard Theorem]\label{lem:flat filtered colimit}
A simplicial module $M\in \sMod_R^\NN$ is  levelwise flat if and only if it is a filtered colimit of simplicial modules which are levelwise finite free.
\end{lemma}
\begin{proof}
The ``if'' direction is immediate from Theorem \ref{thm:lazard} above. For the converse, it suffices to show that any levelwise flat chain complex $V\in \Ch_{\geq 0}(\Mod^{\NN}_{R})$ is a filtered colimit of levelwise finite free chain complexes, since $N$ and $\Gamma$ preserve colimits, $N$ preserves flatness, and $\Gamma$ preserves freeness by Lemma \ref{lem:flat dold--kan}. We may further assume that $V_n$ vanishes above some given degree $N$, since an arbitrary chain complex is a filtered colimit of its truncations.

We will show by downward induction on $r$ that for every $r \geq 0$, the complex $V$ is a filtered colimit of complexes $\{V^i\}_{i\in I}$ satisfying the following two properties:\begin{enumerate} 
\item $V^i_n$ is finite free for $r\leq n\leq N$ and all $i\in I$.
\item the map $V^i_n\rightarrow V_n$ is an isomorphism for all $0\leq n< r$ and all $i\in I$.
\end{enumerate} 
We aim to prove the case $r=0$. The base case $r=N+1$ is \mbox{trivially true.}
Assume that the claim holds for a given $r \geq 1$, and that this is verified by a filtered diagram $\{V^i\}$ of chain complexes over $V$.
 By Theorem \ref{thm:lazard}, we can write $\mycolim{j\in J} W^j\cong V_{r-1} $  \mbox{with $J$ filtered and each $W^j$ finite  free. }

Let $K$ denote the category of triples $(i\in I,j \in J, \delta: V^i_r \rightarrow W^j)$ such that the diagram 
 \[\xymatrix{
V^i_r\ar[d]_-{\delta}\ar[r]&V_r\ar[d]^-d\\
W^j\ar[r]&V_{r-1}
}\] commutes  and  the composite $V^i_{r+1}\xrightarrow{d}V^i_r\xrightarrow{\delta} W^j$ vanishes. A morphism from $(i,j,\delta)$ to $(i',j',\delta')$ is a pair of morphisms in $I$ and $J$ compatible with \mbox{$\delta$ and $\delta'$.}

By the conditions on objects of $K$, replacing $V_{r-1}\cong V^i_{r-1}$ with $W^j$ in $V^i$ defines a chain complex, and we obtain in this way a diagram $\varphi:K\to \Ch_{\geq 0}(\Mod^{\NN}_{R})$ over $V$. 

Thus, it will suffice to show, first, that the colimit of $\varphi$ is $V$; and, second, that $K$ is filtered.

It is enough to verify the first claim in each degree $n$, ignoring the differential. There is an obvious forgetful functor $\pi:K\to I\times J$, and we have a commuting diagram \[\xymatrix{
K\ar[d]_-\varphi\ar[r]^-\pi&I\times J\ar[d]^-{\psi}\\
\Ch_R\ar[r]^-{(-)_n}&\Mod_R,
}\] where \[\psi(i,j)=\begin{cases}
V^i_n&\quad n\geq r\\
W^j&\quad n=r-1\\
V_n&\quad n<r-1.
\end{cases}\]  By the assumptions on $\{V^i\}_{i\in I}$, it suffices to check that $\pi$ is a cofinal functor, i.e. that for all $(i,j) \in I \times J$, the comma category $K_{(i,j)/}  $ is nonempty and connected.
This follows easily using that each $V^i_r$ is a compact object and that $I$ and $J$ are filtered.

The verification that $K$ is filtered proceeds in a similar manner, using compactness and the fact that $I$ and $J$ are filtered.  
\end{proof}

\subsection{Symmetric powers}
The $k^{th}$ symmetric power of a module $M \in \Mod^\NN_{R}$ is defined  as  $$\Sym_R^k(M)=(M^{\otimes k})_{\Sigma_k}.$$ This functor preserves filtered colimits.
The total symmetric power is \mbox{$\Sym_R(M) = \bigoplus_k \Sym_R^k(M)$.}

 In this section, we will study the homotopical behaviour of the functor \mbox{$\Sym_R^k:  \sMod^\NN_{R} \rightarrow  \sMod^\NN_{R}$} obtained by applying $\Sym_R^k$ in each simplicial degree.
This will lead to a   technical tool for  our subsequent study of Lie algebra homology.

It is well-known \cite{DoldPuppe} that symmetric powers preserve weak equivalences between levelwise projective modules. In fact, the following slightly stronger claim holds:
 
\begin{lemma}\label{lem:flat sym}
\mbox{If $M$ is   levelwise flat, then  $\mathbb{L}\Sym^k_R(M)\to \Sym^k_R(M)$ is a weak equivalence.}
\end{lemma}
\begin{proof}
By  Lemma \ref{lem:flat filtered colimit}, we can write $M\cong \mycolim{i\in I}M_i$ with $I$ filtered and each $M_i$ levelwise  finite free. 
Consider the following commutative square:
\[\xymatrix{
\mathbb{L}\Sym^k_R\ (\mycolim{i\in I} \ M_i)\ar[r] &\Sym^k_R(\mycolim{i\in I} \ M_i)\\
\mycolim{i\in I} \ \mathbb{L}\Sym^k_R\ (M_i)\ar[u]\ar[r]^-\sim&\mycolim{i\in I} \ \Sym^k_R\ (M_i).\ar[u]_-{}
}\] 
The right vertical arrow is an equivalence since $\Sym_R^k$ preserves filtered colimits. This implies that the left derived functor $\mathbb{L}\Sym^k$ preserves filtered homotopy colimits, which, by Lemma \ref{lem:filtered hocolim}, shows that the left vertical arrow is an equivalence.
This lower horizontal arrow is an equivalence since levelwise finite free modules are projective.
\end{proof}
\begin{corollary}[Invariance of symmetric powers for flat modules]\label{cor:flatsym invariance}
If $f:M\to M'$ is a weak equivalence of levelwise flat
 simplicial $R$-modules, then $\Sym_R^k(f)$ is a weak \mbox{equivalence for all $k\geq0$.}
\end{corollary}

We will also need the following fact in our later computations: 

\begin{lemma}\label{lem:sym projective}
Assume that $2$ is invertible in $R$. If $M\in \sMod^\NN_{R}$ is levelwise free (respectively  levelwise projective or levelwise  flat), then $\Sym_R^k(M)$ has the same property.
\end{lemma}
\begin{proof}
Since $\Sym_R^k$  satisfies a binomial formula on direct sums, it suffices to treat the case  $M=\Sigma^nR$   to verify that $\Sym_R^k$ preserves freeness.
Indeed, observe that for $k>1$, we have  \[\Sym_R^k\left(\Sigma^nR\right)\cong\begin{cases}
\Sigma^{nk}R&\quad n \text{ even}\\
0&\quad n \text{ odd}.
\end{cases}\] If $M$ is projective, then $M$ is a summand of a free $R$-module $F$, so $\Sym_R^k(M)$ is a summand of $\Sym_R^k(F)$, which is free by the previous case. If $M$ is flat, then $M$ is a filtered colimit of levelwise finite free $R$-modules by \Cref{lem:flat filtered colimit}. Since $\Sym_R^k(M)$ preserves filtered colimits, a second application of \Cref{lem:flat filtered colimit}  shows that $\Sym_R^k(M)$ is levelwise flat.
\end{proof}

\subsection{Divided Powers}\label{divpow}
The $n^{th}$ divided power of a module $M\in \Mod_R^\NN$  is given by  \[\Gamma_R^n(M)=(M^{\otimes n})^{\Sigma_n}.\] 
The functor $\Gamma_R$ preserves filtered colimits.
The total divided power is \mbox{$\Gamma_R(M) = \bigoplus_k \Gamma_R^k(M)$.}

\begin{proposition}\label{exponential} Divided powers satisfy the following well-known properties:
\begin{enumerate} 
\item There is a natural isomorphism $\Gamma_R(M_1\oplus M_2)\cong \Gamma_R(M_1)\otimes \Gamma_R(M_2)$ for all $M_1, M_2$.
\item If $M\cong R\langle x\rangle$ is free on one generator, then there is an isomorphism of $R$-modules \[\Gamma_R(M)\cong \begin{cases}
R\langle \gamma_i(x): 0\leq i< \infty,\, |\gamma_i(x)|=i|x|\rangle&\quad |x|\,\text{even}\\
R\oplus R\langle x\rangle&\quad |x|\,\text{odd}.
\end{cases}\] 
\end{enumerate}
\end{proposition} 
\begin{proof}
The first fact is standard (cf. \cite[Theorem III.4]{roby1963lois}). The second and third claim follow from \cite[Lemma 5.1.2]{Neisendorfer:AMUHT} after tensoring up from $\mathbb{Z}$ to $R$.
\end{proof}

Divided powers satisfy many of the same desirable properties of symmetric powers. The following results  are classical and proven along the same lines as in the preceeding section;   we will therefore be brief:

\begin{lemma}[Invariance of divided powers for flat modules]\label{cor:flatgamma invariance}
If $f:M\to M'$ is a weak equivalence of levelwise flat
 simplicial $R$-modules, then $\Gamma_R^k(f)$ is a weak \mbox{equivalence for all $k\geq0$.}
\end{lemma}
 
\begin{lemma}\label{lem:gamma projective}
Assume that $2$ is invertible in $R$. If $M\in \sMod^\NN_{R}$ is levelwise free (respectively  levelwise projective or levelwise  flat), then $\Gamma_R^k(M)$ has the same property.
\end{lemma}

\section{Lie algebras and their homology} \label{sec:SimplicialLieAlgebras}
Given a Lie algebra over a graded ring $R$, we can divide out  the ideal spanned by all brackets and hence obtain its  abelianisation.  Carrying out this construction in a suitably derived way leads to a definition of the Lie algebra homology  $H^{\Lie}(\mathfrak{g})$ of $\mathfrak{g}$.
In this section, we show that if $\mathfrak{g}$ is flat over $R$, then  $H^{\Lie}(\mathfrak{g})$ is computed by the  \mbox{classical Chevalley--Eilenberg complex  \vspace{-2pt}$\CE(\mathfrak{g})$ of $\mathfrak{g}$. }

\subsection{Simplicial Lie algebras over a graded ring}
We fix a graded-commutative ring $R$ with $2\in R^\times$   and work in the category $\sMod^\NN_{R}$ of weighted graded $R$-modules.
\begin{definition}[Weighted graded Lie algebra]\label{def:Lie algebra}
A (weighted, graded) \emph{Lie algebra}  in $\Mod^\NN_{R}$ consists of a weighted graded $R$-module $\mathfrak{g}\in \Mod^\NN_{R}$ together with a map $[-, - ] :  \mathfrak{g}\otimes_R \mathfrak{g}\to \mathfrak{g}$   \vspace{-2pt}
satisfying the following identities for all homogeneous elements $a\in \mathfrak{g}_i$, $b\in \mathfrak{g}_j$,  $c\in \mathfrak{g}_k$:
\begin{enumerate} 
\item $[a,b]+(-1)^{i j }[b,a]=0\vspace{2pt}  $
\item $(-1)^{i k}  [a,[b,c]]   +  (-1)^{j i }[b,[c,a]]   +    (-1)^{k j }  [c,[a,b]]      =\vspace{2pt}    0  $
\item $[a,[a,a]]=0\vspace{2pt}  $. 
\end{enumerate}
A map of Lie algebras is an $R$-module map intertwining the respective brackets.
We write $\Lie_{R}^\NN$ for the category of (weighted, graded) Lie algebras over $R$.  \vspace{-2pt}
\end{definition}
Due to our standing assumption that $2\in R^\times$, we have $[a,a]=0$ for $a$ in even internal degree.
 
\begin{remark}
There is disagreement in the literature over the definition of Lie algebras over general rings. An operadic definition would only enforce the first two axioms (and hence \mbox{$3\cdot [a,[a,a]]=0$} for any $a$), but the resulting ``operadic Lie algebras'' in general fail to inject into their universal enveloping algebras, since the third axiom is satisfied by any Lie algebra obtained from an associative algebra. 
Since the relationship with the universal enveloping algebra is crucial in what we do, we have chosen the above set of axioms.  \vspace{-2pt}
\end{remark}

The   forgetful functor $\U^{\Lie_R}_{\Mod_R}:\Lie_R^\NN \to \Mod_R^\NN$ admits a left adjoint $\Free^{\Lie_R}_{\Mod_R}$, the free Lie algebra functor. The category $\Lie_R^\NN$ is equivalent to  algebras for the corresponding \mbox{monad $\LAA$.} Carrying out these constructions degreewise, we obtain the category $\sLie_R^\NN$ of simplicial Lie algebras. It is linked to $\sMod_R^\NN$ by a (monadic) free-forgetful adjunction. Abusing   notation, we will write $\Free^{\Lie_R}_{\Mod_R}$ and $\U^{\Lie}$
for its constituent functors,  and  denote the resulting monad by
$\LAA$.

\begin{lemma}\label{lem:free lie}
If $M\in \sMod^\NN_{R}$ is levelwise free, projective,  or flat, then so is \vspace{-2pt} $\LAA(M)$.
\end{lemma}
\begin{proof}
The levelwise free case reduces via extension of scalars to the case $R=\mathbb{Z}[\frac{1}{2}]$ of \cite[Prop. 8.5.1]{Neisendorfer:AMUHT} (the standing assumption that $2\in R^\times$ is used to guarantee that our definition of Lie algebra coincides with \cite[Def. 8.1.1]{Neisendorfer:AMUHT}). If $M$ is levelwise projective, then $M$ is a summand of a levelwise  free simplicial module, so the same is true of $\LAA(M)$ by the  levelwise free case. If $M$ is levelwise flat, then $M$ is a filtered colimit of levelwise finite free modules by \Cref{lem:flat filtered colimit}, so the same is true of $\LAA(M)$ by the free case since $\U_{\Mod}^{\Lie}$ preserves filtered colimits.
\end{proof}

The category $\sLie_R^\NN$ of simplicial Lie algebras carries a standard model structure:

\begin{proposition}[Model structure on simplicial Lie algebras]\label{prop:lie model structure}
The right transferred model structure along the following adjunction exists: \begin{diagram}\sMod^\NN_{R} &    &\pile{\rTo^{\Free^{\Lie_R}_{\Mod_R}}  \\  \lTo_{\U^{\Lie_R}_{\Mod_R}}}&  &\sLie^\NN_{R}  \end{diagram} Its weak equivalences and fibrations are defined on underlying simplicial modules.
\end{proposition}
\begin{proof} 
Since every object in $\sMod_R^\NN$ is fibrant and the path object of $\sMod^\NN_{R}$ lifts to $\sLie^\NN_{R}$, the claim follows from well-known existence criteria---see e.g. \cite[Thm. 3.2, Rmk. 3.3]{johnson2014lifting}.
\end{proof}
 
\begin{notation}[Bar construction]
Given a monad $\TT$ acting on a category $\C$, a   right $\TT$-functor $F:\C\to \D$, and $X$ a simplicial $\TT$-algebra, we write $\Barr_\bullet(F,\TT,X)$ for the simplicial object of $\D$ given by taking the diagonal of the bisimplicial object obtained by applying the two-sided monadic bar construction on $F$ and $\TT$ levelwise to $X$.\vspace{-5pt}
\end{notation}
Using this, we can  construct an explicit cofibrant replacement functor :
\begin{lemma}\label{lem:bar resolution}
Let $\mathfrak{g}\in \sLie_R^\NN$ be a simplicial (weighted and graded) Lie algebra. The natural map $\Barr_\bullet(\Free_{\Mod}^{\Lie}, \LAA, \mathfrak{g})\to \mathfrak{g} $ is a weak equivalence. Moreover, if $\mathfrak{g} $ is levelwise projective, then $\Barr_\bullet(\Free_{\Mod}^{\Lie}, \LAA, \mathfrak{g})$ is a cofibrant object in $\sLie_R^\NN$.\
\end{lemma}
\begin{proof} 
That the map in question is a weak equivalence follows from \cite[Prop. 3.13]{johnson2014lifting}, and cofibrancy follows from \cite[Prop. 3.17, 3.22]{johnson2014lifting}.
\end{proof}
 
\subsection{The universal enveloping algebra}\label{univ}
If $A\in \Alg_R^\NN$ is an associative  algebra object in $\Mod_R^\NN$, then $A$ determines a  Lie algebra  in the sense of \Cref{def:Lie algebra} with bracket given by  
$[a,b]=ab-(-1)^{|a||b|}ba$. This resulting functor $\Alg_R^\NN \rightarrow \Lie_R^\NN$ admits a left adjoint, the \emph{universal enveloping algebra} functor, which may be constructed explicitly as the quotient \[U(\mathfrak{g})=\frac{T_R(\mathfrak{g})}{[a,b]=a\otimes b-(-1)^{|a||b|}b\otimes a}.\] Here $T_R$ denotes the tensor algebra over $R$. Note that, since the Lie algebra $0$ is terminal and $U(0)=R$, the functor $U$ factors canonically through the category of augmented algebras.

The algebra $U(\mathfrak{g})$ is naturally filtered by word length. After passing to the associated graded algebra, the defining relation becomes that of the free graded-commutative $R$-algebra $\Sym_R(\mathfrak{g})$. The following classical theorem summarises this observation:

\begin{theorem}[Poincar\'{e}--Birkhoff--Witt]\label{thm:PBW} 
If $\mathfrak{g}\in \Lie_R^\NN$ is a Lie algebra with flat underlying module, then the natural map $\Sym_R(\mathfrak{g})\to \mathrm{gr}\,U(\mathfrak{g})$ of augmented bigraded $R$-algebras is \mbox{an isomorphism.}
\end{theorem}

\begin{remark}
A reference for the flat but ungraded case is \cite{Higgins:BIBWT}, and the same argument applies in the graded context after adding appropriate Koszul signs. We do not spell out this straightforward adaptation, noting only that the assumption that $2\in R^\times$ seems to be necessary.
\end{remark}
 
\begin{proposition}[Invariance of the universal enveloping algebra]\label{prop:derived enveloping algebra}
The functor $U$ preserves weak equivalences between levelwise flat  {simplicial Lie algebras.}
\end{proposition}
\begin{proof}
The claim follows from the five lemma, Theorem \ref{thm:PBW}, and repeated use of Lemma \ref{cor:flatsym invariance}.
\end{proof}

\subsection{Lie algebra homology}\label{section:lie algebra homology}

A module $M\in \Mod_\RR^\NN$ determines a Lie algebra $\triv_{\Lie}^{\Mod}(M)$ in $ \Lie_R^\NN$ with underlying  module $M$ and vanishing Lie bracket. This construction gives a functor $\triv_{\Lie}^{\Mod}:\Mod_R^\NN\to \Lie_R^\NN$, which admits a left adjoint $\Q_{\Lie}^{\Mod}$. Explicitly, $\Q_{\Lie}^{\Mod}(\mathfrak{g})$ is obtained by forming the quotient of the underlying  module of $\mathfrak{g}$ by the submodule generated by all iterated brackets of elements of $\mathfrak{g}$. Applying these functors levelwise, we obtain an adjunction at the level of categories of simplicial objects. Since fibrations and weak equivalences of simplicial Lie algebras are defined on underlying modules, it follows that this is in fact a Quillen adjunction.

\begin{definition}[Lie algebra homology]
The  \emph{Lie algebra homology} of a  Lie algebra $\mathfrak{g} \in \sLie_R^\NN$ is defined as  \[\HLie(\mathfrak{g})=H\left(N\left(\LQ_{\Lie}^{\Mod}(\mathfrak{g})\right)[1]\oplus R\right).\]
\end{definition}
\begin{remark}
Since $\HLie(\mathfrak{g})$ is the homology of a chain complex of objects in $\Mod_R^\NN$, it is a graded object in $\Mod_R^\NN$, i.e.,  a bigraded weighted $R$-module.
\end{remark}

There is a small complex that is often available for computing this Lie algebra homology. To state the following   construction (essentially due to \cite{ChevalleyEilenberg:CTLGLA} and \cite{may1966cohomology}), we will need the divided power functors from \Cref{divpow}:
\begin{definition}[Chevalley--Eilenberg, May]\label{dfn:chevalley--eilenberg}
The \emph{Chevalley--Eilenberg complex} of a Lie algebra $\mathfrak{g}\in \Lie_R^\NN$ is the chain complex \[\CE(\mathfrak{g})=(\Gamma_R(\mathfrak{g}[1]),d),\] where $d$ is defined as follows; 
if $a_1,\ldots, a_m\in \mathfrak{g}^{\odd}$ have odd total degree and $b_1,\ldots, b_n \in \mathfrak{g}^{\even}$ have even total degree, then the value of $d$ on a generic element   $$ \gamma_{r_1}(\sigma a_1)\cdots \gamma_{r_m}(\sigma a_m)\   \langle \sigma b_1,\cdots,  \sigma b_n\rangle $$ of $\Gamma_R(\mathfrak{g}[1])$ is given by the following formula:

{\small\begin{align*}
&\ \ \   \sum_{1\leq i<j\leq m} \gamma_{r_1}(\sigma a_1)\cdots\gamma_{r_{i}-1}(\sigma a_{i})\cdots\gamma_{r_{j}-1}(\sigma a_{j}) \cdots\gamma_{r_m}(\sigma a_m)  \  \langle\sigma[a_{i},a_{j}], \sigma b_1,\cdots, \sigma b_n\rangle\\
&+\sum_{1\leq i<j\leq n}(-1)^{i+j-1} \gamma_{r_1}(\sigma a_1)\cdots  \gamma_{r_m}(\sigma a_m)  \   \langle  \sigma [b_i,b_j], \sigma b_1,\cdots\widehat{\sigma b_i},\cdots, \widehat{\sigma b_j},\cdots , \sigma b_n \rangle\\
& {+\ \frac{1}{2}}\ \sum_{i=1}^m  \gamma_{r_1}(\sigma a_1)\cdots\gamma_{r_i-2}(\sigma a_i)\cdots \gamma_{r_m}(\sigma a_m) \  \langle\sigma[a_i,a_i], \sigma b_1,\cdots , \sigma b_n\rangle\\
&+\sum_{i=1}^m\sum_{j=1}^n(-1)^{ {j-1}  } \gamma_1\left(\sigma[a_i,b_j]\right)\gamma_{r_1}(\sigma a_1)\cdots \gamma_{r_i-1}(\sigma a_i)\cdots \gamma_{r_m}(\sigma a_m) \  \langle\sigma b_1,\cdots, \widehat{\sigma b_j},\cdots, \sigma b_n\rangle
\end{align*}}
\end{definition}

If $\mathfrak{g}\in \sLie_R^\NN$ is a simplicial Lie algebra, we define $\CE(\mathfrak{g})$ by applying the previous construction in each simplicial degree, thereby obtaining a simplicial chain complex in $\Mod_R^\NN$. Its homotopy groups are  isomorphic to the homology of the corresponding total complex (obtained by using the Dold--Kan correspondence).

We have the following theorem about this homology, which in some form goes back at least as far as \cite{ChevalleyEilenberg:CTLGLA}---see also \cite{may1966cohomology,priddy1970koszul} for settings closer to ours: 

\begin{theorem} \label{thm:chevalley--eilenberg works}
If $\mathfrak{g} \in \sLie_R^\NN$ is levelwise flat, then there is a natural isomorphism $$H(\CE(\mathfrak{g}))\cong \HLie(\mathfrak{g}).$$
\end{theorem}

\begin{remark}
Away from characteristic zero, it is important to remember that $\CE(\mathfrak{g})$ is a simplicial chain complex with simplicial structure  induced by that of $\mathfrak{g}$. In particular, it is not in general  the result of an operation applied to any differential graded Lie algebra.
\end{remark}

 \begin{remark}
Our standing assumption is that $2$ is invertible, but \Cref{dfn:chevalley--eilenberg} and \Cref{thm:chevalley--eilenberg works} extend to $p=2$ 
if one adopts the small changes described in \cite[Section 5]{may1966cohomology}.
\end{remark}

Although this result is classical, we do not know of a statement in the literature covering exactly the required level of generality. In order to be self-contained, we offer a complete proof. 
The first step is to reinterpret Lie algebra homology as Tor groups  over the universal enveloping algebra from 
\Cref{univ}.  A proof of the following well-known result is contained in the following subsection:

\begin{proposition}\label{prop:lie homology is tor}
Given any Lie algebra $\mathfrak{g}\in \Lie_R^\NN$ whose underlying module is flat, there is a natural isomorphism $\HLie(\mathfrak{g})\cong \Tor^{U(\mathfrak{g})}(R,R).$
\end{proposition}

Assuming this result for now, the second step is to connect the Chevalley--Eilenberg complex from \Cref{dfn:chevalley--eilenberg} 
to a suitable $U(\mathfrak{g})$-resolution of the ground ring $R$.

\begin{definition}\label{def:extended chevalley--eilenberg}
The \emph{extended Chevalley--Eilenberg complex} of $\mathfrak{g}$ is the bigraded $R$-module \[\overline{\CE}(\mathfrak{g})= U(\mathfrak{g})\otimes_R\Gamma_R(\mathfrak{g}[1])\] equipped with the differential $d$ sending 
$$(-1)^{|a_0|}d\left(a_0\otimes \gamma_{r_1}(\sigma a_1)\cdots \gamma_{r_m}(\sigma a_m)\ \otimes \   \langle\sigma b_1,\cdots,  \sigma b_n\rangle\right)$$
to  the following expression: 
{\small\begin{align*}
&\ \ \ \ \ \   \sum_{i=1}^m a_0a_i\otimes\gamma_{r_1}(\sigma a_1)\cdots \gamma_{r_i-1}(\sigma a_i)\cdots \gamma_{r_m}(\sigma a_m) \  \langle\sigma b_1,\cdots, \sigma b_n\rangle\\
&+  \ \ \ \sum_{j=1}^n(-1)^{j-1}a_0b_j\otimes \gamma_{r_1}(\sigma a_1)\cdots \gamma_{r_m}(\sigma a_m) \  \langle\sigma  b_1,\cdots ,\widehat{\sigma b_j},\cdots, \sigma b_n\rangle\\
& +  \sum_{1\leq i<j\leq m}a_0\otimes \gamma_{r_1}(\sigma a_1)\cdots\gamma_{r_{i}-1}(\sigma a_{i})\cdots\gamma_{r_{j}-1}(\sigma a_{j}) \cdots\gamma_{r_m}(\sigma a_m) \  \langle\sigma[a_{i},a_{j}], \sigma b_1,\cdots, \sigma b_n\rangle\\
&+\sum_{1\leq i<j\leq n}(-1)^{i+j-1}a_0\otimes \gamma_{r_1}(\sigma a_1)\cdots \gamma_{r_m}(\sigma a_m) \  \langle \sigma [b_i,b_j], \sigma b_1,\cdots\widehat{\sigma b_i},\cdots, \widehat{\sigma b_j},\cdots , \sigma b_n\rangle\\
&+ \frac{1}{2} \sum_{i=1}^m a_0\otimes \gamma_{r_1}(\sigma a_1)\cdots\gamma_{r_i-2}(\sigma a_i)\cdots \gamma_{r_m}(\sigma a_m) \  \langle\sigma[a_i,a_i], \sigma b_1,\cdots , \sigma b_n\rangle\\
&+ \ \sum_{i=1}^m\sum_{j=1}^n(-1)^{j-1}a_0\otimes \gamma_1\left(\sigma[a_i,b_j]\right)\gamma_{r_1}(\sigma a_1)\cdots \gamma_{r_i-1}(\sigma a_i)\cdots \gamma_{r_m}(\sigma a_m) \  \langle\sigma b_1,\cdots, \widehat{\sigma b_j},\cdots, \sigma b_n\rangle
\end{align*}}
\end{definition}

\begin{lemma}\label{lem:chevalley--eilenberg projective}
If $\mathfrak{g}\in \Lie_R^\NN$ is $R$-flat, then the augmentation $\overline{\CE}(\mathfrak{g})\to R$ is a quasi-isomorphism. Therefore, if $\mathfrak{g}$ is $R$-projective, then $\overline{\CE}(\mathfrak{g})$ is a $U(\mathfrak{g})$-projective resolution of $R$.
\end{lemma}
\begin{proof}
It suffices to prove the claim after passing to the associated graded for the diagonal filtration induced by the natural filtrations of $U(\mathfrak{g})$ (cf. \Cref{univ}) and $\Gamma_R(\mathfrak{g}[1])$. Invoking the Poincar\'{e}--Birkhoff--Witt theorem,  this complex is isomorphic to the Koszul complex \[\left( \Sym_R(\mathfrak{g})\otimes_R\Gamma_R(\mathfrak{g}[1]),\, \partial \right)\to R\] for the graded $R$-module underlying $\mathfrak{g}$, which is acyclic for flat modules. Indeed, the flat case reduces to the case of a suspension of $R$ via filtered colimits and direct sums, and we may set $R=\mathbb{Z}$ without loss of generality, in which case the claim is classical (cf.\ \cite[Prop. 7.1]{McCleary:UGSS}).  
\end{proof}

\begin{proof}[Proof of Theorem \ref{thm:chevalley--eilenberg works}]
In light of the isomorphism $\CE(\mathfrak{g})\cong R\otimes_{\otimes_{U(\mathfrak{g})}}\overline{\CE}(\mathfrak{g})$, the $R$-projective case follows from Proposition \ref{prop:lie homology is tor} and Lemma \ref{lem:chevalley--eilenberg projective}. In the general case, it suffices to show that levelwise application of $\CE(-)$ preserves weak equivalences between $R$-flat simplicial Lie algebras, which follows by induction along the  filtration of $\CE(\mathfrak{g})$ from  \Cref{cor:flatgamma invariance}. 
\end{proof}

\subsection{Derivations and the proof of Proposition \ref{prop:lie homology is tor}}

We begin by observing an immediate consequence of Lemma \ref{lem:bar resolution}.

\begin{corollary}
There is a natural weak equivalence $\LQ_{\Lie}^{\Mod}(\mathfrak{g})\simeq \Barr_\bullet(\id, \LAA, \mathfrak{g})$ for $R$-projective Lie algebras $\mathfrak{g} $.
\end{corollary}

In order to compare this monadic bar construction to the derived tensor product in question, we will make use of some  classical ideas (cf. \cite{MR0257068,Barr:CECT}) relating algebraic homology theories to modules and derivations. We begin with several definitions.

\begin{definition}
Let $\mathfrak{g} $ be a Lie algebra. A $\mathfrak{g} $-\emph{module} is a module $N$ equipped with a  \mbox{linear map} \[\mathfrak{g}\otimes_RN\to N,\] written $a\otimes x\mapsto ax$, such that $[a,b]x = a(bx)-(-1)^{|a||b|}b(ax)$ for all    homogeneous  $a,b\in \mathfrak{g} $ and $x\in N$. A map of $\mathfrak{g} $-modules is an $R$-linear map intertwining the action maps.
\end{definition}

Write $\Mod_{\mathfrak{g}} $ for the category of $\mathfrak{g} $-modules, and note that this category is naturally isomorphic to the category of (left) $U(\mathfrak{g})$-modules.

\begin{construction}
Let $N$ be a $\mathfrak{g} $-module. Define a bracket on $\mathfrak{g}\ltimes N:= \mathfrak{g}\oplus N$ by the formula \[\left[(a,x),(b,y)\right]=\left([a,b], ay+(-1)^{|b||x|}by\right)\] on   homogeneous  elements. One checks from the definition of a $\mathfrak{g} $-module that this bracket satisfies the axioms of a graded Lie algebra in the sense of \Cref{def:Lie algebra}.
\end{construction}

We refer to $\mathfrak{g}\ltimes N$ as the \emph{split square-zero extension} of $\mathfrak{g}$ by $N$. This construction extends in the obvious way to a limit-preserving functor $\Mod_{\mathfrak{g}}\to \Lie_{R/\mathfrak{g}}$ with left adjoint $\Omega_{\mathfrak{g}}$. This object has a further functoriality for base change in that, given maps of Lie algebras $\mathfrak{g}'\to \mathfrak{g}_1\to \mathfrak{g}_2$, there is a canonical map $\Omega_{\mathfrak{g}_1}(\mathfrak{g}')\to \Omega_{\mathfrak{g}_2}(\mathfrak{g}')$ of $\mathfrak{g}_1$-modules, where the target is viewed as a $\mathfrak{g}_1$-module by restriction along the second map. This map arises as the adjoint of the dashed filler in the following commuting diagram of Lie algebras:
 \[\xymatrix{
& \mathfrak{g}_1\ltimes\Omega_{\mathfrak{g}_2}(\mathfrak{g}')\ar[d]\ar[r]&\mathfrak{g}_2\ltimes\Omega_{\mathfrak{g}_2}(\mathfrak{g}')\ar[d]\\
\mathfrak{g}'\ar@{-->}[ru]\ar[r] & \mathfrak{g}_1\ar[r]& \mathfrak{g}_2.
}\]

\begin{lemma}\label{lem:free differentials}
Let $\mathfrak{g}'\to \mathfrak{g} $ be a map of Lie algebras, and regard $\Free_{\Mod}^{\Lie}\circ \LAA^{\circ n}(\mathfrak{g}')$ as an object of $\Lie_{R/\mathfrak{g}}$ via the structure map to $\mathfrak{g}'$. There is an isomorphism of $\mathfrak{g}$-modules (natural in $\mathfrak{g}'\to \mathfrak{g} $):\[\Omega_{\mathfrak{g}}(\Free_{\Mod}^{\Lie}\circ \LAA^{\circ n}(\mathfrak{g}'))\cong U(\mathfrak{g})\otimes_R \LAA^{\circ n}(\mathfrak{g}')\] .
\end{lemma}
\begin{proof}
By adjunction, for any $\mathfrak{g} $-module $N$, we have \begin{align*}
\Hom_{\Mod_{\mathfrak{g}}}(\Omega_{\mathfrak{g}}(\Free_{\Mod}^{\Lie}\circ \LAA^{\circ n}(\mathfrak{g}')),N)&\cong \Hom_{\Lie_{R/\mathfrak{g}}}(\Free_{\Mod}^{\Lie}\circ \LAA^{\circ n}(\mathfrak{g}'), \mathfrak{g}\ltimes N)\\
&\cong \Hom_{\Mod_{R/U^{\Lie}_{\Mod}(\mathfrak{g})}}(\LAA^{\circ n}(\mathfrak{g}'), \mathfrak{g}\oplus N)\\
&\cong \Hom_{\Mod_R}(\LAA^{\circ n}(\mathfrak{g}'),  N)\\
&\cong \Hom_{\Mod_{\mathfrak{g}}}(U(\mathfrak{g})\otimes_R\LAA^{\circ n}(\mathfrak{g}'),  N).
\end{align*}
\end{proof}

Split square-zero extensions are closely related to the theory of derivations.

\begin{definition}  A \emph{derivation} of $\mathfrak{g} $ into a $\mathfrak{g} $-module $N$ is an $R$-module map $f: \mathfrak{g}\to N$ \mbox{such that} \[f([a,b])=af(b)+(-1)^{|a||b|}bf(a)\] for all  homogeneous $a,b\in \mathfrak{g} $. We write $\mathrm{Der}(\mathfrak{g},N)$ for set of derivations of $\mathfrak{g} $ into $N$. 
\end{definition}

\begin{lemma}\label{lem:derivations}
There is a bijection $\Hom_{\Lie_{R/\mathfrak{g}}}(\mathfrak{g}, \mathfrak{g}\ltimes N)\cong \mathrm{Der}(\mathfrak{g},N)$ naturally in $N$ and $\mathfrak{g} $.
\end{lemma}
\begin{proof}
Both sets inject into the set of $R$-module maps $f: \mathfrak{g}\to N$, so it suffices to show that the condition of being a derivation is the same as the condition that $(\id_{\mathfrak{g}},f)$ be a map of Lie algebras. This comparison is implied by  the following simple  computation  in $\mathfrak{g}\ltimes N$:  \[[a+f(a), b+f(b)]=[a,b]+af(b)+(-1)^{|a||b|}bf(a)\]
\end{proof}
Using this universal property, we now connect $\Omega_{\mathfrak{g} }$ to the universal enveloping algebra. As a matter of notation, we write $I(A)$ for the augmentation ideal of an augmented algebra.

\begin{lemma}\label{lem:augmentation and derivations}
There is a natural  isomorphism of $\mathfrak{g}$-modules $\Omega_{\mathfrak{g}}(\mathfrak{g})\cong I(U(\mathfrak{g}))$.
\end{lemma}
\begin{proof}
By Lemma \ref{lem:derivations}, it suffices to show that $I(U(\mathfrak{g}))$ corepresents the functor $\mathrm{Der}(\mathfrak{g},-)$. 

First, if $f: \mathfrak{g}\to N$ is a derivation, then we obtain a map $\overline f:I(T_R(\mathfrak{g}))\to N$ by setting \[\overline f(a_1\otimes\cdots \otimes a_n)=a_1\cdots a_{n-1}f(a_n).
\] Since $\mathfrak{g} $ is a $\mathfrak{g} $-module,  we observe that for all homogeneous $t_1,t_2\in T_R(\mathfrak{g})$ with $t_2\neq1$, we have:
 \begin{align*}
\overline f\left(t_1\left(a\otimes b-(-1)^{|a||b|}b\otimes a-[a,b]\right) t_2\right)=0.
\end{align*}  
In the other case, we can use that $f$ is a derivation to deduce 
 \begin{align*}
\overline f\left(t_1\left(a\otimes b-(-1)^{|a||b|}b\otimes a-[a,b]\right)\right)&=0.
\end{align*} Thus, $\overline f$ descends to the quotient $I(U(\mathfrak{g}))$; the resulting map is a map of $\mathfrak{g} $-modules by \mbox{construction.}

Conversely, given a map of $\mathfrak{g} $-modules $f:I(U(\mathfrak{g}))\to N$,   the composite $\mathfrak{g}\to U(\mathfrak{g})\to N$ is a derivation as \begin{align*}
0&=f(ab)-(-1)^{|a||b|}f(ba)-f([a,b])=af(b)-(-1)^{|a||b|}bf(a)-f([a,b]).
\end{align*} We have defined inverse bijections $\Hom_{\Mod_{\mathfrak{g}}}(I(U(\mathfrak{g})), N)\cong \mathrm{Der}(\mathfrak{g},N)$; naturality is obvious.
\end{proof}  
\begin{corollary}\label{cor:flat differentials}
Levelwise application of $\Omega_{\mathfrak{g} }$ preserves weak equivalences of levelwise flat simplicial Lie algebras over $\mathfrak{g} $.
\end{corollary}

In particular, $\Omega_{\mathfrak{g} }$ preserves weak equivalences between cofibrant objects and so admits a total left derived functor.

\begin{lemma}
If $\mathfrak{g} $ is an $R$-projective Lie algebra, then the augmentation \[N(\Omega_{\mathfrak{g}}(\Barr_\bullet(\Free_{\Mod}^{\Lie},\LAA, \mathfrak{g})))\to \Omega_{\mathfrak{g}}(\mathfrak{g})\cong I(U(\mathfrak{g}))\] is a cofibrant replacement in the category of chain complexes of $U(\mathfrak{g})$-modules. 
\end{lemma}
\begin{proof}
In each chain degree, the chain complex in question is $U(\mathfrak{g})$-free on a projective $R$-module by Lemmas \ref{lem:free lie} and \ref{lem:free differentials}. To see that the map is a quasi-isomorphism, we note that the lefthand side computes $\mathbb{L}\Omega_{\mathfrak{g}}(\mathfrak{g})$ by Lemma \ref{lem:bar resolution}, while the righthand side computes $\mathbb{L}\Omega_{\mathfrak{g}}(\mathfrak{g})$ by the projectivity of $\mathfrak{g} $, Lemma \ref{lem:augmentation and derivations}, and Corollary \ref{cor:flat differentials}.
\end{proof}

\begin{corollary}
For $R$-projective simplicial Lie algebras $\mathfrak{g} $, there is a natural weak equivalence \[\LQ_{\Lie}^{\Mod}(\mathfrak{g})\simeq R\otimes^\mathbb{L}_{U(\mathfrak{g})} I(U(\mathfrak{g})).\]
\end{corollary}

\begin{proof}[Proof of Proposition \ref{prop:lie homology is tor}]
Assume first that $\mathfrak{g} $ is $R$-projective. The short exact sequence $$I(U(\mathfrak{g}))\to U(\mathfrak{g})\to R$$ of $\mathfrak{g} $-modules gives rise to a cofiber sequence \[\LQ_{\Lie}^{\Mod}(\mathfrak{g})\to R \to R\otimes^\mathbb{L}_{U(\mathfrak{g})} R.\] The augmentation $U(\mathfrak{g})\to R$ gives rise to a retraction of the righthand map, and the claim follows. In the general case, it suffices to verify that $R\otimes^\mathbb{L}_{U(\mathfrak{g})}R$ preserves levelwise weak equivalences between levelwise flat simplicial Lie algebras, which follows from the fact that $U(\mathfrak{g})$ is flat for $R$-flat $\mathfrak{g} $ by Theorem \ref{thm:PBW} and the functor $U$ preserves weak equivalences between $R$-flat simplicial Lie algebras by Proposition \ref{prop:derived enveloping algebra}.
\end{proof}

  \label{sec:Hecke}\section{A Chevalley--Eilenberg complex for Hecke Lie algebras} \label{sec:Hecke} 
In this section, we  recall the definition of Hecke Lie algebras  and 
develop their homotopy theory. Hecke Lie algebras were introduced in the first author's thesis    \cite{brantnerthesis} in order to describe the operations acting on the $E$-theory of $K(h)$-local Lie algebras.
In this work, we restrict attention to the case of an odd prime $p$,   where  Hecke Lie algebras \mbox{are particularly simple.}

We fix a (smooth, $1$-dimensional, commutative) formal group $G_0$ of height $0<h<\infty$ over a perfect field $k$ of characteristic $p$. We write $E$ for the corresponding Lubin--Tate spectrum  constructed by Goerss, Hopkins, and Miller \cite{MR2125040} \cite{MR1642902}. The ring spectrum $E$ is complex orientable, and we fix a complex orientation $x^E \in \widetilde{E}^0(\CC P^\infty)$. Given a (virtual) complex vector bundle $V\to X$, there is a Thom class $\tau_V \in \widetilde{E}^0(X^V)$ such that multiplication by $\tau_V$ defines an isomorphism $E^\ast(X) = \widetilde{E}^\ast(X_+) \xrightarrow{\simeq} \widetilde{E}^\ast(X^V)$.
In the case of the trivial complex line bundle over a point, the Thom isomorphism gives rise to an isomorphism $E^\ast \cong E^{\ast-2}$ and hence to a periodicity generator $u \in \pi_2(E) = { {E}_2}$.\vspace{-2pt}

\subsection{Power operations on $\EE_\infty$-rings} We begin with some recollections concerning power operations on $K(h)$-local $\EE_\infty$-$E$-algebras, as developed in  \cite{wilkerson1982lambda}, \cite{hopkins1998k}, and \cite{rezk2009congruence}).

Given an integer $i\in \ZZ$, we write $\Gamma^{-i}$ for Rezk's (uncompleted) ring of  additive degree $i$ power operations, which acts naturally on the $i^{th}$ homotopy group $\pi_i(R)$ for any $K(h)$-local $
\EE_\infty$-$E$-algebra $R$ (cf. \cite[Section 6.2.]{rezk2009congruence}).
 The ring $\Gamma^{-i}$ is endowed with a canonical weight grading $\Gamma^{-i} \cong \bigoplus_w  \Gamma^{-i}(w)$, where
\begin{equation} \label{e1} \Gamma^{-i}(w) \cong \ker \left( \pi_{i} \left((\Sigma^{-i}E)^{\otimes w}_{h\Sigma_{w}}\right)  \rightarrow \bigoplus_{0<j<w} \pi_{i} \left((\Sigma^{-i}E)^{\otimes w}_{h(\Sigma_j \times \Sigma_{w-j})}\right)\right).\end{equation} 
The map shown is induced by the transfer. We have\footnote{We differ from Rezk's ``logarithmic'' grading convention.}  $\Gamma^{-i}(w)= 0$ unless $w$ is a power of $p$.
\begin{warning}
The rings $\Gamma^i$ are not related to the free divided power functor $\Gamma^\ast$ appearing in the Chevalley--Eilenberg complex; the meaning of the symbol $\Gamma$ will be clear from the context.\vspace{-3pt}
\end{warning}

We can also consider \emph{weighted} $K(h)$-local $\EE_\infty$-$E$-algebras, i.e., commutative algebra objects in the functor $\infty$-category $\Fun(\mathbb{Z}_{\geq 0}, \Mod_{E}^{\wedge})$ 
with its Day convolution symmetric monoidal structure (cf. \cite{glasman2016day},\cite[Section 2.2.6]{Lurie:HA}). In this context, Rezk's rings act in a weighted manner, in the sense that there are action maps \[\Gamma^{-i}(w) \times \pi_i(R(n)) \longrightarrow \pi_i(R(w\cdot n))\] 
refining the usual action on the underlying  $K(h)$-local $\EE_\infty$-$E$-algebra $L_{K(h)}\left( \bigoplus_n R(n)\right)$.
 
The rings $\Gamma^{-i}$ are linked by \emph{twisting} homomorphisms
$E_{ k} \myotimes{E_0} \Gamma^{-i} \myotimes{E_0} E_{-k} \to  \Gamma^{-i-k}$ sending a tensor $\lambda \otimes \alpha \otimes \mu$ to the operation $x \mapsto ( \lambda \cdot \alpha(\mu\cdot x))$. This twisting map respects weight and is an isomorphism whenever $k$ is even. 

More interesting are the \emph{suspension} homomorphisms \[\ldots\xrightarrow{\cong} \Gamma^2 \hookrightarrow \Gamma^1 \xrightarrow{\cong} \Gamma^0 \hookrightarrow \Gamma^{-1} \xrightarrow{\cong} \Gamma^{-2} \hookrightarrow  \ldots.\]

These morphisms, which we denote generically by $\Susp$, also respect weight. They are defined as follows. First, note   that for all  $w$ and $i$, there is an identification  of $E_0$-modules 
\[\pi_i\left( (\Sigma^i E)^{\otimes w}_{h\Sigma_w}\right) \cong  \widetilde{E}^{\wedge}_{ i} \left((S^{ i})^{\otimes w}_{h\Sigma_w}\right).\]
Smashing the diagonal $S^1 \rightarrow (S^1)^{\otimes w}$ with $(S^i)^{\otimes w}$ and applying $(-)_{\Sigma_w}$ produces a \vspace{-1pt}map \[\widetilde E^{\wedge}_i \left( (S^{ i})^{\otimes w}_{h\Sigma_w} \right)\cong \widetilde E^{\wedge}_{i+1} \left(  S^1 \otimes_{h\Sigma_w} (S^{ i})^{\otimes w} \right) \rightarrow \widetilde E^{\wedge}_{i+1} \left( (S^{ i+1})^{\otimes w}_{h\Sigma_w} \right),\vspace{-1pt} \]
which restricts to the desired map $\Gamma^{-i}(w) \rightarrow \Gamma^{-i-1}(w)$ \vspace{-1pt}(cf. e.g. \cite[Remark 7.5]{rezk2009congruence}).

The following result provides a useful alternative description of \vspace{-1pt}the suspension.

\begin{proposition}[Suspensions via   Euler class]\label{suspensiondiagram} For each  $w\geq 0$, there is a \vspace{-2pt}commutative diagram 
\begin{diagram}
\cdots & \rTo &  \Gamma^0(w) &&   \rInto ^{(e\cdot -)^\vee} &&  \Gamma^0(w)& & \rTo^{\cong} &  &\Gamma^0(w)& &   \rInto^{(e\cdot -)^\vee}  & & \Gamma^0(w)& & \rTo  & \cdots \\
 & &  \dTo^\cong  & & & &  \dTo^\cong    && &  &  \dTo^\cong  & && &  \dTo^\cong    \\
\cdots& \rTo  & \Gamma^2(w) & &   \rInto^\Susp  & & \Gamma^1(w)&&  \rTo^\Susp & &\Gamma^0(w)&&  \rInto^\Susp & & \Gamma^{-1}(w)&&  \rTo  &  \cdots, \\
\end{diagram}
where $e \in  \widetilde{E}^0(B\Sigma_{w+})$ is the Euler class of the reduced complex\vspace{-2pt} \vspace{-1pt}standard \mbox{representation of $\Sigma_w$.}
\end{proposition}
\begin{proof}When $i=2j$, we can identify $ (S^{2j})^{\otimes w}_{h\Sigma_w} $ with the Thom spectrum
 $B\Sigma_{w }^{j\cdot V_{w}}$, where $V_{w}$ denotes the (complexified) standard \mbox{representation of $\Sigma_w$.} 
Under this identification, the map $\Sigma^{2 } (S^{2j})^{\otimes w}_{h\Sigma_w} \rightarrow  (S^{2(j+1)})^{\otimes w}_{h\Sigma_w}$ induced by the diagonal, as above, corresponds to the ``Thomification'' of the map of $\Sigma_m$-representations $ j \cdot V_m\oplus \CC  \rightarrow (j+1) \cdot V_{m } $ induced by the diagonal embedding of the trivial representation $\CC$ into the standard representation $V_m$.\vspace{0pt}
 
The $E_0$-linear dual of the Thom isomorphism provides the\vspace{-1pt} identification $\widetilde{E}^{\wedge}_{2j}(B\Sigma_{w}^{j\cdot V_{w}})\cong \widetilde{E}^{\wedge}_0(B\Sigma_{w+})$, and, since the Thom isomorphism respects transfers, this identification restricts to an isomorphism of left $E_0$-modules $\Gamma^{-2j}(w) \cong \Gamma^{0}(w)$ defined in \eqref{e1}. We have used repeatedly that the completed $E$-homology of symmetric groups and the $E_0$-modules $\Gamma^i(w)$ are finitely generated and free \cite[Theorem 1.1.]{strickland1998morava}. The composite\vspace{-1pt} \[\widetilde{E}^0(B\Sigma_{w+}) \cong \widetilde{E}^{2(j+1)}(B\Sigma_{w+}^{(j+1)\cdot V_m}) \rightarrow \widetilde{E}^{2(j+1)}(B\Sigma_{w+}^{ j\cdot V_m\oplus \CC }) \cong \widetilde{E}^0(B\Sigma_{w+})\vspace{-2pt}\] induced by the diagonal $\CC \rightarrow V_m$ is well known to be multiplication by $e \in  \widetilde{E}^0(B\Sigma_{w+})$ -- see \cite[p.14]{strickland1998morava}, for example. 
Passing to \mbox{$E_0$-linear} duals and to the respective submodules of additive operations, and then inserting the odd degree rings ``by hand'' (using that every second suspension is an isomorphism), we obtain the result.\vspace{-4pt}
\end{proof}
At weight $w=p$, the module $\Gamma^0(p)^\vee$ is   given by $E^0(B\Sigma_{p})/(\tr)$, there $\tr $ denotes the image of the transfer map from the trivial group (cf. \cite[Section 10.3]{rezk2009congruence}).
The Euler class of the reduced standard representation of $B\Sigma_p$ may be viewed as an element $e \in E^0(B\Sigma_p)$, and by abuse of notation, we shall also use $e$ to denote its image in $\Gamma^0(p)^\vee$.\vspace{-1pt}

\begin{example}[The $K$-theory of $\EE_\infty$-rings]
The $E$-theory corresponding to the commutative formal group law $\widehat{\GG}_m$ of \mbox{height $1$} over $\FF_p$ is given by $p$-complete complex $K$-theory. 
We recall from \cite{hopkins1998k} that
 there is a $K(1)$-local equivalence  of spectra $S^0_{h\Sigma_p} \simeq B\Sigma_{p+} \xrightarrow{ \ (\epsilon, \tr) \ } S^0 \oplus S^0$, where $\epsilon$ denotes the collapse map and $\tr$ is the transfer map.
The homotopy unique map $\psi:S^0 \rightarrow B\Sigma_{p+}$
with $\epsilon \circ \psi \simeq 1$ and $\tr \circ \psi \simeq 0$ defines a class $\Psi_0\in K^{\wedge}_0(B\Sigma_{p+}) = \pi_0(L_{K(1)}K^{\otimes p}_{h\Sigma_p})$. As the transfer vanishes, this class lies in $\Gamma^0(p)$. Using periodicity and the suspensions $\Gamma^{2n+1}(p) \xrightarrow{\cong}\Gamma^{2n}(p)$, we obtain classes $\Psi_i \in \Gamma^i$ for all $i$.
At height $1$, the Euler class appearing in \Cref{suspensiondiagram} is just $p$, and the  suspension homomorphism $\Gamma^i \rightarrow \Gamma^{i-1}$ is therefore determined by the assignment 
$$\Psi_i \ \ \  \mapsto \ \ \  \begin{cases} \ \ \ \ \ \ \Psi_{i-1} & \mbox{for }   i   \mbox{ odd} \\  \ \ \ 
p \cdot \Psi_{i-1} &\mbox{for }   i  \mbox{ even.}\end{cases}$$ 
The class $\Psi_i$ generates $\Gamma^i$ freely and there is an equivalence $\Gamma^i= \ZZ_p[\Psi_i]$.
\end{example}

In fact, we can  explicitly describe the ring $\Gamma^0(p)^\vee$ and the Euler class $e$ at all heights:
\begin{proposition}[Euler class basis]\label{explicit!}
If the coordinate on  $E$-theory is chosen so that the resulting formal group law is $p$-typical, then there are isomorphisms
$$\Gamma^0(p)^\vee \cong E^0(B\Sigma_p)/(\tr) \cong E_0[e]/f(e),$$ 
where 
$e $ is the Euler class of the reduced standard representation and
$f(e)=e^{\frac{p^h-1}{p-1}}+\cdots+p$
is the unique monic degree $\frac{p^h-1}{p-1}$ polynomial over $E_0$ for which $f(-x^{p-1}) = \frac{[p](x)}{x}$ in $E_0[[x]]/[p](x)$.
\end{proposition}
\begin{proof}
A standard transfer argument shows that $E^*(B\Sigma_p)$ is the subring of $(\mathbb{Z}/p)^{\times} = \text{Aut}(C_p)$-fixed points inside $E^*(BC_p)$.  The Gysin sequence associated to the fibration $S^1 \longrightarrow BC_p \longrightarrow BS^1 \simeq \mathbb{CP}^{\infty}$ yields the formula $E^*(BC_p) \cong E_*[[x]]/[p](x)$.  The Euler class $e \in E^*(B\Sigma_p)$ sits inside of $E^*(BC_p)$ as $\prod_{k=1}^{p-1} [k](x) = -x^{p-1}$, where the last equality is a consequence of the $p$-typicality of the formal group law.  Finally, the transfer ideal $(\text{tr})$ is generated as an $E_*$-module by $\frac{[p](x)}{x}$.  For details, we refer to \cite{hopkins2000generalized}  and  \cite[Section 4.3.6]{marsh2010morava}. 
\end{proof} 
\subsection{Hecke modules and Hecke Lie algebras}\label{four}
We now review the theory of Hecke modules and Hecke Lie algebras \cite[Section 4.3-4]{brantnerthesis}. As we will see in Section \ref{section:spectral lie}, these definitions exactly capture the structure of the operations acting on the completed $E$-homology of $K(h)$-local Lie algebras.

\begin{remark}[Grading convention]
Our conventions differ from those of \cite{brantnerthesis} by a shift. Specifically, if $\mathfrak{g}$ is a shifted Lie algebra with shifted bracket  $[-,-]'$, then $\Sigma^{-1} \mathfrak{g}$ becomes a graded Lie algebra with bracket defined by the formula $[u,v] := (-1)^{\deg(u)}\sigma^{-1} \left( [\sigma(u),\sigma(v)]' \right)$, where $\sigma:   \Sigma^{-1} \mathfrak{g} \xrightarrow{ } \mathfrak{g}$ denotes the evident shifting bijection and $\sigma^{-1}$ its inverse.

Our definitions of the Hecke power ring,  Hecke modules,   and Hecke Lie algebras will therefore all differ from the corresponding notions in \cite{brantnerthesis} by a shift. To make this small difference clear throughout, we have added the subscript ``u" in various places, thereby stressing that we are working \textit{u}nshifted Hecke Lie algebras.
 \end{remark}
 
\begin{definition}[Power rings] A \textit{power ring} is a collection $P=\{P_i^j(w)\}_{(i,j,w)\in \ZZ^2 \times \NN}$ of abelian groups with elements $\iota_{ii}\in P_i^i(1)$ for all $i$, together with associative and unital composition maps $$P_i^j(v)  \otimes P_j^k(w) \longrightarrow P_i^k(vw).$$ 
\end{definition}

\begin{example}
The power ring $P^{\Comm}$ of additive operations on $K(h)$-local $\EE_\infty$-rings \mbox{is given by}
\[(P^{\Comm})_{i}^{j}(w) =  E_{j-i} \myotimes{E_0} \Gamma^{-i}(w),\] with composition defined using the twisting maps and multiplication in the $\Gamma^{i}$.
\end{example}
\begin{definition}[Modules over power rings]
A (weighted) \textit{module} over the power ring $P$ is a weighted graded abelian group $M\in \Mod_\ZZ^\NN$ equipped with multiplication maps $P_{i}^{j}(w)\otimes M_i \rightarrow M_j$ compatible with composition in $P$. 
A map of $P$-modules is a map of weighted graded abelian groups intertwining with multiplication.
\end{definition}

We can now define the power ring of primary interest:

\begin{definition}[Hecke operations on Lie algebras] \label{Heckepowerring}
The power ring $ {\mathcal{H}_u^{\Lie}}$ of additive operations on (unshifted) $K(h)$-local spectral Lie algebras is given by 
\[({\mathcal{H}_u^{\Lie}})_{i}^{j}(w) = \begin{cases}  \Ext^a_{\Gamma^{i}}({E}_0,{E}_{-i+j+a} ) &\mbox{if } w  = p^a \\ 
0 & \mbox{if } w \mbox{ is not a power of $p$,} \end{cases}\]
where we regard ${E}_0$ and ${E}_{-i+j+a}$ as trivial $\Gamma^{i}$-modules, with composition defined as the counterclockwise composite in the commutative diagram
\[\xymatrix{
({\mathcal{H}_u^{\Lie}})_{i}^{j}(p^a)  \otimes ({\mathcal{H}_u^{\Lie}})_{j}^{k}(p^b)\ar@{=}[d]\ar@{-->}[r]&({\mathcal{H}_u^{\Lie}})_{i}^{k}(p^{a+b})\ar@{=}[ddd]\vspace{-3pt}\\
(\Ext^a_{\Gamma^{i}}({E}_0,{E}_{-i+j+a}))\otimes (\Ext^b_{\Gamma^{j}}({E}_0,{E}_{-j+k+b}))\ar[d]\vspace{-3pt}\\
( \Ext^a_{\Gamma^{i}}({E}_0,{E}_{-i+j+a}) ) \otimes(\Ext^b_{\Gamma^{j+a}}({E}_0,{E}_{-j+k+b}) )\ar[d]\vspace{-3pt}\\
( \Ext^a_{\Gamma^{i}}({E}_0,{E}_{-i+j+a}))  \otimes (\Ext^b_{\Gamma^{i}}({E}_{-i+j+a},{E}_{-i+k+a+b}) )\ar[r]
&(\Ext^{a+b}_{\Gamma^{i}}({E}_0,{E}_{-i+k+a+b})),\vspace{-3pt}
}\] The first map is  suspension, the second twisting, and the third is the Yoneda product. 
\end{definition} 

\begin{remark}
The modules $\left(\mathcal{H}^{\Lie}\right)_i^j(w)$ are free; their ranks are given by the generating series $$ \sum_{a=0}^\infty  
\left(\mathcal{H}^{\Lie}\right)_i^j(p^a) \cdot T^a  = (1+T) \cdot (1+pT) \cdot \ldots \cdot (1+p^{h-1} T);$$ cf.\ \cite[Prop. 4.6, Prop. 8.8]{rezk2017rings}  \cite[Sec. 1.4.1]{brantnerthesis}. In particular, $ (\mathcal{H}^{\Lie})_i^j (w) = 0$ for $w>p^h$.
\end{remark}

\begin{example}[Hecke operations  at height one]\label{Lieheightone} For $p$-adic $K$-theory defined over $\ZZ_p$, we have 
$$({\mathcal{H}_u^{\Lie}})_{i}^{j}(w) = \begin{cases} E_{j-i} \cdot \iota_i &\mbox{if } w  = 1 \\ 
E_{j-i+1} \cdot \alpha_i &\mbox{if } w  = p\\
0 & \mbox{else,}  \end{cases} \vspace{-1pt}$$
Here $\iota_i$ is the identity operation in degree $i$, and the weight of $\alpha_i$ is $p$. Composition is defined as
\begin{align*}(\lambda_{j-i}\cdot  \iota_i) \otimes  (\lambda_{k-j} \cdot \iota_j) &\mapsto  (\lambda_{k-j}\lambda_{j-i}  \cdot  \iota_i)\vspace{-1pt}\\
(\lambda_{j-i}\cdot  \alpha_i) \otimes  (\lambda_{k-j+1} \cdot \iota_{j-1}) & \mapsto (\lambda_{k-j+1}\lambda_{j-i}  \cdot  \alpha_i)\vspace{-1pt} \\
(\lambda_{j-i}\cdot  \iota_i) \otimes  (\lambda_{k-j+1} \cdot \alpha_{j}) & \mapsto  (\lambda_{k-j+1}\lambda_{j-i}  \cdot  \alpha_i)\vspace{-1pt}\\
(\lambda_{j-i}\cdot  \alpha_i) \otimes  (\lambda_{k-j+1} \cdot \alpha_{j-1}) &  \mapsto  0.
\end{align*}
Here we have used that the Frobenius on $\ZZ_p = W(\FF_p)$ is trivial at height $h=1$.
\end{example}

\begin{definition}[Hecke modules]
A  \textit{Hecke module} is a weighted module over the power ring ${\mathcal{H}_u^{\Lie}}$. We write $\Mod_{{\mathcal{H}_u}}^\NN $ for the category of Hecke modules. 
\end{definition}
 Note that a Hecke module is in particular a weighted $E_*$-module.  We write $\Free_{\Mod_{E_\ast}}^{\Mod_{{\mathcal{H}_u}}}$ and $\U_{\Mod_{E_\ast}}^{\Mod_{{\mathcal{H}_u}}}$
 the corresponding free and forgetful functors.
The category $\Mod_{{\mathcal{H}_u}}^\NN $ can also be interpreted as 
 the category of algebras for the following additive monad $\AAA$ on $\Mod_{E_\ast}^\NN$:
\begin{definition}[The monad of additive operations]\label{AdditiveMonadAAA}
Let us  consider the endofunctor 
 $\AAA$ on $ \Mod_{E_\ast}^\NN $  sending $M \in \Mod_{E_\ast}^\NN$ to the weighted $E_\ast$-module  $\AAA(M) \in \Mod_{E_\ast}^\NN$
whose $j^{th}$  degree  $\AAA(M)_j$ is the quotient of the free abelian monoid on symbols  
\[\{ \  [\alpha|x] \ | \ \alpha \in ({\mathcal{H}_u^{\Lie}})_{i}^j , x_i \in M_{i} \}\] by all relations 
$[\alpha_1,x] + [\alpha_2,x] = [\alpha_1+\alpha_2 , x]  \ ,  \ [\alpha , x+y] = [\alpha ,x] + [\alpha , y] \ ,  \ [\lambda |x] =  [1|\lambda x]$
for $x_i \in M_{i}$,   $\alpha \in ({\mathcal{H}_u^{\Lie}})_{i}^j$,  
 $\lambda \in ({\mathcal{H}_u^{\Lie}})_{i}^j(1)$.  {We implicitly endow symbols  with their natural weight.}

The composition map $({\mathcal{H}_u^{\Lie}})_{i}^{j}  \otimes ({\mathcal{H}_u^{\Lie}})_{j}^{k} \rightarrow ({\mathcal{H}_u^{\Lie}})_{i}^{k}$ from \Cref{Heckepowerring}  equips $\AAA$ with the structure of a monad whose category of algebras is given by  $\Mod_{{\mathcal{H}_u}}^\NN $ .
\end{definition}

\begin{definition}[Hecke Lie algebras] \label{HLA}
A \textit{Hecke Lie algebra} consists of an (unshifted, weighted) Hecke module $\mathfrak{g}\in \Mod_{{\mathcal{H}_u}}^\NN$ equipped with the structure of a Lie algebra on its underlying weighted $E_\ast$-module, subject to the relation $[x,\alpha(y)] = 0 $ for all $x\in \mathfrak{g}_k, y \in \mathfrak{g}_i$, and $\alpha\in ({\mathcal{H}_u^{\Lie}})_{i}^{j}(w)$ with $w>1$. A map of Hecke Lie algebras is a map of Hecke modules intertwining the brackets. 
\end{definition}

We write  $\Lie_{{\mathcal{H}_u}}^\NN$ for the resulting  category of Hecke Lie algebras. This category is the category of modules for a  monad $\LL$ on $\Mod_{E_\ast}^\NN$, and we denote the corresponding free and forgetful functors by $\Free_{\Mod_{E_\ast}}^{\Lie_{{\mathcal{H}_u}}}$ and $\U_{\Mod_{E_\ast}}^{\Lie_{{\mathcal{H}_u}}}$, respectively. These adjunctions factor through free and forgetful adjunctions to $\Lie_{E_*}^\NN$ and $\Mod_{{\mathcal{H}_u}}^\NN$, which we write as $(\Free_{\Lie_{E_\ast}}^{\Lie_{{\mathcal{H}_u}}},\U_{\Lie_{E_\ast}}^{\Lie_{{\mathcal{H}_u}}})$ and $(\Free_{\Mod_{{\mathcal{H}_u}}}^{\Lie_{{\mathcal{H}_u}}},\U_{\Mod_{{\mathcal{H}_u}}}^{\Lie_{{\mathcal{H}_u}}})$, respectively.
We now place these notions within a homotopical setting. The next result follows from the same considerations as Proposition \ref{prop:lie model structure} above.

 \begin{proposition}[Transferred model structures]\label{prop:hecke lie model structure}
Right transferred model structures exist along the following adjunctions:\vspace{-6pt}
\begin{diagram}
\sLie_{E_\ast}^\NN&  & &\pile{ \rTo^{\Free_{\Lie_{E_\ast}}^{\Lie_{{\mathcal{H}_u}}}}_{\perp} \\ \lTo_{\U^{\Lie_{{\mathcal{H}_u}}}_{\Lie_{E_\ast}}}} & & & \sLie_{{\mathcal{H}_u}}^\NN\\
\\
\uTo^{ \Free^{\Lie_{E_\ast}}_{\Mod_{E_\ast}}} \ \dashv \ \dTo_{\U^{\Lie_{E_\ast}}_{\Mod_{E_\ast}}} & & & &&&   \dTo^{\U^{\Lie_{{\mathcal{H}_u}}}_{\Mod^{{\mathcal{H}_u}}}}\  \vdash \ \uTo_{\Free^{\Lie_{{\mathcal{H}_u}}}_{\Mod^{{\mathcal{H}_u}}}}  \\
\\
\sMod_{E_*}^\NN& && \pile{\lTo^{\U^{\Mod_{{\mathcal{H}_u}}}_{\Mod_{E_\ast}}}_{\top} \\\rTo_{\Free^{\Mod_{{\mathcal{H}_u}}}_{\Mod_{E_\ast}}} } && &  \sMod_{\mathcal{H}_u}^\NN, 
\end{diagram} 
\end{proposition}

Bar constructions again provide convenient cofibrant replacements in these model categories:
\begin{lemma}[Cofibrant replacement]\label{lem:hecke bar constructions}
Let $\mathfrak{g}$ be a simplicial Hecke Lie algebra and $M$ a simplicial Hecke module. \vspace{-3pt}
\begin{enumerate} 
\item The natural maps $\Barr_\bullet(\Free_{\Mod_{E_\ast}}^{\Lie_{{\mathcal{H}_u}}}, \LL, \mathfrak{g})\to \mathfrak{g} $ and $\Barr_\bullet(\Free_{\Mod_{E_\ast}}^{\Mod_{{\mathcal{H}_u}}}, \AAA, M)\to M$ are weak equivalences. 
\item If $\mathfrak{g} $ is levelwise projective, then $\Barr_\bullet(\Free_{\Mod_{E_\ast}}^{\Lie_{{\mathcal{H}_u}}}, \LL, \mathfrak{g})$ , $\U_{\Mod^{{\mathcal{H}_u}}}^{\Lie^{\mathcal{H}_u}}(\Barr_\bullet(\Free_{\Mod_{E_\ast}}^{\Lie_{{\mathcal{H}_u}}}, \LL, \mathfrak{g}))$ are both cofibrant.
\item If $M$ is levelwise projective, then $\Barr_\bullet(\Free_{\Mod_{E_\ast}}^{\Mod_{{\mathcal{H}_u}}}, \AAA, M)$ is cofibrant.
\end{enumerate}
\end{lemma} 
\begin{proof}
The first and third claim follow in the same way as Lemma \ref{lem:bar resolution}. For the second, one needs the further observation that $\U_{\Mod_{\mathcal{H}_u}}^{\Lie_{\mathcal{H}_u}}\circ\Free_{\Mod_{E_\ast}}^{\Lie_{\mathcal{H}_u}}$ takes values in free Hecke modules.
\end{proof}

\subsection{Indecomposables}\label{indec}
In this section, we will extend the adjunction considered in Section \ref{section:lie algebra homology} to a commuting diagram of adjunctions of the form
\begin{diagram}
\Lie_{{\mathcal{H}_u}}^\NN &  & &\pile{\rTo^{\Q_{\Lie_{{\mathcal{H}_u}}}^{\Lie_{E_\ast}}}_\perp\\ \lTo_{\triv_{\Lie_{{\mathcal{H}_u}}}^{\Lie_{E_\ast}}}} & & & \Lie^{\NN}_{E_\ast}\\
\\
\dTo^{\Q_{\Lie_{{\mathcal{H}_u}}}^{\Mod_{{\mathcal{H}_u}}}} \ \dashv \ \uTo_{\triv_{\Lie_{{\mathcal{H}_u}}}^{\Mod_{{\mathcal{H}_u}}}} & & & &&&      \uTo^{\triv_{\Lie}^{\Mod}} \ \dashv \ \dTo_{\Q_{\Lie}^{\Mod}}\\
\\
\Mod^{{\mathcal{H}_u}}_{E_\ast} & && \pile{\lTo^{\triv_{\Mod_{{\mathcal{H}_u}}}^{\Mod_{E_\ast}}}_\top \\\rTo_{\Q_{\Mod_{{\mathcal{H}_u}}}^{\Mod_{E_\ast}}} } && &  \Mod^\NN_{E_\ast} 
\end{diagram}
The right adjoints take an algebraic structure and produce a richer one by defining certain operations to be identically zero. The left adjoints are functors of indecomposables, which take an algebraic structure and produce a simpler one by forming the quotient by the image of certain operations. 

More formally, we make the following definitions at the level of objects for an $E_*$-module $M$, a Hecke module $N$, a Lie algebra $\mathfrak{g}$, and Hecke Lie algebra $\mathfrak{h}$.

\begin{enumerate} 
\item The underlying Hecke module of $\triv_{\Lie_{{\mathcal{H}_u}}}^{\Mod_{{\mathcal{H}_u}}}(N)$ is $N$, and the Lie bracket is the zero map. The underlying $E_*$-module of $\Q_{\Lie_{{\mathcal{H}_u}}}^{\Mod_{{\mathcal{H}_u}}}(\mathfrak{h})$ is the quotient of $\U_{\Mod_{E_\ast}}^{\Lie_{\mathcal{H}_u}}(\mathfrak{h})$ by the submodule generated by the elements of the form $\alpha([x_1,[x_2,[ \ldots , [x_{n-1},x_n]\ldots ]]])$ with $\alpha\in {\mathcal{H}_u^{\Lie}}$, $n>1$, and $x_1,\ldots,x_n \in \mathfrak{g}$, and the Hecke operations descend to the quotient.\\

\item The underlying graded abelian group of ${\triv_{\Mod_{{\mathcal{H}_u}}}^{\Mod_{E_\ast}}}(M)$ is that of $M$, and the Hecke module structure is defined by setting \[ \alpha (x) = \begin{cases}\ \  \alpha \cdot x & \mbox{if }  \alpha \in ({\mathcal{H}_u^{\Lie}})_i^j[1]\cong E_{j-i}
\\ \ \ 0 & \mbox{if } w>1  \end{cases}\]
for $x\in {M}_i$ and $\alpha \in ({\mathcal{H}_u^{\Lie}})_i^j(w)$. The $E_*$-module $\Q_{\Mod_{{\mathcal{H}_u}}}^{\Mod_{E_\ast}}(N)$ is the quotient of $\U_{\Mod_{E_\ast}}^{\Mod_{\mathcal{H}_u}}(N)$ by the submodule generated by the elements of the form $\alpha(x)$ with $ \alpha \in ({\mathcal{H}_u^{\Lie}})_i^j(w)$ with $w>1$ and $x\in N_i$.\\

\item The underlying Hecke module of $\triv_{\Lie_{{\mathcal{H}_u}}}^{\Lie_{E_\ast}}(\mathfrak{g})$ is $\triv_{\Mod_{{\mathcal{H}_u}}}^{\Mod_{E_\ast}}(\U_{\Mod}^{\Lie}(\mathfrak{g}))$, the underlying Lie algebra is $\mathfrak{g}$, and the two define a Hecke Lie algebra structure. The underlying $E_*$-module of $\Q_{\Lie_{{\mathcal{H}_u}}}^{\Lie_{E_\ast}}(\mathfrak{h})$ is $\Q_{\Mod_{{\mathcal{H}_u}}}^{\Mod_{E_\ast}}(\U^{\Lie_{{\mathcal{H}_u}}}_{\Mod^{{\mathcal{H}_u}}}(\mathfrak{h}))$, and the Lie bracket descends to the quotient.
\end{enumerate}

These definitions extend to arrows in obvious ways, and it is easily checked that  the claimed adjunctions hold and that the respective diagrams of left and right adjoints commute up to unique natural isomorphism. 

We obtain a simplicial variant of our diagram of adjunctions by applying $\Fun(\Delta^{op}, -)$. These adjunctions are all Quillen, since the right adjoints each preserve fibrations and trivial fibrations by definition. \\

In particular, we obtain a\vspace{3pt} Quillen adjunction
\begin{diagram}\Q_{\Lie_{{\mathcal{H}_u}}}^{\Mod_{E_\ast}}: \sLie_{{\mathcal{H}_u}}^\NN &    &&\pile{\rTo
  \\ \perp \\  
\lTo} & &  & \sMod_{E_\ast}^\NN: \triv_{\Lie_{{\mathcal{H}_u}}}^{\Mod_{E_\ast}},\end{diagram}   
\ \\ where $\Q_{\Lie_{{\mathcal{H}_u}}}^{\Mod_{E_\ast}}=\Q_{\Lie}^{\Mod}\circ\Q_{\Lie_{{\mathcal{H}_u}}}^{\Lie_{E_\ast}}\cong\Q_{\Mod_{{\mathcal{H}_u}}}^{\Mod_{E_\ast}}\circ\Q_{\Lie_{{\mathcal{H}_u}}}^{\Mod_{{\mathcal{H}_u}}}$.\vspace{5pt}

\begin{definition}[Hecke Lie algebra homology]\label{def:hecke lie homology}
The  \emph{Hecke Lie algebra homology} of a simplicial Hecke Lie algebra $\mathfrak{g} $ is the bigraded $E_\ast$-module \[\HHLie(\mathfrak{g})=H\left(N\left(\LQ_{\Lie_{\mathcal{H}_u}}^{\Mod_{E_\ast}}(\mathfrak{g})\right)[1]\oplus E_*\right).\vspace{5pt}\] 
\end{definition}
 
\subsection{The Hecke--Chevalley--Eilenberg complex}
In this section, we introduce a small complex for computing Hecke Lie algebra homology. Using the identity $$\Q_{\Lie_{{\mathcal{H}_u}}}^{\Mod_{E_\ast}}= \Q_{\Lie_{E_\ast}}^{\Mod_{E_\ast}}\circ\Q_{\Lie_{{\mathcal{H}_u}}}^{\Lie_{E_\ast}},$$ our strategy will be to pair an understanding of the left derived functor of $\Q_{\Lie_{{\mathcal{H}_u}}}^{\Lie_{E_\ast}}$ with our earlier exploration of Lie algebra homology.

\begin{construction}[The additive resolution]\label{construction:additive resolution}
Given a Hecke Lie algebra $\mathfrak{g}$, we endow $$\Barr_\bullet (\id,\AAA,\U_{\Mod^{\mathcal{H}_u}}^{\Lie_{{\mathcal{H}_u}}}(\mathfrak{g}))$$ with the structure of a simplicial Lie algebra by defining the Lie bracket by the equation
\[\bigg[\ [\alpha_1 | \cdots | \alpha_n | x]\ , \ [\beta_1 | \cdots | \beta_n | y] \ \bigg]  = \begin{cases}  \ \ \alpha_1 \cdots \alpha_n \beta_1 \cdots  \beta_n [x,y] &\mbox{if all $\alpha_i, \beta_j$ have weight $1$} \\ 
\ \  0 & \mbox{otherwise.}  \end{cases}\] One checks that this operation is well-defined, satisfies the identities of a Lie algebra, and respects the simplicial structure maps. The same formula defines the structure of a simplicial Hecke Lie algebra on the simplicial Hecke module $\Barr_\bullet (\Free_{\Mod_{E_\ast}}^{\Mod_{{\mathcal{H}_u}}},\AAA,\U_{\Mod^{\mathcal{H}_u}}^{\Lie_{{\mathcal{H}_u}}}(\mathfrak{g}))$. \\ \noindent We denote these algebraically enhanced bar construction by $\AR(\mathfrak{g})$ and $\overline{\AR}(\mathfrak{g})$, respectively, and refer to the former as the \emph{additive resolution} of $\mathfrak{g}$. 
\end{construction}

\begin{definition}[Hecke Chevalley--Eilenberg complex]\label{def:hecke chevalley--eilenberg}
The \emph{Hecke Chevalley--Eilenberg complex} of a Hecke Lie algebra $\mathfrak{g}$ is the simplicial chain complex of weighted $E_\ast$-modules\vspace{-2pt}
\[\CE_{\mathcal{H}_u}(\mathfrak{g})=\CE\left(\AR(\mathfrak{g})\right).\vspace{-2pt}\]
If $\mathfrak{g} $ is a simplicial Hecke Lie algebra, then $\CE_{\mathcal{H}_u}(\mathfrak{g})$ is defined as the bisimplicial chain complex obtained by applying the previous construction levelwise.
\end{definition}

\begin{notation}[Gradings in $\CE_{\mathcal{H}_u}(\mathfrak{g})$]\label{quadruple}
Given a   (weighted) Hecke Lie algebra $\mathfrak{g}$, the Hecke Chevalley--Eilenberg complex $\CE_{\mathcal{H}_u}(\mathfrak{g})$ is a simplicial chain complex in the abelian category $\Mod_{E_\ast}^\NN$ of weighted graded $E_\ast$-modules. It is therefore  equipped with four different gradings, which we will list as a tuple $(i,j,r,w)$ in the following order:
\begin{enumerate} 
\item $i$ is the \textit{internal degree} (coming from the abelian category of graded $E_\ast$-modules).
\item $j$ is the  \textit{homological degree} (corresponding to the divided  power degree in the Chevalley--Eilenberg complex in \Cref{dfn:chevalley--eilenberg}, i.e. the chain complex direction)
\item $r$ is the \textit{simplicial degree} (indicating the position in the additive resolution $\AR(\mathfrak{g})$)
\item $w$ is the weight (coming from the ambient weight-grading on the category $\Mod_{E_\ast}^\NN$).
\end{enumerate}

\end{notation}
We can now state the main theorem of this section.

\begin{theorem}\label{thm:hecke chevalley--eilenberg works}
There is a natural isomorphism $\HHLie(\mathfrak{g})\cong H(\CE_{\mathcal{H}_u}(\mathfrak{g}))$ for $E_*$-projective simplicial Hecke Lie algebras $\mathfrak{g}$.
\end{theorem}

The theorem follows by combining Theorem \ref{thm:chevalley--eilenberg works} with the following result.

\begin{proposition}\label{prop:additive resolution works} There is a natural weak equivalence of simplicial Lie algebras \[\LQ_{\Lie_{{\mathcal{H}_u}}}^{\Lie_{E_\ast}}(\mathfrak{g}) \simeq\AR(\mathfrak{g})\] for $E_\ast$-projective simplicial Hecke Lie algebras $\mathfrak{g}$.
\end{proposition}
\begin{proof}
Because brackets of Hecke power operations vanish, the Hecke Lie structure map of $\mathfrak{g}$ factors canonically as the composite $\LL(\mathfrak{g}) \rightarrow \AAA(\mathfrak{g}) \rightarrow \mathfrak{g}$, and these maps extend to maps $\Barr_\bullet(\Free_{\Mod_{E_\ast}}^{\Lie_{{\mathcal{H}_u}}},\LL,\mathfrak{g}) \to  \overline{\AR}(\mathfrak{g})\to\mathfrak{g}$ of simplicial Hecke Lie algebras, from which we obtain the commutative diagram
\[\xymatrix{
\LQ_{\Lie_{{\mathcal{H}_u}}}^{\Lie_{E_\ast}}\left(\Barr_\bullet(\Free_{\Mod_{E_\ast}}^{\Lie_{{\mathcal{H}_u}}},\LL,\mathfrak{g})\right)\ar[r]\ar[d]&\LQ_{\Lie_{{\mathcal{H}_u}}}^{\Lie_{E_\ast}}\left(\,\overline{\AR}(\mathfrak{g})\right)\ar[r]\ar[d]&\LQ_{\Lie_{{\mathcal{H}_u}}}^{\Lie_{E_\ast}}(\mathfrak{g})\\
\Q_{\Lie_{{\mathcal{H}_u}}}^{\Lie_{E_\ast}}\left(\Barr_\bullet(\Free_{\Mod_{E_\ast}}^{\Lie_{{\mathcal{H}_u}}},\LL,\mathfrak{g})\right)\ar[r]&\Q_{\Lie_{{\mathcal{H}_u}}}^{\Lie_{E_\ast}}\left(\,\overline{\AR}(\mathfrak{g})\right).
}\] We will argue that all of the arrows in this diagram are weak equivalences. With this claim established, the result will follow, since the map $\overline{\AR}(\mathfrak{g})\to \AR(\mathfrak{g})$ induced by the augmentation of $\AAA$ induces an isomorphism \[\Q_{\Lie_{{\mathcal{H}_u}}}^{\Lie_{E_\ast}}\left(\,\overline{\AR}(\mathfrak{g})\right)\xrightarrow{\simeq} \AR(\mathfrak{g})\] of simplicial Lie algebras.

The top row of horizontal maps and the lefthand vertical map are weak equivalences by Lemma \ref{lem:hecke bar constructions}, so it remains to verify that the weak equivalence $\Barr_\bullet(\Free_{\Mod_{E_\ast}}^{\Lie_{{\mathcal{H}_u}}},\LL, \mathfrak{g}) \xrightarrow{\sim} \overline{\AR}(\,\mathfrak{g})$ is preserved by $\Q_{\Lie_{{\mathcal{H}_u}}}^{\Lie_{E_\ast}}$, which, after applying $\U^{\Lie_{E_\ast}}_{\Mod_{E_\ast}}$ and using the isomorphism \[\\U^{\Lie_{E_\ast}}_{\Mod_{E_\ast}}\circ \Q_{\Lie_{{\mathcal{H}_u}}}^{\Lie_{E_\ast}}\cong \Q_{\Mod_{{\mathcal{H}_u}}}^{\Mod_{E_\ast}} \circ \U^{\Lie_{{\mathcal{H}_u}}}_{\Mod^{{\mathcal{H}_u}}},\] is equivalent to showing that $\Q_{\Mod_{{\mathcal{H}_u}}}^{\Mod_{E_\ast}}$ preserves the weak equivalence \[\U^{\Lie_{{\mathcal{H}_u}}}_{\Mod^{{\mathcal{H}_u}}}\left(\Barr_\bullet(\Free_{\Mod_{E_\ast}}^{\Lie_{{\mathcal{H}_u}}},\LL, \mathfrak{g})\right) \xrightarrow{\sim} \U^{\Lie_{{\mathcal{H}_u}}}_{\Mod^{{\mathcal{H}_u}}}\left(\,\overline{\AR}(\mathfrak{g})\right)\cong \Barr_\bullet\left(\Free_{\Mod_{E_\ast}}^{\Mod_{{\mathcal{H}_u}}},\AAA,  \U^{\Lie_{{\mathcal{H}_u}}}_{\Mod^{{\mathcal{H}_u}}}(\mathfrak{g})\right).\] Since $\Q_{\Mod_{{\mathcal{H}_u}}}^{\Mod_{E_\ast}}$ is left Quillen, it suffices to observe that the source and target of this weak equivalence are both cofibrant by Lemma \ref{lem:hecke bar constructions}
\end{proof}

\subsection{Spectral Lie algebras and Hecke Lie algebras}\label{section:spectral lie}

The structure of a Hecke Lie algebra is the shadow in $E$-homology cast by a much richer structure, which lifts the Lie algebra axioms from classical algebra to the world of stable homotopy. These \emph{spectral Lie algebras} are spectra equipped with an action of the spectral Lie operad, the operadic Koszul dual of the commutative operad in spectra. This operad was first studied in \cite{salvatore1998configuration} and \cite{ching2005bar}, and its algebras have been the subject of much recent study in a variety of contexts \cite{antolin2015mod,kjaer2016odd,survey,Gijsbertszeugs}.

In the context of Lubin--Tate theory, we have the following result, which was established in the first author's thesis \cite[Theorem 4.4.4.]{brantnerthesis}. We write $\Free^{\mathscr{L}_E}$ for the left adjoint to the forgetful functor on the $\infty$-category of spectral Lie algebras in $K(h)$-local $E$-module spectra.

\begin{theorem}[Brantner]\label{thm:brantner main}\ \vspace{-4pt}
\begin{enumerate} 
\item The homotopy groups of a (weighted) spectral Lie algebra in $K(h)$-local $E$-module spectra canonically form a (weighted) Hecke Lie algebra.\vspace{2pt}
\item For a flat (weighted) $E$-module spectrum $M$, the canonical map
$$\Free_{\Mod_{E_\ast}}^{\Lie_{{\mathcal{H}_u}}}(\pi_\ast(M))\rightarrow \pi_\ast\left(\Free^{\mathscr{L}_E}( L_{K(h)}(M))\right)$$
induces an isomorphism on completions.\vspace{-3pt}
\end{enumerate}
\end{theorem}

We pause to explain the meaning of the weighted variant of the theorem, which follows immediately from the argument of \cite[Theorem 4.4.4.]{brantnerthesis} by formally recording weights. The $K(h)$-local Lie operad acts on weighted $K(h)$-local $E$-module spectra (i.e. elements of the functor category $\Fun(\NN, \Mod_E^\wedge)$ with Day convolution) by placing the entire operad in weight $0$. 

At the level of homotopy groups, the resulting \emph{weighted $K(h)$-local Lie algebras} are naturally equipped with the following asserted in  Theorem \ref{thm:brantner main} (1):
 \[[-,-]: \pi_i(\mathfrak{g}(a)) \times \pi_j(\mathfrak{g}(b))\rightarrow \pi_{i+j}(\mathfrak{g}(a+b))\]\[( {\mathcal{H}_u^{\Lie}})_{i}^{j}(b)  \times \pi_i (\mathfrak{g}(a)) \rightarrow \pi_j (\mathfrak{g}(b\cdot a))\] 

This \emph{weighted Hecke Lie algebra} structure is controlled by an extension of the monad $\LL$ to weighted $E_*$-modules. For this, we endow the free 
Hecke Lie algebra  $\LL(M)$ on a weighted graded $E_\ast$-module $M$ with a weight-grading. It has the property that if 
$\{ m_s\in M_{i_s}(a_s)\}_{s=1}^k$ are homogeneous, 
$w$ is a Lyndon word in $k$ letters involving the $i^{th}$ letter $n_i$ times, 
and $\alpha \in ( {\mathcal{H}_u^{\Lie}})_{\sum_s i_s}^{j}(b)$ is a Hecke operation of weight $b$, then 
$\alpha (w(m_1,\ldots,m_k))$ has weight $b \cdot (\sum_s m_s a_s)$.

The  free $K(h)$-local Lie algebra on a finite and free $E$-module  spectrum is usually neither finite nor free due to the presence of completed  infinite direct sums. Instead, it is only \textit{completed-free}, i.e. the $K(h)$-localisation of a free $E$-module spectrum. This is reflected by the algebraic completion on homotopy groups appearing in part $(2)$ of Theorem \ref{thm:brantner main}.

The weighted context allows us to bypass this complication by restricting attention to the full subcategory $\Fun(\NN,\Mod_{E}^{\wedge})^{\pff,>0}$ of graded $E$-module spectra which are  pointwise finite  free  in each weight and \textit{concentrated in positive weights}.  Using the explicit description of free Hecke Lie algebras provided in 
\cite[Section 4.4.2]{brantnerthesis}, we see that the free Hecke Lie algebra on some $M\in\Fun(\NN,\Mod_{E}^{\wedge})^{\pff,>0}$  is again finite and free in each weight. Hence, we can deduce the 
following simplification of Theorem \ref{thm:brantner main}:
\begin{corollary}\label{finitefree}
The monad $\Free^{\mathscr{L}}_E$ preserves  $\Fun(\NN,\Mod_{E}^{\wedge})^{\pff,>0}$, and the canonical map $$\Free_{\Mod_{E_\ast}}^{\Lie_{{\mathcal{H}_u}}}(\pi_\ast(M))\rightarrow \pi_\ast\left(\Free^{\mathscr{L}_E}( L_{K(h)}(M))\right)$$ is an isomorphism for any $M\in\Fun(\NN,\Mod_{E}^{\wedge})^{\pff,>0}$.
\end{corollary}

\subsection{Looping}
The $\infty$-category of weighted $K(h)$-local Lie algebras admits small limits and the forgetful functor creates them; in particular, we can form the loop object $\Omega \mathfrak{g} \simeq 0 \times_{\mathfrak{g}} 0$ on such a Lie algebra $\mathfrak{g}$, the underlying spectrum of which is simply the desuspension of the underlying spectrum of $\mathfrak{g}$. 

Our next result describes the Hecke Lie structure on $\pi_\ast(\Omega^n \mathfrak{g})$ in terms of that on $\pi_\ast(\mathfrak{g})$.

\begin{proposition}[Hecke operations on looped Lie algebras]\label{prop:hecke desuspension}
For $\mathfrak{g}$ a weighted $K(h)$-local Lie algebra and $n>0$, the Hecke Lie algebra $\pi_\ast(\Omega^n \mathfrak{g})$ has vanishing Lie bracket, and the Hecke module structure is given by the following commutative diagram: 
\begin{diagram}({\mathcal{H}_u^{\Lie}})_i^j(w)\times \pi_i(\Omega^n \mathfrak{g}(a))   & & \rTo & & \pi_j(\Omega^n \mathfrak{g}(w\cdot a))  \\
\dTo^{\Susp^n \times (\cong)} & && & \dTo_\cong  \\
({\mathcal{H}_u^{\Lie}})_{i+n}^{j+n}(w) \times   \pi_{i+n}(\mathfrak{g}(a))  & & \rTo & & \pi_{j+n}(\mathfrak{g}(w\cdot a)).
\end{diagram} 
For $w=p^s$, the morphism $ ({\mathcal{H}_u^{\Lie}})_i^j(p^s) \xrightarrow{\Susp^n}({\mathcal{H}_u^{\Lie}})_{i+n}^{j+n}(p^s)$ is given by the morphism  $$\Ext^s_{\Gamma^{i}}({E}_0,{E}_{-i+j+s} ) \rightarrow  \Ext^s_{\Gamma^{i+n}}({E}_0,{E}_{-i+j+s} )$$ induced by the suspension morphism $\Gamma^{i+n} \rightarrow \Gamma^i$ linking Rezk's $\Gamma$-rings. 
\end{proposition}
\begin{proof}
Let $\alpha  \in \pi_j(\Free^{\mathscr{L}_E}(\Sigma^iE))$ be a universal operation from degree $i$ to degree $j$ and suppose that $x\in \pi_i(\Omega^n \mathfrak{g})$ is represented by a map $\Sigma^i E \rightarrow \Omega^n \mathfrak{g}$.

The class $\alpha(x)$ is then represented by the following diagram:
\begin{diagram}
\Sigma^jE & \rTo^{\alpha} &  \Free^{\mathscr{L}_E}(\Sigma^i E)& & & & \rTo^{\overline{x}} &     &  \Omega^n(\mathfrak{g}) \\
                 &         &                                                    \dTo       &    & &   &  &    \ruTo(6,2)  &  \uTo \\
                &         &                                               \Free^{\mathscr{L}_E}( \Omega^n(\mathfrak{g}))    & &     &  \rDashto         &  &   &  \Omega^n(\Free^{\mathscr{L}_E}( \mathfrak{g})) 
\end{diagram}
Here we have factored the structure map of the Lie algebra $\Omega^n(\mathfrak{g})$ through the canonical map $      \Free^{\mathscr{L}_E}( \Omega^n(\mathfrak{g}))       \rightarrow  \Omega^n(\Free^{\mathscr{L}_E}( \mathfrak{g})) $ and the structure map of $\mathfrak{g}$; a more detailed construction of this map is given in the proof of Proposition \ref{prop:bracket}.

The map $\alpha(x): \Sigma^jE \rightarrow \Omega^n (\mathfrak{g})$ can therefore be represented by shifting the composite from top left to bottom right in the following diagram down by $n$: 
\begin{diagram}
\Sigma^{j+n}E & \rTo^{\Sigma^n(\alpha)} & \Sigma^n \Free^{\mathscr{L}_E}(\Sigma^i E)& \rTo &  \Sigma^n  \Free^{\mathscr{L}_E}( \Omega^n(\mathfrak{g}))    \\ 
&    \rdDashto(2,3)&  \dTo & & \dTo   &   \rdTo(2,2)  \\ 
  \\ 
&  &     \Free^{\mathscr{L}_E}(\Sigma^{i+n} E) &  \rTo &       \Free^{\mathscr{L}_E}( \mathfrak{g})   & \rTo &  \mathfrak{g}
\end{diagram}
The first claim of the proposition now follows by observing that the dotted diagonal arrow picks out $\Susp^n(\alpha)$.
The second claim is established in Theorem 4.2.19 of \cite{brantnerthesis}.\end{proof}
Motivated by \Cref{prop:hecke desuspension}, we introduce the following definition:
\begin{definition}[Looping Hecke Modules] \label{loopingmodules}
Given a Hecke module $M$, we write $\Omega M$ for the Hecke module with $\U^{\Mod^{\mathcal{H}_u}}_{\Mod}(\Omega M)=\Sigma^{-1} M$ and Hecke operations defined by the formula of Proposition \ref{prop:hecke desuspension}. This construction extends to an endofunctor of Hecke modules. \end{definition}

In order to understand this endofunctor, we will use that Hecke operations of weight $p$ have a particularly simple description.  By the Koszulness of the ring $\Gamma^{i}$ \cite{rezk2017rings}, we have 
$$({\mathcal{H}_u^{\Lie}})_i^{i-1}(p) = 
\Ext^1_{\Gamma^i}( {E}_0, {E}_{0})\cong \Gamma^i(p)^\vee.\vspace{3pt}$$ Combining this isomorphism with the $E_0$-linear dual of Proposition \ref{suspensiondiagram} and invoking Proposition \ref{prop:hecke desuspension}, we deduce the following:\vspace{3pt}

\begin{proposition}[Suspending via the Euler class]\label{Gamma0}
There is a commutative diagram 
\begin{diagram}
\ldots & \rTo^{\cong}&&  \Gamma^0(p)^\vee& &    \rInto ^{(e\cdot -) } &  &\Gamma^0(p)^\vee& &\rTo^{\cong}  & & \Gamma^0(p)^\vee& &   \rInto^{(e\cdot -) }   & & \Gamma^0(p)^\vee&&    \rTo^{\cong} &  \ldots \\
 && &  \dTo^\cong    && &  &  \dTo^\cong  & & & & \dTo^\cong    & && &  \dTo^\cong    \\
\ldots &  \rTo^\cong   & & \Gamma^{-1}(p)^\vee&& \rInto & &\Gamma^0(p)^\vee&&   \rTo^\cong & & \Gamma^{ 1}(p)^\vee&&\rInto    & & \Gamma^{2}(p)^\vee&&     \rTo^\cong&  \ldots \\
 && &  \dTo^\cong    && &  &  \dTo^\cong  & & & & \dTo^\cong    & && &  \dTo^\cong    \\
 \ldots &  \rTo^\cong   & & ({\mathcal{H}_u^{\Lie}})_{-1}^{-2}(p) && \rInto & &({\mathcal{H}_u^{\Lie}})_0^{-1}(p) &&   \rTo^\cong & &({\mathcal{H}_u^{\Lie}})_1^{0}(p) &&\rInto    & & ({\mathcal{H}_u^{\Lie}})_2^{1}(p) &&     \rTo^\cong&  \ldots \\
\end{diagram}\ \vspace{3pt}

\end{proposition}

\begin{corollary}
For $\mathfrak{g}$ a weighted $K(h)$-local Lie algebra and $n>0$, the following diagram commutes: 
\begin{diagram}\label{cor:loop operations}
 & & \Gamma^0(p)^\vee \times \pi_0(\Omega^i(\mathfrak{g}(a))) &  \rTo  & ({\mathcal{H}_u^{\Lie}})_0^{-1}(p) \times \pi_0(\Omega^i(\mathfrak{g}(a))) &  \rTo  &  \pi_{-1}(\Omega^i(\mathfrak{g}(p\cdot a))) \\
 & \ldTo^{\left(e^{\lfloor \frac{i}{2}\rfloor} \cdot -\right)} &\dTo & &  \dTo^{\Susp\otimes (\cong)} & &   \dTo_\cong \\
 \Gamma^{0}(p)^\vee \times   \pi_i(\mathfrak{g}(a)) & \rTo_{} &\Gamma^{i}(p)^\vee \times   \pi_i(\mathfrak{g}(a)) & \rTo & ({\mathcal{H}_u^{\Lie}})_i^{i-1}(p)  \times \pi_i(R(a)) & \rTo & \pi_{i-1}(\mathfrak{g}(p\cdot a))\\ 
\end{diagram} \ \vspace{-8pt} \\ ${}_{ \ }$
\ \ \ \ \ \ \ \ \ \ \ \ \ \ \ \  $^{(\mathrm{Thom})^\vee \times \id}$
  \ \vspace{3pt}
\end{corollary}

\subsection{Delooping}

In good circumstances, maps between   looped Hecke modules can be  delooped.

\begin{proposition}[Delooping maps of Hecke modules]\label{prop:hecke delooping}
Let $f:\Omega M_1\to \Omega M_2$ be a map of Hecke modules. If $M_1$ and $M_2$ are both torsion-free, then $f=\Omega g$ for a unique map of Hecke modules $g:M_1\to M_2$.
\end{proposition}

In order to prove this result, we shall need to bound the torsion of the cokernel of suspensions. We begin by introducing 
 the following notation:
\begin{definition}[Torsion cokernel] Let   $N$ be an integer.
We say that a map of abelian groups $f: A\rightarrow B$ has \textit{torsion cokernel of exponent dividing $N$} if, for any $b\in B$, the element $N\cdot b$ lies in the \mbox{image of $f$.}
\end{definition}
Proposition \ref{explicit!} allows us to deduce the following (known) result:
\begin{proposition}[Torsion cokernel for operations on $\EE_\infty$-rings]\label{gammaexponent}
The $k$-fold suspension homomorphism $\Gamma^i(p^a) \hookrightarrow \Gamma^{i-k}(p^a)$ has torsion cokernel of exponent dividing $p^{a k}$. 
\end{proposition}
\begin{proof}
If $w=p$, then Proposition \ref{explicit!}  (or in fact also \cite[Proposition 10.6]{rezk2009congruence}) shows that the above suspension morphism has torsion cokernel of exponent dividing $p^k$.
If $w=p^a$,  we use that $\Gamma^{i-k}$ is Koszul with respect to the ``$p$-logarithmic" weight-grading (cf. \cite{rezk2017rings}) to write $\alpha$ as a composite of $a$ weight $p$ operations. The exponent   therefore divides  $(p^k)^a$.
If $w$ is not a power of $p$, there is nothing to prove.
\end{proof}
\begin{remark}
The above exponent is not optimal. One can find a sharper bound by taking the parity of $i$ into account.
\end{remark}

On the Lie algebra side, we can deduce the following statement:
\begin{proposition}[Torsion cokernel for operations on Lie algebras]\label{Hexponent}
The  $k$-fold suspension homomorphism $({\mathcal{H}_u^{\Lie}})_i^j(p^a) \rightarrow ({\mathcal{H}_u^{\Lie}})_{i+k}^{j+k}(p^a)$ has torsion cokernel of exponent dividing $p^{a^2k}$.
\end{proposition}
\begin{proof} 
The  morphism in question given by the map
$ \Ext^a_{\Gamma^{i}}({E}_0,{E}_{-i+j+a} ) \rightarrow \Ext^a_{\Gamma^{i+k}}({E}_0,{E}_{-i+j+a} ) $
induced by the suspension homomorphism $\Gamma^{i+k} \rightarrow \Gamma^i$. 
By Koszulness of the ring $\Gamma^{i+k}$, we can represent every class in $ \Ext^a_{\Gamma^{i+k}}({E}_0,{E}_{-i+j+a})$ by an $E_0$-linear map $f:\Gamma^{i+k}(p)^{\otimes a} \rightarrow {E}_{-i+j+a}$.

For any $x\in \Gamma^{i}(p)^{\otimes a}$, Proposition \ref{gammaexponent} implies that
$(p^{ak})^a x$ lies in $\Gamma^{i+k}(p)^{\otimes a}$, where we have identified $\Gamma^{i+k}(p)^{\otimes a}$ with its image under the injection $( \Gamma^{i+k}(p))^{\otimes a} \hookrightarrow (\Gamma^{i}(p))^{\otimes a}$. We can therefore define a map $g: \Gamma^{i}(p)^{\otimes a}  \rightarrow  {E}_{-i+j+a}$ by $g(x) = f(p^{a^2k} x)$
The composite $\Gamma^{i+k}(p))^{\otimes a} \rightarrow \Gamma^{i}(p))^{\otimes a}  \xrightarrow{g}  {E}_{-i+j+a}$ agrees with $p^{a^2k} f$, and this clearly implies that the class $p^{a^2k} [f] \in  \Ext^a_{\Gamma^{i+k}}({E}_0,{E}_{-i+j+a})$ lies in the image of the map $({\mathcal{H}_u^{\Lie}})_i^j(p^a) \rightarrow ({\mathcal{H}_u^{\Lie}})_{i+k}^{j+k}(p^a)$.
\end{proof}

We can finally establish the delooping claim made in the beginning of this section.

\begin{proof}[Proof of \Cref{prop:hecke delooping}]
For $i=1,2$, the $E_\ast$-modules $\Omega M_1$ and $ \Omega M_2$ are simply given by applying the shift functor $\Sigma^{-1}$ to $M_1$ and $M_2$, respectively. As a map of $E_\ast$-modules, define  $$g(x) := \Sigma(f(\Sigma^{-1}x)).$$

It remains to check that $g$ is indeed a map of  Hecke modules.  For this, assume $\alpha \in ({\mathcal{H}_u^{\Lie}})_{i}^j(w)$ is a given Hecke operation and that $x \in  (M_1)_i$ is an element. By Proposition  \ref{Hexponent}, 
we can then choose an integer $N$ such that $N\alpha = \Susp(\beta)$ for some $\beta \in  ({\mathcal{H}_u^{\Lie}})_{i-1}^{j-1}(w)$. 

Since $f$ is a map of Hecke Lie algebras, we have 
$$g((N\alpha)(x)) = g(\Susp(\beta)(x)) = \Sigma(f(\Sigma^{-1} \Susp(\beta)(x)))
=\Sigma(f(\beta(\Sigma^{-1}(x))))$$
$$=\Sigma(\beta(f(\Sigma^{-1}(x))))=\Sigma(\beta(\Sigma^{-1} g(x)))= \Susp(\beta) g(x) = (N\alpha) (g(x))$$

Hence $N(g(\alpha(x))-\alpha(g(x))) = 0 $, which implies $g(\alpha(x))=\alpha(g(x))$ by torsion-freeness. 
\end{proof}

\section{The $E$-theory of labeled configuration spaces}
\label{sec:SpectralSequenceInput}

In this section, we will construct spectral sequences converging to the Morava $E$-theory of (possibly labelled) configuration spaces of manifolds  and examine some of their basic properties. Throughout the remainder of the paper, all manifolds are assumed to be of finite type. 

\subsection{Stable formulas for configuration spaces} The homotopy types of configuration spaces are subtle invariants of the background manifold. For example, according to a theorem of Longoni--Salvatore \cite{LongoniSalvatore:CSNHI}, these homotopy types distinguish certain pairs of lens spaces which are homotopy equivalent  but not homeomorphic. In particular, configuration spaces are not homotopy invariants even of compact manifolds of equal dimension.

This subtlety is an unstable phenomenon. In fact, there is a relatively simple formula expressing the stable homotopy types of the configuration spaces of $M$ in terms of the pointed homotopy type of the one-point compactification $M^+$. This formula is valid for configuration spaces labeled by a spectrum $X$, which are defined as 
\[B_k(M;X):=\Sigma_+^\infty \Conf_k(M)\otimes_{\Sigma_k}X^{\otimes k}.\]
We naturally consider $B_k(M;X)$ as a weighted spectrum placed in weight $k$.
 
Writing $\mathscr{L}$ for the free spectral lie algebra monad acting on spectra or weighted spectra and $\Free^{\mathscr{L}}$ for the corresponding free functor, we recall the following equivalence established by the  third author (cf. \cite[Section 3.4.]{knudsen2016higher}:
 
 \begin{theorem}[Knudsen]\label{thm:conf bar}
Let $M$ be a framed $n$-manifold and $X$ a spectrum. There is a natural equivalence of weighted spectra \[\bigoplus_{k\geq1}B_k(M;X)\simeq \Sigma\left|\mathrm{Bar}\left(\id,\, \mathscr{L},\, \Free^{\mathscr{L}}(\Sigma^{n-1}X)^{M^+}\right)\right|.\] Here the lefthand side is weighted by the index $k$ and the righthand side by operadic arity. The superscript indicates the cotensor in the $\infty$-category of spectral Lie algebras.
\end{theorem}

\begin{example} The wedge of labeled configuration spaces appearing in Theorem \ref{thm:conf bar} has several familiar interpretations.
\begin{enumerate} 
\item In the case $M=\mathbb{R}^n$, this wedge forms the free $\mathbb{E}_n$-algebra on the spectrum $X$. For general $M$, it is the factorization homology (alias topological chiral homology) of $M$ with coefficients in this algebra \cite{Salvatore:CSSL, Lurie:HA, AyalaFrancis:FHTM}.

\item In the case $X=\Sigma^\infty Y$ for $Y$ a connected pointed space, this same wedge is equivalent to the suspension spectrum of the space $\mathrm{Map}_c(M, \Sigma^{n}Y)$ of compactly supported maps \cite{McDuff:CSPNP,Boedigheimer:SSMS}. In particular, setting $M=\mathbb{R}^n$ yields a formula for $\Sigma^\infty \Omega^n\Sigma^n Y$.
\end{enumerate}
\end{example}

\begin{remark}
We make several remarks on Theorem \ref{thm:conf bar}.
\begin{enumerate}
\item An analogous result is available for non-parallelizable and non-smooth manifolds at the cost of keeping track of actions of tangential structure groups. We restrict ourselves to the framed setting in the present work for the sake of simplicity.
\item In the case $M=\mathbb{R}^n$, the equivalence in question is compatible with the respective actions of the operad of little $n$-disks; indeed, the stated result is obtained by applying factorization homology to this more structured equivalence.
\item  The authors have subsequently been informed by Arone that an alternative proof using Goodwillie calculus is available in the case of a suspension spectrum.
\end{enumerate}
\end{remark}

Via the canonical filtration of the bar construction, Theorem \ref{thm:conf bar} supplies spectral sequences converging to $E_*B_k(M;X)$ for each $k$, $M$, $X$, and homology theory $E$. In characteristic zero, these spectral sequences all collapse, and the $\mathrm{E}^2$ page may be identified with the classical Lie algebra homology of a certain graded Lie algebra, which is eminently computable \cite{Knudsen:BNSCSVFH, DrummondColeKnudsen:BNCSS}. We show that the computational utility of Theorem \ref{thm:conf bar} is much wider in scope, and in particular applies to Morava $E$-theory.

\subsection{The $\mathrm{E}^2$-page of the bar spectral sequence} For the remainder of this paper, we take $E$ to be a Lubin--Tate theory of height $h$ as in Section \ref{sec:Hecke}. We now state our first main result.

\begin{theorem}[Hecke spectral sequence] \label{thm:main}
Let $M$ be a framed $n$-manifold and $X$ a spectrum, and suppose that the weighted Hecke Lie algebra \[\mathfrak{g}(M;X):=E_*^{\wedge}\left(\Free^{\mathscr{L}}(\Sigma^{n-1}X)^{M^+}\right)\] is a finite and free $E_\ast$-module in each weight. There is a convergent weighted spectral sequence 
\[\mathrm{E}^2_{s,t} \cong H_{s+1}(\CE_{{\mathcal{H}_u}}\left(\mathfrak{g}(M;X))\right)_{t-1} \implies \bigoplus_{k\geq0} E_{s+t}^{\wedge}(B_k(M;X)).\]
\end{theorem}

In \Cref{cor:finitefree}, we will give a convenient criterion to check that $\mathfrak{g}(M;X)$ is finite and free in each weight.

\begin{notation}\label{confusion}
We denote Lubin-Tate theory by an italic $E$, whereas the $n^{th}$ page of a (homological) spectral sequence is denoted by $\mathrm{E}^n$.
\end{notation}
\begin{proof}[Proof of \Cref{thm:main}]
We apply the functor $L_{K(h)}(E \otimes -)$ weightwise to the equivalence of weighted spectra asserted in Theorem \ref{thm:conf bar} and obtain an equivalence of $K(h)$-local $E$-module spectra 
 \[\left|\Sigma\mathrm{Bar}_\bullet \left(\id,\, \mathscr{L}_E,\,L_{K(h)}\left(E \otimes  \Free^{\mathscr{L}}(\Sigma^{n-1}X)^{M^+}\right)\right)(k)\right|\simeq  L_{K(h)}(E \otimes \Sigma_+^\infty B_k(M;X)).\]
for every $k\geq 1$. Taking the skeletal filtration, we obtain a spectral sequence
$$\mathrm{E}^2_{s,t}(k)= H_s\left( \pi_t\left(\Sigma\mathrm{Bar}_\bullet \left(\id,\, \mathscr{L}_E,\,L_{K(h)}\left(E \otimes  \Free^{\mathscr{L}}(\Sigma^{n-1}X)^{M^+}\right)\right)(k)\right)\right) \Rightarrow E^{\wedge}_{s+t}( B_k(M;X)),$$ which converges, as does the homotopy spectral sequence for any bounded below filtered spectrum (cf. e.g. \cite[Prop. 1.2.2.14]{Lurie:HA} for this classical fact). 

Collecting weights, adding a copy of $E_*$ in weight $0$ and bidegree $(s,t)=(-1,1)$, and applying Proposition \ref{finitefree} repeatedly yields the spectral sequence of the theorem. We conclude by identifying the $\mathrm{E}^2$ page as
\begin{align*}\mathrm{E}^2_{s,t} &= H_{s+1}\left(  \mathrm{Bar}_\bullet \left(\id,\, \mathbf{L}^{\mathcal{H}_u}  ,\,\mathfrak{g}(M;X)\right)[1] \oplus E_*\right)_{t-1 }\\
&\cong H_{s+1}\left(   \Q_{\Lie_{\mathcal{H}_u}}^{\Mod_{E_\ast}}
\mathrm{Bar}_\bullet \left(\mathbf{L}^{\mathcal{H}_u},\ , \mathbf{L}^{\mathcal{H}_u} ,\,\mathfrak{g}(M;X)\right)[1]\oplus E_*\right)_{t-1}\\
&\cong H_{s+1}\left(  \mathbb{L}\Q_{\Lie_{\mathcal{H}_u}}^{\Mod_{E_\ast}}\left(\mathfrak{g}(M;X)\right)[1]\oplus E_*\right)_{t-1}\\
&\cong H_{s+1}^{\Lie^{\mathcal{H}_u}}\left(\mathfrak{g}(M;X)\right)_{t-1}\\
&\cong H_{s+1}(\CE_{\mathcal{H}_u}\left(\mathfrak{g}(M;X))\right)_{t-1},
\end{align*}
where the second isomorphism uses the isomorphism $\Q_{\Lie_{{\mathcal{H}_u}}}^{\Mod_{E_\ast}}\Free_{\Mod_{E_\ast}}^{\Lie_{{\mathcal{H}_u}}}\cong \id$, the third uses Lemma \ref{lem:hecke bar constructions}, the fourth is Definition \ref{def:hecke lie homology}, and the last uses Theorem \ref{thm:hecke chevalley--eilenberg works}.
\end{proof}

We can also set up  a cohomological version of this spectral sequence:

\begin{theorem}[Cohomological Hecke spectral sequence] \label{thm:cohomology}
Let $M$ be a framed $n$-manifold of finite type and $X$ a spectrum, and suppose that the weighted Hecke Lie algebra \[\mathfrak{g}(M;X):=E_*^{\wedge}\left(\Free^{\mathscr{L}}(\Sigma^{n-1}X)^{M^+}\right)\] is a finite and free $E_\ast$-module in each weight. There is a weighted spectral sequence 
\[{\mathrm{E}_2^{s,t}} \cong H^{s+1}\left( \CE_{{\mathcal{H}_u}}\left(\mathfrak{g}(M;X) \right)^\vee \right)_{t+1} \implies \bigoplus_k E^{t-s}(B_k(M;X)).\] 
which converges completely (cf. e.g. \cite{MR0365573} or \cite[Definition 6.18.]{GoerssJardine:SHT}) 
\end{theorem}

\begin{proof}
We  apply the functor $E^{(-)} = \Map_{\Sp}(-,E)$ weightwise  to the equivalence asserted in Theorem \ref{thm:conf bar} and obtain an equivalence of $K(h)$-local $E$-module spectra 
$$ \Tot\left(E^{\Sigma\mathrm{Bar}_\bullet \left(\id,\, \mathscr{L}_E,\,L_{K(h)}\left(E \otimes  \Free^{\mathscr{L}}(\Sigma^{n-1}X)^{M^+}\right)\right)(k)}\right)
\simeq   E^{B_k(M;X))_+} .$$
for every $k\geq 0$. The coskeletal tower of this cosimplicial spectrum gives rise to a  spectral sequence   
$${\mathrm{E}_2^{s,t}} = H^{s}\left( \pi_t\left(\Sigma\mathrm{Bar}_\bullet \left(\id,\, \mathscr{L}_E,\,L_{K(h)}\left(E \otimes  \Free^{\mathscr{L}}(\Sigma^{n-1}X)^{M^+}\right)\right)(k)\right)^{\vee}\right) \Rightarrow E^{t-s}( B_k(M;X))$$
By repeated application of  Proposition \ref{finitefree}, we can rewrite the $\mathrm{E}^2$-page as 
$${\mathrm{E}_2^{s,t}}(k) = H^s\left(  \mathrm{Bar}_\bullet \left(\id,\, \mathbf{L}^{{\mathcal{H}_u}}  ,\,\mathfrak{g}(M;X)\right)^\vee  \right)_{t+1}(k). $$

Since $\mathfrak{g}(M;X)$ is concentrated in weights $\geq 1$, we know that \mbox{for $n>k$,} every element in $(\mathbf{L}^{{\mathcal{H}_u}})^{\circ n}(\mathfrak{g}(M;X)) (k)$ is the degeneracy of an element in $(\mathbf{L}^{{\mathcal{H}_u}})^{\circ k}(\mathfrak{g}(M;X)) (k)$.
The cosimplicial $E_\ast$-module $\mathrm{Bar}_\bullet \left(\id,\, \mathbf{L}^{{\mathcal{H}_u}}  ,\,\mathfrak{g}(M;X)\right)^\vee(k)$ is therefore $k$-coskeletal, and hence $E_s^{s,t}(k)$ vanishes for $s>k$. By \cite[Corollary 6.20.]{GoerssJardine:SHT}, this implies that the spectral sequence converges completely.

The description of the ${\mathrm{E}_2}$ page then follows as in Theorem \ref{thm:main}.
\end{proof}

\begin{remark}
We can compute $H^\ast\left( \CE_{{\mathcal{H}_u}}\left(\mathfrak{g}(M;X) \right)^\vee \right)$ from $H_\ast\left( \CE_{{\mathcal{H}_u}}\left(\mathfrak{g}(M;X) \right)\right)  $ by a universal coefficient spectral sequence 
$$\mathrm{E}_2^{p,q} = \Ext^{p}\left(H_\ast\left( \CE_{{\mathcal{H}_u}}\left(\mathfrak{g}(M;X)\right)\right) , E_\ast \right)_q  \Rightarrow H^{q-p}\left( \CE_{{\mathcal{H}_u}}\left(\mathfrak{g}(M;X) \right)^\vee \right).$$
\end{remark}

\subsection{The Hecke Lie algebras of interest} \label{recipe} In the remainder of the paper, we will use  Theorem \ref{thm:main} as a practical tool for explicit calculations. For this, it is necessary to understand $\mathfrak{g}(M;X)$ as a Hecke Lie algebra. There are three parts to this problem:
\begin{enumerate} 
\item compute $\mathfrak{g}(M;X)$ as an $E_*$-module;
\item calculate the Lie bracket;
\item understand the Hecke operations.
\end{enumerate}

\noindent The answers to $(1)$ and $(2)$  can be obtained from knowledge of $\widetilde E^*(M^+)$ as a nonunital $E^*$-algebra (with the cup product) and $E_*(X)$ as an $E_*$-module.

\begin{proposition}[Lie bracket from cup product]\label{prop:bracket}
If $\widetilde E^*(M^+)$ and $E_*(X)$ are both finite free $E_*$-modules, then \[\U^{\Lie_{\mathcal{H}_u}}_{\Lie_{E_\ast}}\left(\mathfrak{g}(M;X)\right)\cong \widetilde E^*(M^+)\otimes_{E_*} \LL(E_*(\Sigma^{n-1}X))\] with Lie bracket given  by the formula $[\alpha\otimes x, \beta\otimes y]=(-1)^{|x||\beta|}\alpha\beta\otimes [x,y].$
\end{proposition}

\begin{proof}
We begin by briefly recalling  several  well-known facts about algebras over monads in our context; a detailed treatment can be found in \cite[Section 3.4.3]{Lurie:HA}.
Given a spectral Lie algebra $\mathfrak{g}$ and a space $Z$, the $Z$-shaped diagram with constant value $\mathfrak{g}$ in spectral Lie algebras admits a limit. Moreover, the underlying spectrum of this limit is given by  the mapping spectrum $\mathfrak{g}^{Z_+} = \Map_{\Sp}(\Sigma^{\infty}_+Z, \mathfrak{g})$.
Since the limit (in Lie algebras) maps to all its factors via Lie algebra maps, we obtain a $Z^{\lhd}$-shaped diagram of arrows informally given by 
\begin{diagram}
\mathscr{L}(\mathfrak{g}^{Z_+}) & \rTo & \mathfrak{g}^{Z_+} \\
\dDashto & & \dDashto \\
\mathscr{L}(\mathfrak{g} ) & \rTo & \mathfrak{g}
\end{diagram}
Passing to the limit therefore gives rise to a factorization of the structure map of $ \mathfrak{g}^{Z_+}$ as 
\begin{equation} \label{eq:factn}\mathscr{L}(\mathfrak{g}^{Z_+}) \rightarrow \left(\mathscr{L}(\mathfrak{g} )\right)^{Z_+} \rightarrow  \mathfrak{g}^{Z_+}, \end{equation}
where the first map is the natural one defined for every limit, and the second map is obtained from the structure map of $\mathfrak{g}$.
If instead we  consider a  pointed space $\ast \rightarrow Z$, then we define $\mathfrak{g}^Z:= \mathfrak{g}^{Z_+} \times_{\mathfrak{g}^{S^0}} \mathfrak{g}^{\ast} = \fib(\mathfrak{g}^{Z_+} \rightarrow \mathfrak{g})$.
We can then use the previous observation to see that the structure map of $\mathfrak{g}^Z$ is given by $\mathscr{L}(\mathfrak{g}^Z) \rightarrow \left(\mathscr{L}(\mathfrak{g} )\right)^{Z } \rightarrow  \mathfrak{g}^{Z}$, 
where the first map is the natural one defined for every pointed limit, and the second again comes from the structure map of $\mathfrak{g}$.

We now  take $\mathfrak{g} = \mathscr{L}(\Sigma^{n-1}X)$ and $Z=M^+$, the one-point-compactification of $M$.
The first claim follows from Proposition \ref{finitefree} in light of our assumptions, since limits in  spectral Lie algebras are computed in spectra and $\mathscr{L}(\Sigma^{n-1}X)^{M^+}\simeq \DD(M^+)\wedge \mathscr{L}(\Sigma^{n-1}X)$ because $\Sigma^\infty M^+$ is dualisable. 
For the second claim, we need to examine the structure map of $\mathscr{L}(\Sigma^{n-1}X)^{M^+}$ in \mbox{weight $2$.}
Observe that if $Z$ is a finite pointed space and $\mathfrak{g}$ is a spectrum, 
then the canonical map of naive $\Sigma_2$-spectra
$ (\mathfrak{g}^Z)^{\wedge 2} \rightarrow (\mathfrak{g}^{\wedge 2})^Z$
is given by $(\mathfrak{g}^Z)^{\wedge 2} = (\mathfrak{g}\wedge \mathfrak{g})^{Z \wedge Z}  \xrightarrow{\ \ \ \mathfrak{g}^{\delta ^\ast} \ \ \ }  (\mathfrak{g}\wedge \mathfrak{g})^{ Z} = (\mathfrak{g}^{\wedge 2})^Z,$ where $\delta_{ }$ denotes the diagonal map on $Z$.
If $Z$ is finite, then  $\Sigma^{-1}(-)_{h\Sigma_2}: \Sp^{\Sigma_2} \rightarrow \Sp$ is exact, and we  deduce that the canonical map
$ \Sigma^{-1}({\mathfrak{g}}^Z)^{\wedge 2}_{h\Sigma_2} \rightarrow (\Sigma^{-1}{\mathfrak{g}}^{\wedge 2}_{h\Sigma_2})^Z$
is given by the composite 
$$ \Sigma^{-1}({\mathfrak{g}}^{Z})^{\wedge 2}_{h\Sigma_2} =\Sigma^{-1} \left(({\mathfrak{g}}\wedge {\mathfrak{g}})^{Z \wedge Z} \right)_{h\Sigma_2} \xrightarrow{\ \ \ \Sigma^{-1}{\mathfrak{g}}^{\delta^\ast}_{h\Sigma_2}  \ \ \ }  \Sigma^{-1} \left(({\mathfrak{g}}\wedge {\mathfrak{g}})^{Z} \right)_{h\Sigma_2}=(\Sigma^{-1}{\mathfrak{g}}^{\wedge 2}_{h\Sigma_2})^Z.$$

For $\mathfrak{g}= \mathscr{L}(\Sigma^{n-1}X)$ and $Z = M^+$, this is the left part of structure map (\ref{eq:factn}) of  $\mathscr{L}(\Sigma^{n-1}X)^{M^+}$ in weight $2$.
Using again that $M^+$ is dualisable, we can therefore rewrite the structure map of $\mathscr{L}(\Sigma^{n-1}X)^{M^+}$ in weight $2$ as 
$$ 
 \Sigma^{-1} \DD(M^+)^{\wedge 2}_{h\Sigma_2} \wedge  \mathscr{L}(\Sigma^{n-1}X) ^{\wedge 2}_{h\Sigma_2} \xrightarrow{\Sigma^{-1}\DD( \delta^\ast) \wedge \mu} 
\Sigma^{-1} \DD(M^+)  \wedge \mathscr{L}(\Sigma^{n-1}X),$$
where $\mu_2$ denotes the structure map of $\mathscr{L}(\Sigma^{n-1}X)$ in weight $2$.
The second claim now follows by a straightforward diagram chase (keeping careful track of signs).
\end{proof}
We can therefore read off the following statement from \Cref{finitefree} above:
\begin{corollary} \label{cor:finitefree}
If $\widetilde E^*(M^+)$ and $E_*(X)$ are both $E_*$-free and finite, then so is $\mathfrak{g}(M;X)$.
\end{corollary}

In particular, Theorem \ref{thm:main} applies under the    assumptions made in \Cref{prop:bracket}.

In the examples of greatest interest, $X$ is the suspension spectrum of a sphere, so the second condition is satisfied, and it is often the case that the first is as well. In Sections \ref{Euclid} and \ref{surfaces} below, we study two such examples, namely $M=\mathbb{R}^n$ and $M$ a
punctured orientable surface $\mathcal{S}_{g,1}$ . In the former example, $\widetilde E^*((\mathbb{R}^n)^+)\cong \widetilde E^*(S^n)$ is free of rank $1$, and, in the latter example, the standard CW structure splits after suspensions, so that $\widetilde E^*\left(\mathcal{S}_{g,1}^+\right) = \widetilde E^*\left(\mathcal{S}_{g,1}\right)$ \mbox{is free of rank $2g+2$.}

The computation of the Hecke operations on $\mathfrak{g}$ is a more delicate task, but may also be accomplished in many situations of interest.

\begin{lemma}[Stably split manifolds]\label{lem:stably split}  
Let $M, X$ be as in Theorem \ref{thm:main}. Assume that $M^+$ admits the structure of a CW complex whose skeletal filtration splits stably. \mbox{There is an isomorphism} \[\U^{\Lie_{\mathcal{H}_u}}_{\Mod_{{\mathcal{H}_u}}}\left(\mathfrak{g}(M;X)\right)\cong\bigoplus_{r\geq0}\bigoplus_{A_r}\Omega^r\U^{\Lie_{\mathcal{H}_u}}_{\Mod_{{\mathcal{H}_u}}}\Free_{\Mod_{E_\ast}}^{\Lie_{\mathcal{H}_u}}\left(\Sigma^{n-1}E_*(X)\right),\] where $A_r$ denotes the set of $r$-cells of $M^+$. 
\end{lemma}
\begin{proof}
We proceed inductively along the skeletal filtration of $M^+$, using Proposition \ref{prop:hecke desuspension} for the base case. For the inductive step, our assumptions and Proposition \ref{prop:hecke desuspension} imply that the sequence of Hecke Lie modules \[\xymatrix{0\ar[r]&\bigoplus_{A_r}\Omega^r\U^{\Lie_{\mathcal{H}_u}}_{\Mod_{{\mathcal{H}_u}}}\Free_{\Mod_{E_\ast}}^{\Lie_{\mathcal{H}_u}}\left(E_*(X)\right) \ar[d]\\
&\U^{\Lie_{\mathcal{H}_u}}_{\Mod_{{\mathcal{H}_u}}}\left(E_*\left(\Free^{\mathscr{L}}(\Sigma^{n-1}X)^{M^+_{r}}\right)\right)\ar[r]&\U^{\Lie_{\mathcal{H}_u}}_{\Mod_{{\mathcal{H}_u}}}\left(E_*\left(\Free^{\mathscr{L}}(\Sigma^{n-1}X)^{M^+_{r-1}}\right)\right)\ar[r]&0}\] is exact. Since $M$ is of finite type, and since the skeletal filtration of $M^+$ splits stably, the collapse map $M_{r}^+\to \bigvee_{A_r}S^r$ splits after a finite suspension. Thus, the lefthand map in the above short exact sequence splits after applying a finite iteration of $\Omega$ and hence splits outright by Proposition \ref{prop:hecke delooping}. Applying the inductive hypothesis completes the proof.
\end{proof}

\subsection{Inverting the implicit prime}   
We can  
rationalise the $E$-theory  of a given a space $Y$
 in two   ways: first, we can tensor the $E$-cohomology groups with the rationals to obtain $p^{-1}E^\ast(Y)$; second, we can consider its cohomology   $(p^{-1}E)^\ast(Y)$  with respect \mbox{to the rationalisation of  $E$.}
In general, these   procedures will lead to different results. For example, the $p$-adic $K$-theory of $B{\Sigma_p}$ is free on two generators, whereas $H^\ast(B\Sigma_p,\QQ_p[\beta^{\pm 1}])\cong \QQ_p[\beta^{\pm 1}]$ is generated by \mbox{one element.}
However,  they agree, for example, when   $Y$ is a finite  CW  complex (as $p^{-1}E_\ast$ is  \mbox{flat  over $E_*$).}

In this section, we will examine the effect of these two procedures on the cohomological Hecke spectral sequence from \Cref{thm:cohomology}. Working with cohomology instead of homology is simpler in this context, as mapping spectra to $E$ are automatically $K(h)$-local, whereas smash products with $E$ are not. To avoid any confusion concerning the letter ``E", recall \Cref{confusion}.  
 
Assume  that $M$ is a stably split, framed  $n$-manifold such that $\widetilde{E}^\ast(M^+)$ is a finite   free $E_*$-module, and let $X=S^a$.  
The spectral sequence $\{\mathrm{E}_r(M;S^a)\}$ in \Cref{thm:cohomology} has signature
 \begin{equation} \label{ESpS} {\mathrm{E}_2^{s,t}}  = H^s\left(  \mathrm{Bar}_\bullet \left(\id,\, \mathbf{L}^{{\mathcal{H}_u}}  ,\,\mathfrak{g}(M;S^a)\right)^\vee  \right)_{t+1} \Longrightarrow \ \ \ \bigoplus_k  E^{t-s}   (B_k(M;S^a)). \end{equation} 
Let $\{p^{-1} \mathrm{E}_r(M;S^a)\}$ be the spectral sequence obtained from $\{\mathrm{E}_r(M;S^a)\}$ by inverting the \mbox{prime $p$.}   The main purpose  of this section is to establish the following helpful computational tool, which will play a role in 
\Cref{surfaces}:
\begin{theorem}\label{thm:free part}
Under the above assumptions,  $\{p^{-1} \mathrm{E}_r(M;S^a)\}$ collapses at the $\mathrm{E}_2$-page.  In other words, every  non-trivial differential $d_r$  in $\{\mathrm{E}_r(M;S^a)\}$
 with $r\geq 2$   has $p$-power \mbox{torsion target.} 
\end{theorem}

To prove this theorem, we will consider a variant of   \eqref{ESpS}  for the cohomology theory $(p^{-1}E)^\ast(-)$.
Indeed, write $\mathbf{L}^{p^{-1}E_*}$ for  the monad on 
$\Mod_{p^{-1}E_*}^{\NN}$ parametrising Lie algebras in the sense of \Cref{def:Lie algebra}. 
Since
$p^{-1}E\simeq  p^{-1}(\mathrm{W}(k)[[u_1,\ldots,u_{n-1}]])[u_n^{\pm 1}]$  is a form of rational cohomology, 
 a much   easier variant of \cite[Theorem 4.4.4]{brantnerthesis} asserts the existence of a  natural isomorphism
$$\mathbf{L}^{p^{-1}E_*} (\pi_\ast(Y)) \rightarrow \pi_\ast\left(\Free^{\mathscr{L}_{p^{-1}E}} Y\right)  $$
  for any $p^{-1}E$-module spectrum $Y$.
The same argument as in 
\Cref{thm:cohomology}  then shows that there is a completely convergent  weighted spectral sequence $\widehat{\mathrm{E}} = \{\widehat{\mathrm{E}}_r(M;S^a)\}$ with signature
$$\widehat{\mathrm{E}}_2^{s,t} = H^s\left(  \mathrm{Bar}_\bullet (\id,\, \mathbf{L}^{p^{-1}E_*} ,\,\mathfrak{h}(M;S^a)_{p^{-1}})^\vee  \right)_{t+1}  \Longrightarrow \ \ \ \bigoplus_k \  (\small{p^{-1}}E)^{t-s}   (B_k(M;S^a))$$ 
Here 
$\mathfrak{h}(M;S^a)_{p^{-1}}$ denotes the Lie algebra $\left(p^{-1}E\right)_* (\Free^{\mathscr{L}}(\Sigma^{n-1}S^a)^{M^+} )$.
A rational variant of the argument for \Cref{prop:bracket} allows us to identify this Lie algebra as 
$$\mathfrak{h}(M;S^a)_{p^{-1}}\cong \widetilde{(p^{-1} E)}^\ast(M^+) \myotimes{p^{-1} E_*} \mathbf{L}^{p^{-1} E_*}\left((p^{-1} E)_*(S^a)\right),$$
where the Lie bracket on the right hand side is again defined in terms of the cup product.
Since $M^+$ is a finite CW complex,  the canonical map 
$p^{-1} (\widetilde{E}^\ast(M^+) \myotimes{ E_*} \mathbf{L}^{E_*}\left( E_*(S^a)\right) )\rightarrow \mathfrak{h}(M;S^a)_{p^{-1}}$ is an isomorphism. Hence, we  have   identified an integral form $\mathfrak{h}(M;S^a)\in \Lie_{E_*}$ of $\mathfrak{h}(M;S^a)_{p^{-1}}$.

Sending the simplicial weighted spectrum $\mathrm{Bar}_\bullet (\id,\, \mathscr{L},\,  \Free^{\mathscr{L}}(\Sigma^{n-1}S^a)^{M^+}  )  )$  into the map   $E\rightarrow  p^{-1}E$, we obtain a canonical map of spectral sequences $$   \phi_r:  p^{-1}\mathrm{E}_r(M;S^a)\to \widehat{\mathrm{E}}_r(M;S^a).$$
A key step in the proof of \Cref{thm:free part} will be to show that these two rationalisations agree:
\begin{proposition}\label{prop:ss comparison}
Under the above assumptions, the map $\phi_r$    \mbox{ is an isomorphism for all $r\geq2$.}  
\end{proposition}
\begin{proof} In this proof, we will often omit forgetful functors from \vspace{2pt} the \mbox{notation to increase readability.}
We   first   need to explicitly identify the map $\phi_1$   on    $\mathrm{E}_1$-pages.
To this end, consider the following composite of maps of cosimplicial modules: 
\begin{align*}
\hspace{-10pt} p^{-1}  \mathrm{Bar}_\bullet \left(\id,\, \mathbf{L}^{{\mathcal{H}_u}}  ,\,\mathfrak{g}(M;S^a)\right)^\vee  &\xrightarrow{\phi_1} (\mathrm{Bar}_\bullet (\id,\, \mathbf{L}^{p^{-1}E_*} ,\,\mathfrak{h}(M;S^a)_{p^{-1}}) )^\vee\\
 & \xrightarrow{\cong} \left(p^{-1}\mathrm{Bar}_\bullet (\id,\, \mathbf{L}^{ E_*} ,\,\mathfrak{h}(M;S^a))\right)^\vee.
\end{align*}
The second map identifies the Bar construction of $\mathfrak{h}(M;S^a)_{p^{-1}}$ with the rationalised Bar construction of its integral form $\mathfrak{h}(M;S^a)$.

Unravelling the computation of the $E$-cohomology of free spectral Lie algebras, we see that 
the map $\phi_1$ kills all Hecke operations of higher weight. We deduce   that the above composite is obtained by applying the functor  $(p^{-1}(-))^\vee$
to the following natural inclusion:
\begin{equation} \label{bartrans}\mathrm{Bar}_\bullet (\id,\, \mathbf{L}^{ E_*} ,\,\mathfrak{h}(M;S^a)) \rightarrow 
  \mathrm{Bar}_\bullet \left(\id,\, \mathbf{L}^{{\mathcal{H}_u}}  ,\,\mathfrak{g}(M;S^a)\right)\end{equation}
To prove \Cref{prop:ss comparison}, it   suffices to show that \eqref{bartrans} is a weak equivalence after \mbox{inverting $p$.} 
The map \eqref{bartrans} is obtained by applying the functor  $Q_{\Lie_{E_*}}^{\Mod_{E_*}}$ from \Cref{indec}
to the natural inclusion 
\begin{equation} \label{six}\mathfrak{h}(M;S^a) \xrightarrow{\simeq}  \mathrm{Bar}_\bullet (\Free_{\Mod_{E_*}}^{\Lie_{E_*}},\, \mathbf{L}^{ E_*} ,\,\mathfrak{h}(M;S^a)) \rightarrow 
  Q_{\Lie_{{\mathcal{H}_u}}}^{\Lie_{E_*}} \mathrm{Bar}_\bullet \left(\Free_{\Mod_{E_*}}^{\Lie_{\mathcal{H}_u}}
,\, \mathbf{L}^{\mathcal{H}_u} ,\,\mathfrak{g}(M;S^a) \right)\end{equation}
Since  
$p^{-1} (Q_{\Lie_{E_*}}^{\Mod_{E_*}}(\mathfrak{g})) \cong Q_{\Lie_{p^{-1}E_*}}^{\Mod_{p^{-1}E_*}}(p^{-1}\mathfrak{g}) $ (as the corresponding square of right adjoints evidently commutes),  
 it suffices to check that \eqref{six}   is a weak equivalence after inverting $p$. It suffices to verify this on modules, and since  \mbox{$U_{\Lie_{E_*}}^{\Mod_{E_*}} \circ Q_{\Lie_{{\mathcal{H}_u}}}^{\Lie_{E_*}} \cong  Q_{\Mod_{{\mathcal{H}_u}}}^{\Mod_{E_*}}\circ U_{\Lie_{\mathcal{H}_u}}^{\Mod_{\mathcal{H}_u}}$}, proceeding as in the proof of \Cref{prop:additive resolution works}
allows us to further reduce to the claim that the following map is a weak equivalence after inverting $p$:
\begin{equation} \label{seven}\mathfrak{h}(M;S^a) \rightarrow \Barr_\bullet (\id,\AAA,\U_{\Mod^{\mathcal{H}_u}}^{\Lie_{{\mathcal{H}_u}}}( \mathfrak{g}(M;S^a))).\end{equation}

Here, $\AAA$ is the monad introduced in \Cref{AdditiveMonadAAA}. Since $M$ is stably split, we may by \Cref{lem:stably split} assume without restriction that $M=\RR^n$. In this case, the above map of modules  is given by
\begin{equation}\Omega^n \mathbf{L}^{ E_*} (x_{n+a-1}) \rightarrow \Barr_\bullet (\id,\AAA,\Omega^n \mathbf{L}^{\mathcal{H}_u} (x_{n+a-1})),\end{equation} 
On the left, $\Omega^n \mathbf{L}^{ E_*} (x_{n+a-1})$ is simply the $(-n)^{th}$ shift of $\mathbf{L}^{ E_*} (x_{n+a-1})$. On the right, we have used the looping functor $\Omega^n$ on Hecke modules from \Cref{loopingmodules}.
We now observe a commuting diagram  
\begin{diagram} 
\Omega^n \mathbf{L}^{ E_*} (x_{n+a-1}) & \rTo  & \Barr_\bullet (\id,\AAA,\Omega^n \mathbf{L}^{\mathcal{H}_u} (x_{n+a-1}))& \rTo  & \Omega^n \mathbf{L}^{\mathcal{H}_u} (x_{n+a-1})\\
\dTo^\cong & & \dTo^{\Barr(\id, \Susp^a, - )}& & \dTo^\cong \\
\mathbf{L}^{ E_*} (x_{n+a-1})[-n] & \rTo  & \Barr_\bullet (\id,\AAA, \mathbf{L}^{\mathcal{H}_u} (x_{n+a-1}))[-n]& \rTo  &  \mathbf{L}^{\mathcal{H}_u} (x_{n+a-1})[-n]
\end{diagram}
The middle  map sends a typical element $[\alpha_1 | \ldots |  \alpha_r| \ y \ ]$ to $[\Susp^a(\alpha_1) | \ldots |  \Susp^a(\alpha_r)|\  y\ ]$, and the lower composite is the identity.

Since the suspension morphisms $ ({\mathcal{H}_u^{\Lie}})_i^j(p^s)    \xrightarrow{\Susp^a}   ({\mathcal{H}_u^{\Lie}})_{i+a}^{j+a}(p^s) $ are rational isomorphisms for all $i,j$,   it suffices to prove that the map 
$$ \mathbf{L}^{E_\ast} (x_{n+a-1})   \rightarrow \Barr_\bullet (\id,\AAA, \mathbf{L}^{\mathcal{H}_u} (x_{n+a-1})) $$ 
is a rational equivalence,  which holds as  $\mathbf{L}^{\mathcal{H}_u} (x_{n+a-1}) \cong \AAA \circ \mathbf{L}^{E_\ast} (x_{n+a-1})$.\end{proof}

In order to conclude \Cref{thm:free part}, we require one further result.

\begin{lemma}\label{lem:collapse}  
Under the above assumptions,  $\{\widehat{\mathrm{E}}_r(M;S^a)\}$ degenerates at the $\widehat{\mathrm{E}}_2$-page.
\end{lemma}
\begin{proof}
Write $\{\overline{\mathrm{E}}_r(M;S^a)\}$ for the spectral sequence obtained by mapping the weighted simplicial spectrum
$\mathrm{Bar}_\bullet (\id,\, \mathscr{L},\,  \Free^{\mathscr{L}}(\Sigma^{n-1}S^a)^{M^+}  )  )$ into the rational
Eilenberg--MacLane spectrum  of  $\QQ$.
There is an evident map of spectral sequences
$$(p^{-1}E_\ast) \otimes_{\QQ} \overline{\mathrm{E}}_r(M;S^a) \rightarrow \widehat{\mathrm{E}}_r(M;S^a),$$
and this map is an isomorphism as $p^{-1}E $ is a rational spectrum.

It therefore suffices to prove that the $ \QQ$-based spectral sequence $ \overline{\mathrm{E}}_r(M;S^a) $ degenerates. This follows from  the main result of \cite{Knudsen:BNSCSVFH}, which shows that the rational cohomology of the weighted spectrum $\bigoplus_k B_k(M;S^a)$ agrees with the  $ \overline{\mathrm{E}}_2 $-page of this spectral sequence. Thus, no further differentials can occur.
\end{proof}

\begin{remark}
For general $X$, higher differentials do occur. These differentials are ``tensored up'' from differentials that occur rationally, which are related to Massey products in $\widetilde H^*(M^+;\mathbb{Q})$. In particular, the conclusion of Lemma \ref{lem:collapse} holds for  general $X$ under the assumption that $M^+$ is rationally formal.
\end{remark}

\section{Configurations of $p$ points in $\RR^n$}\label{Euclid}
Given a  Euclidean space $\RR^n$ and an \mbox{integer $k$,} we consider the following spectrum \label{removed "fix E-theory" comment as done again below}
$$B_p(\RR^n; S^k)= \Sigma_+^{\infty} \Conf_p(\RR^n) \otimes_{h\Sigma_p} (S^k)^{\otimes p},$$ which is the weight $p$ component in the  free $\EE_n$-algebra ${\EE_n}(S^k) = B(\RR^n;S^k)$ on a  single generator in degree $k$. If $k$ is nonnegative, then Snaith's theorem shows that ${\EE_n}(S^k)$ is  equivalent to
$\Sigma^\infty_+ \Omega^n  S^{n+k}$.
In this section, we will use the methods introduced above to compute the $E$-theory of these spaces for all heights $h$ and all dimensions $n$.
\subsection{Warmup: The $K$-theory of $p$ points in $\RR^n$.}\label{warmupsection}
We shall begin at height $1$, where calculations are notationally simpler, geometrically most  significant, and an illustrative template for the general height calculations in Section \ref{section:higher heights} below. 

We fix a specific height one Morava $E$-theory $K$ with $E_*=K_{*}=\mathbb{Z}_p[u^{\pm1}]$, namely  $p$-completed complex  $K$-theory. In this section, we will prove:

\begin{theorem}[$K$-theory of configurations in Euclidean space] \label{warmup}
Given $n\geq 1$ and $k\in \ZZ$, there are isomorphisms of  $K_*$-modules  
$$ K_*^{\wedge} \left(\Conf_p(\RR^n)_+ \otimes_{h\Sigma_p}(S^k)^{\otimes p} \right) \ \  \cong \ \    \begin{cases}
       \Sigma^{kp}K_\ast     \oplus \Sigma^{pk+n-1} K_*    \oplus \Sigma^{k-1} K_\ast/p^{ \frac{n }{2} -1 }     & \mbox{for } n \mbox{   even, } k \mbox{   even}\\
 \Sigma^{k-1}K_\ast/ p^{ \frac{n}{2} } & \mbox{for } n \mbox{ even, } k \mbox{  odd}\\
 \Sigma^{kp}K_\ast \oplus \Sigma^{k-1} K_\ast/p^{ \frac{n-1}{2}  } & \mbox{for } n \mbox{  odd, } k \mbox{  even}\\
 \Sigma^{k+ (2k+n-1)(\frac{p-1}{2})} K_\ast \oplus \Sigma^{k-1} K_*/p^{  \frac{n-1}{2} }& \mbox{for } n  \mbox{  odd, } k \mbox{  odd}
\end{cases} $$ 
  
\end{theorem}
\begin{remark} 
As a sanity check,  observe that the torsion-free components take the expected form. Indeed, $ K_*\left({\EE_n}(S^k)\right)$ forms a Poisson algebra with commutative product $\cdot$ and Lie bracket $[-,-]$ of degree $n-1$ satisfying $[x,x] = (-1)^{|x|+n}[x,x]$.

Heuristically, we therefore expect to see the following torsion-free classes:\vspace{-10pt}
\end{remark}
 
\begin{figure}[H]
\begin{center}\def\arraystretch{1.5}
\begin{small}
\begin{tabular}{c|cccc} 
$k\backslash n$ & \ \ even & odd & \\
\hline 
even \ \ \ \  &$  \ \  \ \   [x_k,x_k]\cdot x_k^{p-2} \ , \   x_k^p$  \ \ \ \ \ &$x_k^p\vspace{-15pt}  $  
\\\\
odd \ \ \ \   &$   $ &  $ [x_k,x_k]^{\frac{p-1}{2}} \cdot x_k$
\end{tabular}
\end{small}
\end{center}  
\end{figure} \vspace{-10pt}

To prove \Cref{warmup}, we introduce the following very simple Hecke Lie algebras: 
\begin{definition}[Atomic Hecke Lie algebras at height $1$]\label{atomicalgebra}
Given positive integers $n, w$ and an integer $a$, the atomic Hecke Lie algebra $\mathfrak{g}_{a,n}^{(w)}$ is defined as follows:
\begin{enumerate} 
\item The underlying $K_*$-module of $\mathfrak{g}$ is free on two generators:
\begin{itemize} \item a generator $x$ in internal degree $a$ and weight $w$ 
\item a generator $y_{a-1}$ in internal degree $a-1$ and weight $pw$. \end{itemize}
\item The  Hecke operations are determined  by    (cf. \Cref{Lieheightone}):
$$\alpha x =   \alpha_a x  = \begin{cases} 
p^{ \lfloor \frac{n}{2} \rfloor} y  & \mbox{ if } a \mbox{ is even}\\
   p^{ \lceil \frac{n}{2} \rceil} y     & \mbox{ if } a \mbox{ is odd}    \end{cases}$$
\item All Lie brackets vanish.
\end{enumerate} 
\end{definition}

In the following sections, we will use small subscripts $(-)_{(i,j,r,w)}$ to  indicate the quadruple grading\footnote{We grade the Hecke Chevalley--Eilenberg complex by  $\mbox{(internal degree, homological degree, simplicial degree, weight).}$}  on the Hecke Chevalley--Eilenberg  complex, which \mbox{was defined in \Cref{quadruple} above.}
Our computation will rely on knowing the homology of the  atomic algebras from \Cref{atomicalgebra}:

\begin{lemma}[Homology of atomic Hecke Lie algebras]\label{atomiclemma}
In weights $k\leq pw$,  
\mbox{we have  isomorphisms:}  
$$\mbox{For $a$ even: }\HH^{\Lie^{\mathcal{H}_u}}(\mathfrak{g}_{a,n}^{(w)})(k) 
 \cong ( \Lambda_{K_\ast}[x_{(a,1,0,w)}] \oplus K_\ast/p^{\lfloor \frac{n}{2} \rfloor} [  y_{(a-1,1,0,pw)} ]) (k).$$
$$\mbox{For $a$ odd: } \    \HH^{\Lie^{\mathcal{H}_u}}(\mathfrak{g}_{a,n}^{(w)})(k) \cong (\Gamma_{K_\ast}[x_{(a,1,0,w)}]  \oplus  K_\ast/p^{\lceil \frac{n}{2} \rceil} [y_{(a-1,1,0,pw)} ] ) (k) .\vspace{5pt}$$
Here $\Lambda_{K_\ast}$ constructs the free exterior    and $\Gamma_{K_\ast}$ the free divided $K_\ast$-algebra. \vspace{-3pt}
\end{lemma} 
\begin{proof}[Proof of \Cref{atomiclemma}]
We depict the additive resolution 
  $\AR(\mathfrak{g}_{a,n}^{(w)})$ of \Cref{construction:additive resolution} in the following diagram:\vspace{-10pt}
\begin{figure}[H]
\begin{small}
\begin{center}\def\arraystretch{1.5}
\begin{tabular}{c|cccc}
$r\backslash weight$ &$1w$ & $pw$ & $p^2w$&$\cdots$\\
\hline
$0$&$\substack{\left[x\right]_{_{(a,1,0,1w)}}}$ & \ \ \ \ \ $\substack{\left[y\right]_{_{(a-1,1,0,pw)}}}$ & $0$  \ \ \ \ \ &$\cdots$\ \vspace{-9pt}
\\\\
$1$&$\substack{\left[1|x\right]_{_{(a,1,1,1w)}}}$ & \ \ \ \ \  $\substack{\left[ \alpha | x \right]_{_{(a-1,1,1,pw)}}\\ \left[ 1 | y \right]_{_{(a-1,1,1,pw)}}}$  &$\substack{ \ \ \ \ \ \left[\alpha|y\right]_{_{(a-2,1,1,p^2w)}}}$&$\cdots$ \vspace{-5pt}
\\\\
$2$&$\substack{\left[1|1|x\right]_{_{(a,1,2,1w)}}}$ &  \ \ \ \ \    $\substack{  \left[ \alpha | 1 | x \right]_{_{(a-1,1,2,pw)}} \\\left[1 | \alpha | x \right]_{_{(a-1,1,2,pw)}} \\ \left[ 1 | 1 | y \right]_{_{(a-1,1,2,pw)}}}$  &$\cdots$&$\cdots$\\
\vdots&\vdots&\vdots&\vdots
\end{tabular}\vspace{-10pt}
\end{center} 
\end{small}
\end{figure} \ \vspace{-5pt}
 
The second index records 
 the homological degree 
 in the simplicial chain complex $\CE(\AR(\mathfrak{g}_{a,n}^{(w)}))$. The shift by $+1$ 
  in the Chevalley--Eilenberg complex is reflected by the second entries  
being $1$ rather than $0$.  
By \Cref{thm:hecke chevalley--eilenberg works},  $\HH^{\Lie^{\mathcal{H}_u}}(\mathfrak{g}_{a,n}^{(w)})$ is the homology \mbox{of the     complex}
$$\CE(\AR(\mathfrak{g}_{a,n}^{(w)})) =  \Gamma_{K_\ast}(\AR(\mathfrak{g}_{a,n}^{(w)})[1]).$$

For $a$  even,   $[x]$ has odd total degree in the chain complex  $\AR_0(\mathfrak{g}_{a,n}^{(w)})[1]$. Since the tensor product of chain complexes of $K_\ast$-modules incorporates the Koszul sign rule for both internal and chain  degree, this means that $[x]^i$ vanishes in $\CE(\AR(\mathfrak{g}_{a,n}^{(w)})) $ when $i>1$. The normalised chain complex of $\CE(\AR(\mathfrak{g}_{a,n}^{(w)}))$ is  therefore given by $$ \ldots \rightarrow 0\rightarrow 0\rightarrow 0 \rightarrow 0 \rightarrow [1] \ \mbox{ in weight } 0,$$
$$ \ldots  \rightarrow 0 \rightarrow 0\rightarrow 0 \rightarrow [x]\rightarrow 0  \ \mbox{ in weight } w,$$
$$ \ \  \ \ \ \     \ldots  \rightarrow 0 \rightarrow [\alpha|x] \xrightarrow{p^{\lfloor\frac{n}{2} \rfloor}} [y]\rightarrow 0 \ \mbox{ in weight } pw, \ \ \ \ \  $$
and vanishes for all other smaller weights.

For $a$ odd, a similar argument shows that the normalised chain complex of $\CE(\AR(\mathfrak{g}_{a,n}^{(w)}))$  
is  
$$\ \  \ \  \ \   \ \  \ \ \ \   \ldots \rightarrow 0\rightarrow 0\rightarrow 0\rightarrow 0 \rightarrow   0    \xrightarrow{  }  \ [x]^i   \ \rightarrow 0 \ \rightarrow  \ldots \ \rightarrow 0 \mbox{ \ \  \ in weight } iw  \mbox{ for } i=0, 1, \ldots,  (p-1), \mbox{\    \ \  \ \ \  } \ \ \ \ \ \vspace{-2pt}  $$ 
$$   \ \   \ \ \ \  \ \ \ldots  \rightarrow \gamma_p([x]) \rightarrow  0 \rightarrow \ldots  \rightarrow 0 \rightarrow [\alpha|x] \xrightarrow{p^{\lceil \frac{n}{2} \rceil}}  [y] \rightarrow 0\ \   \ \mbox{ in weight }pw,   \ \ \ \  \ \  \ \  \ \ \ \   \   \ \  \ \ \ \ \ \ \ \  \ \ \  \ \ \ \ \ \vspace{2pt}  $$
and vanishes in all other weights below $pw$. Here $\gamma_p([x])$ denotes the $p^{th}$ divided power of $[x]$.
\end{proof}

We recall  the explicit structure of some free Hecke Lie algebras from \cite[Section 4.4.2]{brantnerthesis}:
\begin{corollary}[Free Hecke Lie algebras on one generator at height $1$] \label{freelieone}
If $x_a$ is a generator in even degree $a$, then $\LL(x_a)$  is generated as a $K_\ast$-module by  $x\in \LL(x_a)_{a}$ and $y\in \LL(x_a)_{a-1}.$
The Lie bracket vanishes, and the Hecke operations are determined by  $\alpha_a(x)=y$ and $\alpha(y)=0$.

If $x_a$ is a generator in odd degree $a$, then  $\LL(x_a)$  is generated as $K_\ast$-module by four classes
$$x\in \LL(x_a)_{a} \ , \ \ y\in \LL(x_a)_{a-1} \ , \ \
\widetilde{x}\in \LL(x_i)_{2a} \ , \ \ \widetilde{y}\in \LL(x_i)_{2a-1}$$
The Lie bracket satisfies $[x,x]=\widetilde{x}$ and vanishes otherwise; the Hecke operations \mbox{are determined by} $$\alpha_a(x)=y \ , \ \  \alpha_{2a}(\widetilde{x})=\widetilde{y}\ , \ \ \alpha_{a-1}(y)=\alpha_{2a-1}({\widetilde{y}})=0.$$

\end{corollary}
 
With the homology of atomic Lie algebras at hand, we can carry out the desired computation.
Our proof will proceed in three steps: 
\begin{enumerate}
\item   Identify the Hecke Lie algebra $\mathfrak{g}(\RR^n, S^k) 
:=K_*^{\wedge} (\Omega^n\Free_{\Lie}(\Sigma^{n+k-1})).$
\item Compute the Hecke   Lie algebra homology of $\mathfrak{g}(\RR^n, S^k) 
:=K_*^{\wedge} (\Omega^n\Free_{\Lie}(\Sigma^{n+k-1})).$
\item Compute the  $\mathrm{E}^{\infty}$-page of the spectral sequence  
and solve extension problems.
\end{enumerate}

In order to apply \Cref{thm:main}, we start with the following observation:
\begin{proposition}\label{step1}
The Hecke Lie algebra  $\mathfrak{g}(\RR^n, S^k)$ is determined as follows in terms of the atomic Hecke Lie algebras from \Cref{atomicalgebra}:
\begin{figure}[H]
\begin{small}
\begin{center}\def\arraystretch{1.5}
\begin{tabular}{c|cccc} 
$k\backslash n$ & \ \ even & odd & \\
\hline 
even \ \ \ \  &$ \ \ \ \  \mathfrak{g}^{(1)}_{k-1,n} \oplus  \mathfrak{g}^{(2)}_{n+2k-2,\lfloor \frac{n}{2}\rfloor} $ & $\mathfrak{g}^{(1)}_{k-1, n} $  
\\
odd \ \ \ \   &$ \ \ \ \  \mathfrak{g}^{(1)}_{k-1, n}$ &  $\mathfrak{g}^{(1)}_{k-1,n} \oplus  \mathfrak{g}_{n+2k-2,n}^{(2)}  $
\end{tabular}
\end{center}
\end{small} 
\end{figure}
\end{proposition}
\begin{proof}
By \Cref{prop:hecke desuspension}, $\mathfrak{g}(\RR^n, S^k) $ is obtained by 
first taking the  free Hecke Lie algebra $\LL(x_{n+k-1})$ on a class in degree ${n+k-1}$,
 then desuspending $n$ times in internal degree (which kills the Lie bracket), 
and then acting through $n$-fold suspended Hecke operations. We will freely use \Cref{freelieone} in our arguments below.\\
 \\
If  $n+k-1$ is even, then $\LL(x_{n+k-1})$ is generated (as a $K_\ast$-module)  by a class $u$ in internal degree $n+k-1$ and a class  $v=\alpha_{n+k-1}(u)$ in internal degree $n+k-2$. 
Hence $\mathfrak{g}(\RR^n, S^k)= \Sigma^{-n} \LL(x_{n+k-1})$ is generated by classes $x=\Sigma^{-n}u$ and $y=\Sigma^{-n}v$ in internal degree $k-1$ and $k-2$, respectively.
To act with $\alpha_{k-1}$ on $x$, we can (by \Cref{prop:hecke desuspension}) instead act by $\Sigma^n(\alpha_{k-1})$
on $u$ and then apply $\Sigma^{-n}$ the resulting class.

If $k-1$ is even, then \Cref{Gamma0} and the identification of the Euler class $e$ with  $p$ implies:
$$ \alpha_{k-1}(x) = \Sigma^{-n} (\Susp^n(\alpha_{k-1}) (u))  = \Sigma^{-n} (p^{\lfloor \frac{n}{2}\rfloor}\alpha_{n+k-1} (u))  = p^{\lfloor \frac{n}{2}\rfloor} y$$

 \ \ As the action on $y$ vanishes,   $\mathfrak{g}(\RR^n, S^k) $ is equal to the atomic \mbox{Lie algebra $\mathfrak{g}^{(1)}_{k-1, n}$.}\vspace{3pt} 

If $k-1$ is odd, a similar argument shows that $$ \alpha_{k-1}(x) = \Sigma^{-n} (\Susp^n(\alpha_{k-1}) (u))  = \Sigma^{-n} (p^{\lceil \frac{n}{2}\rceil }\alpha_{n+k-1} (u))  = p^{\lceil\frac{n}{2}\rceil} y$$

\ \ which implies that $\mathfrak{g}(\RR^n, S^k) $ is equal to the atomic Lie algebra $\mathfrak{g}^{(1)}_{k-1, n}$.\vspace{-3pt} 
 \\
\\
If $n+k-1$ is odd, then $\LL(x_{n+k-1})$ is generated (as a  $K_\ast$-module)  by four classes:
\begin{itemize}
\item  $u$ in internal degree $n+k-1$ and weight $1$;
\item $v=\alpha_{n+k-1}(u)$ in internal degree $n+k-2$ and weight $p$;
\item $u' = [u,u]$ in internal degree $2n+2k-2$ and weight $2$;
\item $v' = \alpha_{2n+2k-2}[u,u]$ in internal degree $2n+2k-3$ and weight $2p$.
\end{itemize}
Hence $\mathfrak{g}(\RR^n, S^k)$ is generated by the four classes $x = \Sigma^{-n}u $, $ y= \Sigma^{-n}v$, $x' = \Sigma^{-n}u' $, and $y' = \Sigma^{-n}v'$
in internal degrees $k-1$, \ $k-2$, \ $n+2k-2$, \  and $n+2k-3$,\ respectively.

If  $k-1$ is even, then $n+2k-2$ is odd, and so a similar argument as before shows that we \\ \ $_{}$\ \ \ \ \  \  have 
$ \alpha_{k-1}(x) =  p^{\lfloor \frac{n}{2}\rfloor} y \mbox{ and } \alpha_{n+2k-2}(x) =  p^{\lceil  \frac{n}{2}\rceil} y'$.
Hence $\mathfrak{g}(\RR^n, S^k) = \mathfrak{g}^{(1)}_{k-1,n} \oplus  \mathfrak{g}^{(2)}_{n+2k-2,n} $.\vspace{3pt} 

If  $k-1$ is odd,  
$ \alpha_{k-1}(x) =  p^{\lceil \frac{n}{2}\rceil} y$ and $\alpha_{n+2k-2}(x) =  p^{\lfloor  \frac{n}{2}\rfloor} y'$, so
 {$\mathfrak{g}(\RR^n, S^k) = \mathfrak{g}^{(1)}_{k-1,n} \oplus  \mathfrak{g}^{(2)}_{n+2k-2,n} $}.
\end{proof}
Combining this with  our previous work, we can immediately deduce:
\begin{proposition} \label{step2}
The homology of  $\mathfrak{g}(\RR^n, S^k)$ in weight $p$  is generated by the following elements:
\begin{figure}[H]
\begin{small}
\begin{center}\def\arraystretch{1.5}
\begin{tabular}{c|cccc} 
$k\backslash n$ & \ \ even & odd & \\
\hline
even \ \ \ \  &$ \ \ \ \ \substack{ \  \\ \ \\ \gamma_p([x_{(k-1,1,0,1)}])\\  \\  [y_{(k-2,1,0,p)}]/p^{\lceil \frac{n}{2}\rceil}\ \vspace{2pt} \\ \\
 [x_{(k-1,1,0,1)}]^{p-2}\cdot [x'_{(n+2k-2, 1, 0 ,2)}]
} $  \ \ \ \ \ &$\substack{ \gamma_p([x_{(k-1,1,0,1)}]) \\ \\  [y_{(k-2,1,0,p)}]/p^{ \lceil \frac{n}{2}\rceil} }$  
\\\\
odd \ \ \ \   &$ \ \ \ \  \substack{ [y_{(k-2,1,0,p)}]/p^{ \lfloor \frac{n}{2}\rfloor}}$ &  $\substack{[x_{(k-1,1,0,1)}] \cdot [x'_{(n+2k-2,1,0,2)}]^{\frac{p-1}{2}}\\ \\
[y_{(k-2,1,0,p)}]/p^{ \lfloor \frac{n}{2}\rfloor}
}$
\end{tabular}
\end{center}
\end{small} 
\end{figure} 
\end{proposition}
\begin{proof}
This follows from \Cref{step1}, \Cref{atomiclemma} since
$\HH^{\Lie^{\mathcal{H}_u}}(\mathfrak{g} \oplus \mathfrak{g}')(w)$ is isomorphic to $  (\HH^{\Lie^{\mathcal{H}_u}}(\mathfrak{g}) \otimes \HH^{\Lie^{\mathcal{H}_u}}(  \mathfrak{g}') )(w)$ for all involved Hecke Lie algebras $\mathfrak{g}, \mathfrak{g}'$ and all weights $w\leq p$.
\end{proof}

\begin{remark}Before proceeding to the proof \Cref{warmup},  
we note that elements $[z_{(i,j,r,w)}]$ in internal degree $i$, homological degree $j$, simplicial degree $r$, and weight $w$ (cf. \Cref{quadruple})
contribute to the 
$\mathrm{E}^2_{j+r-1, i+1} = \HH^{\Lie^{\mathcal{H}_u}}_{j+r}( \mathfrak{g}(\RR^n, S^k))_{i}$ \ term of the weight $w$ component of our spectral sequence in \Cref{thm:main}. 
\end{remark}

\begin{proof}[Proof of \Cref{warmup}] 
 By  \Cref{step2}, the $\mathrm{E}^2$-term  $\mathrm{E}^2_{s, t} = \HH^{\Lie^{\mathcal{H}_u}}_{s+1}( \mathfrak{g}(\RR^n, S^k))_{t-1}$ 
of the spectral sequence  \Cref{thm:main} at weight $p$ \mbox{is  given by:}   
 $$\ \ \ \ \ \mathrm{E}^2_{0, \ast} (p) \  \ \ \ = \ \  \ \ \  \HH^{\Lie^{\mathcal{H}_u}}_{1}( \mathfrak{g}(\RR^n, S^k))(p))_{\ast-1}  \    \ \ \ = \ \ \ \ \  \begin{cases}
  \Sigma^{k-1}K_\ast/p^{\lceil \frac{n}{2}\rceil}& \mbox{for } n \mbox{ even, } k \mbox{ even}\\
 \Sigma^{k-1}K_\ast/ p^{ \lfloor \frac{n}{2}\rfloor} & \mbox{for } n \mbox{  even, } k \mbox{  odd}\\
   \Sigma^{k-1} K_{\ast}/ p^{ \lceil \frac{n}{2}\rceil}& \mbox{for } n \mbox{  odd, } k \mbox{  even}\\
   \Sigma^{k-1} K_\ast/p^{\lfloor \frac{n}{2} \rfloor}& \mbox{for } n \mbox{ odd, } k \mbox{   odd}
\end{cases} \ \ \ \ \ \  \ \ \ \ \ \  \ \ \ \ \ \  \ \ \ \ \ \ $$
  \vspace{-10pt}
 
$$\mathrm{E}^2_{\frac{p-1}{2}, \ast-\frac{p-1}{2}} (p)= \HH^{\Lie^{\mathcal{H}_u}}_{\frac{p+1}{2}}( \mathfrak{g}(\RR^n, S^k))(p))_{\ast-\frac{p+1}{2}}  =  \begin{cases}
0 & \mbox{for } n \mbox{   even, } k \mbox{   even}\\
  0 & \mbox{for } n \mbox{ even, } k \mbox{  odd}\\
0  & \mbox{for } n \mbox{  odd, } k \mbox{  even}\\
 \Sigma^{(2k+n-1)(\frac{p-1}{2})+k}K_\ast   & \mbox{for } n \mbox{  odd, } k \mbox{  odd}
\end{cases} $$
  \ \vspace{-10pt} 

$$\ \mathrm{E}^2_{p-2, \ast-p+2} (p)= \HH^{\Lie^{\mathcal{H}_u}}_{p-1}( \mathfrak{g}(\RR^n, S^k))(p))_{\ast-p+1} \  =  \begin{cases}
 \Sigma^{kp+n-1} K_\ast  & \mbox{for } n \mbox{   even, } k \mbox{   even}\\ 
  0 & \mbox{for } n \mbox{ even, } k \mbox{  odd}\\
0  & \mbox{for } n \mbox{  odd, } k \mbox{  even}\\
0 & \mbox{for } n \mbox{  odd, } k \mbox{  odd}
\end{cases} \ \ \ \ \ \ \ \  \ \ \ \ \  $$
 \ \vspace{-10pt}

$$\mathrm{E}^2_{p-1, \ast-p+1} (p) \   =  \    \HH^{\Lie^{\mathcal{H}_u}}_{p}( \mathfrak{g}(\RR^n, S^k))(p))_{\ast-p}   \    =  \    \begin{cases}
 \Sigma^{kp}K_\ast   & \mbox{for } n \mbox{   even, } k \mbox{   even}\\
  0 & \mbox{for } n \mbox{ even, } k \mbox{  odd}\\
\Sigma^{kp}K_\ast  & \mbox{for } n \mbox{  odd, } k \mbox{  even}\\
0 & \mbox{for } n \mbox{  odd, } k \mbox{  odd}
\end{cases}  \ \ \  \ \ \  \ \ \  \ \ \  \ \ \  \ \  $$
 \ \vspace{5pt} \\
The modules $\mathrm{E}^2_{s,t} (p)$ in our spectral sequence vanish for all other values of $s$ and $t$.
\\
\\
\textit{Step 3)}.  We will now compute the differentials $d_r : \mathrm{E}^r_{s,t}(p) \rightarrow \mathrm{E}^r_{s-r,t+r-1}(p)  $.\vspace{5pt}
\\
Case 1: $n$ even, $k$ odd. In this simplest case, the spectral sequence is concentrated on the $(s=0)$-line, which allows us to read off the result.\vspace{5pt}
\\
Case 2: $n$ odd, $k>0$ even. The spectral sequence is concentrated on the lines $(s=0)$ and $(s=p-1)$.  Hence $E^2_{s,t}(p)  \cong   \ldots \cong E^{p-1}_{s,t}(p) $.
The only possible non-vanishing differential is   $$d_{p-1}:  \mathrm{E}^{p-1}_{p-1,t}(p)  \rightarrow \mathrm{E}^{p-1}_{0,t+p-2}(p) .$$
Since $\mathrm{E}^{p-1}_{0,\ast}(p) $ is concentrated in odd and 
$\mathrm{E}^{p-1}_{p-1,\ast}(p) $   in even degrees, we cannot conclude  that $d_{r-1}$ is zero.
Indeed, a second thought reveals that it in fact should not vanish. \vspace{5pt}

Already for $n=1$, the James splitting shows that $K^\wedge_\ast(\Omega S^{k+1})$ is torsion-free, and hence the 
$p$-torsion generator $[y_1]\in \mathrm{E}^2_{0, \ast} (p,1,n)=   \Sigma^{k-1} K_\ast/p$ on the zero-line must be hit under $d_{p-1}$ (up to a unit) by the class $\gamma_p([x])$. Hence $\gamma_p(x)$ dies, whereas $x^p$ survives.

This observation propagates to higher loop spaces, and we can use it
to compute $d_{p-1}$ for all higher odd $n$.
 For this, write $\mathrm{E}^{p-1}_{s,t}(p,n,k)$ for the spectral sequence converging to $K_\ast^{\wedge}(\Sigma^\infty\Conf_p(\RR^n)_+ \otimes_{h\Sigma_p} (S^k)^{\otimes p})$.  
The map  $\Sigma^\infty_+\Omega S^{1+k} \rightarrow \Sigma^\infty_+ \Omega^n S^{n+k}$ corresponds to the canonical  map of spectral Lie algebras $$\Omega  \Free^{\mathscr{L}}(S^{k} ) \rightarrow  \Omega^n  \Free^{\mathscr{L}}(S^{n+k-1} ).$$ 
We obtain a map of spectral sequences $\phi: \mathrm{E}^r_{s,t}(p,1,k)\rightarrow\mathrm{E}^r_{s,t}(p,n,k) $.
On  $\mathrm{E}^1$, it is obtained by applying the Hecke--Chevalley--Eilenberg complex to 
the induced map of \mbox{Hecke Lie algebras} $$\mathfrak{g}_{{k-1}, 1}^{(1)} = \mathfrak{g}(\RR^1, S^k) \ \xrightarrow{ \ \ \ \phi \ \  \ } \ \mathfrak{g}(\RR^n, S^k) = \mathfrak{g}_{{k-1},n}^{(1)}.$$
We can describe this map entirely explicitly. To this end, denote the natural generators by
$$x_1 \in \mathfrak{g}(\RR^1, S^k)_{k-1}\  , \  y_1 \in \mathfrak{g}(\RR^1, S^k)_{k-2}, \ \ \mbox{ and } \ \  x_n \in \mathfrak{g}(\RR^n, S^k)_{k-1} \  , \  y_n \in \mathfrak{g}(\RR^n, S^k)_{k-2}.$$
As $\Omega  \Free^{\mathscr{L}}(S^{k} ) \rightarrow  \Omega^n  \Free^{\mathscr{L}}(S^{n+k-1} )$ is obtained by looping a free map, we find
$$\phi(x_1) = x_n \ , \ \ \  \phi(y_1) = p^{\lfloor \frac{n}{2}\rfloor} y_n .$$
Hence the induced map on the $(s=0)$--line of $\mathrm{E}_2$-pages is given by  
$$  \mathrm{E}^2_{0,\ast}(p,1,k) = \Sigma^{k-1}K^\ast/p\xrightarrow{ \ \  \ p^{\frac{n}{2}} \ \ \ \ }  \Sigma^{k-1}K_\ast/{ p^{  \frac{n+1}{2}  }}  = \mathrm{E}^2_{0,\ast}(p,n,k)$$
Since $ \mathrm{E}^{p-1}_{0,\ast}(p,1,k) \cong \mathrm{E}^{2}_{0,\ast}(p,1,k) $ is killed by $d^{p-1}$ and 
$\mathrm{E}^{p-1}_{0,\ast}(p,1,k) \rightarrow \mathrm{E}^{p-1}_{0,\ast}(p,n,k)$ respects the differentials, we deduce
$\mathrm{E}^{p}_{0,\ast}(p,n,k) = \Sigma^{k-1}K_\ast/p^{  \frac{n-1}{2}  }$ and \mbox{$\mathrm{E}^{p}_{p-1,\ast-p-1}(p,n,k) = \langle p \gamma_p([x]) \rangle =\langle [x]^p\rangle$.} 

As no further differentials can occur, we have computed the $ \mathrm{E}^\infty$-page. But $\mathrm{E}^{\infty}_{p-1,\ast-p-1}$ is a free $K_\ast$-module, and the only other nontrivial line is at $(s=0)$, so there are no extension problems. \mbox{This finishes the proof for $n$ odd, $k>0$ even.}

If $n$ and $k$ have the same parity and $k>0$,  we can avoid further hard work by remembering that $p>2$.  In particular, setting $2m=n+k$, we may apply Serre's $p$-local equivalence  \cite{serre1953groupes}:
$$\Omega S^{2m}_{(p)} \simeq S^{2m-1}_{(p)} \times \Omega S^{4m-1}_{(p)}.$$
Taking $n-1$ additional loops yields the $p$-local equivalence
$$\Omega^n S^{n+k}_{(p)} \simeq \Omega^{n-1} S^{n+k-1}_{(p)} \times \Omega^{n} S^{2n+2k-1}_{(p)},$$
which is given \cite[\S 2.1]{Grbic} by the product of the suspension map $E:\Omega^{n-1} S^{n+k-1} \to \Omega^{n} S^{n+k}$ with the $n$-fold loops of the Whitehead square $[\iota,\iota]:S^{2n+2k-1} \to S^{n+k}$.  After stabilising, the map $\Sigma^{\infty} \Omega^n [\iota,\iota]$ has domain a free $\mathbb{E}_n$-algebra with generator mapping to a weight $2$ class.
Adding basepoints, stabilising, and passing to weight $p$, we obtain a $p$-local decomposition
$$\Sigma_+^\infty\Omega^n S^{n+k}(p) \simeq_{(p)} \bigoplus_{i+2j=p} (\Sigma^\infty_+ \Omega^{n-1} S^{n+k-1})(i) \otimes ( \Sigma^\infty_+ \Omega^{n} S^{2n+2k-1})(j).$$

As  the reduced $p$-adic $K$-theory of symmetric groups $B\Sigma_i$ vanishes whenever $i<p$, 
the summand $K^\wedge_\ast(\EE_n(S^a)(w)) = K^\wedge_\ast(B_w(\RR^n; S^a))$ is isomorphic to the weight $w$ component  in the free Poisson algebra in \mbox{$K_*$-modules} on a class in degree $a$, for all $w<p$.  
In particular, they are finite free $K_*$-modules in this range. This can also be  proven directly from our spectral sequence. 
We obtain  \mbox{an isomorphism}
$$K_\ast^\wedge (\Sigma_+^\infty\Omega^n S^{n+k}(p))\cong\bigoplus_{i+2j=p} K_\ast^\wedge (\EE_{n-1}(S^k)(i))
\otimes_{K_\ast} K_\ast^\wedge (\EE_{n}(S^{n+2k-1})(j)).$$

If $k$ and $n$ are even, then $n+2k-1$ is odd, which implies that the right hand side vanishes for degree reasons unless $j=0$ (which gives a copy of $ \Sigma^{kp}K_\ast \oplus \Sigma^{k-1} K_\ast/p^{\lfloor\frac{n-1}{2} \rfloor}$ ) or $j=1$ (which gives a copy of $\Sigma^{pk+n-1} K_*$).

 If $k$ and $n$ are odd, then the right hand side vanishes unless   $i=1$ (in which case $j=\frac{p-1}{2}$ and we obtain a copy of $K_*$), or $i=p$ (which contributes a copy of $\Sigma^{k-1}K_*/p^{\lfloor \frac{n-1}{2}\rfloor }$. This finishes the proof whenever $k>0$.

Finally, the computation extends to general $k$ by using that over a complex oriented homology theory, we have an equivariant equivalence  of (na\"{i}ve) $\Sigma_n$-spectra $K \otimes ({S^2})^{\otimes n} \simeq \Sigma^{2n}K $.

\end{proof}

\subsection{Interpretation in terms of power operations} 
Recall that the free $K(1)$-local  $\mathbb{E}_\infty$-$K^{\wedge}_p$-algebra on a class $x$ in degree $2$ is given by
$$\left(K \wedge \Sigma^{\infty}_+ \Omega^{\infty} S^2 \right)^{\wedge}_p.$$
By a classical computation (cf. e.g. \cite{wilkerson1982lambda} \cite{hopkins1998k}), its homotopy groups are isomorphic to the completion of the  algebra $\mathbb{Z}_p[u^{\pm1}][x,\theta x,\theta \theta x, \cdots].$
Here, the class $x$ is in weight $1$, while the class $\theta x$ is in weight $p$, the class $\theta \theta x$ is in weight $p^2$, and so on.
One sees that the entire algebra is generated by the class $x$ together with a single {Dyer--Lashof} operation $\theta$.
The weight  $p$ component is given by a  free $K_*$-module on $x^p$ \mbox{and $\theta x$.}

The above calculations allow us to understand this component not only for free $\mathbb{E}_\infty$-algebras, but also for free $\mathbb{E}_n$-algebras with $n<\infty$.  Indeed,  the  free $K(1)$-local $\mathbb{E}_n$-$K^{\wedge}_p$-algebra on a \mbox{class $x$} in degree $2$ is given by
$$\left(K \wedge \Sigma^{\infty}_+ \Omega^{n} S^{n+2} \right)^{\wedge}_p.$$
By examining the weight $p$ component of the above algebra, we can determine what vestige of the $\theta$ operation remains present in $\mathbb{E}_n$-algebras.  According to the above calculations, the nature of this weight $p$ component varies depending upon the parity of $n$.

When $n$ is odd, we see from \Cref{warmup} that the weight $p$-component of the free $K(1)$-local $\mathbb{E}_n$-$K^{\wedge}_p$-algebra on a degree $2$ class is equivalent to
$$\Sigma^{2p} K^{\wedge}_{p} \oplus \Sigma K/p^{\frac{n-1}{2}}.$$
 The copy of $\Sigma^{2p} K^{\wedge}_p$ corresponds to the element $x^p$ in the free $\mathbb{E}_\infty$-algebra, while the torsion module $\Sigma K/p^{\frac{n-1}{2}}$ is related to $\theta x$.  Indeed, there is a diagram  
$$
\begin{tikzcd}
\left(K \wedge \Sigma^{\infty}_+ \Omega^{n} S^{n+2} \right)^{\wedge}_p \arrow{r} & \left(K \wedge \Sigma^{\infty}_+ \Omega^{\infty} S^{2} \right)^{\wedge}_p \\
\Sigma K/p^{\frac{n-1}{2}} \arrow{u} \arrow{r}{\text{Bockstein}} & \Sigma^2 K^{\wedge}_p. \arrow{u}{\theta x}
\end{tikzcd}
$$

The sequence of suspension maps
$\Omega S^{3} \longrightarrow \Omega^3 S^5 \longrightarrow \Omega^5 S^7 \longrightarrow \cdots \longrightarrow \Omega^{\infty} S^2$  
induces a sequence of $K^{\wedge}_p$-module Bocksteins
$$0 \longrightarrow \Sigma K/p \longrightarrow \Sigma K/p^2 \longrightarrow \Sigma K/p^3 \longrightarrow \cdots \longrightarrow \Sigma^2 K^{\wedge}_p,$$
exhibiting the   relation $\text{hocolim}_n (\Sigma K/p^n) \simeq \Sigma^2 K^{\wedge}_p$ in $K(1)$-local $K$-modules. 

When $n$ is even, we read off from \Cref{warmup} that the weight $p$-component of the free $\mathbb{E}_n$-$K^{\wedge}_p$-algebra on a degree $2$ class is equivalent to
$$\Sigma^{2p} K^{\wedge}_{p} \oplus \Sigma^{2p+n-1} K^{\wedge}_p \oplus \Sigma K/p^{\frac{n}{2}}.$$
The copy of $\Sigma^{2p} K^{\wedge}_p$ is related to $x^p$ in the free $\mathbb{E}_\infty$-algebra, and $\Sigma K/p^{\frac{n}{2}}$ is a shadow of $\theta x$.  

The copy of $\Sigma^{2p+n-1} K^{\wedge}_p$ represents a
phenomenon specific to $\EE_n$-algebras: their homotopy groups form Poisson algebras in $K_*$-modules, with an $(n-1)$-shifted Lie bracket $[-,-]$.  The class in degree $(2p+n-1)$ then corresponds to the Poisson word $[x,x]x^{p-2}$.

\begin{remark}
Throughout the above discussion, we have focused on the situation of a free $K(1)$-local $\mathbb{E}_n$-$K^{\wedge}_p$-algebra generated by a degree $2$ class $x$.  We could similarly use  \Cref{warmup}  to understand the free $\mathbb{E}_n$-algebra on a class in any other degree.  
Four  cases will arise, depending 
on what the parities of $n$ and $k$ imply about the vanishing of \mbox{$x^2$, $[x,x]$ based on Koszul sign rules.} 
\end{remark}  
\subsection{The $E$-theory of $p$ points in $\RR^n$} \label{section:higher heights} 
Fix a form of $E$-theory at height $h$ and prime $p>2$. We will carry out the following computation, which, to the best of our knowledge, is new:

\begin{theorem}[$E$-theory, Euclidean case]\label{prop:general calculations}
For any positive integer $n$ and integer $k$, the $E_*$-module $E_*^\wedge \left(\Conf_p(\RR^n)_+ \otimes_{h\Sigma_p} (S^k)^{\otimes p} \right)$ is given by 
$$ E^\wedge_\ast(B_p(\RR^n; S^k))=     \begin{cases}
  \ \ \ \      \Sigma^{kp}E_\ast   \oplus \Sigma^{pk+n-1} E_*  \oplus \Sigma^{k-1} E^\ast(B\Sigma_p)/( \tr, e^{\frac{n}{2}-1} )    & \mbox{for } n \mbox{   even, } k \mbox{   even}\\
  \ \ \ \  \Sigma^{k-1} E^\ast(B\Sigma_p)/( \tr, e^{ \frac{n}{2}  }  )& \mbox{for } n \mbox{ even, } k \mbox{  odd}\\
  \ \ \ \  \Sigma^{kp}E_\ast  \oplus \Sigma^{k-1} E^\ast(B\Sigma_p)/(\tr, e^{ \frac{n-1}{2} }) & \mbox{for } n \mbox{  odd, } k \mbox{  even}\\
  \ \ \ \  \Sigma^{k+ (2k+n-1)(\frac{p-1}{2})} E_\ast  \oplus  \Sigma^{k-1} E^\ast(B\Sigma_p)/(\tr, e^{  \frac{n-1}{2} })& \mbox{for } n \mbox{  odd, } k \mbox{  odd}
\end{cases} $$  Here $(\tr)$ denotes the transfer ideal associated to the inclusion of the trivial group, whereas $e\in E^0(B\Sigma_p)$ is the Euler class of the reduced complex standard representation.
\end{theorem} 
\begin{remark}\label{2iseasy} 
The $p=2$ analogue of the above theorem is much more elementary.  Here, one may rely on the equivalence
$$\Conf_2(\RR^n) \otimes_{\Sigma_2} (S^{k})^{\otimes 2} \simeq \Sigma^{k} \mathbb{RP}_{k}^{n+k-1},$$
where $\mathbb{RP}_{k}^{n+k-1}$ denotes the cofiber of the inclusion $\mathbb{RP}^{k-1} \to \mathbb{RP}^{n+k-1}$.  This reduces the question to the calculation of the $E$-theory of real projective spaces, which is classically accomplished via the Gysin sequences associated to the fibrations 
\mbox{$S^1 \to \mathbb{RP}^{2n+1} \to \mathbb{CP}^{n}.$}

As pointed out by an anonymous referee, one can also prove at least the second part of the above theorem by an analogous observation at an odd prime, using the cofiber sequence
\[\Sigma^n B_p(\mathbb{R}^n;S^1) \to \Sigma^n B_p(\mathbb{R}^{\infty};S^1) \to B_p(\mathbb{R}^{\infty};S^{n+1})\]
that exists for even integers $n$ (cf.\ \cite[Proposition 1.3]{kuhn1982geometry}). Such arguments are special to weight $p$,  but  the methods of this paper can also be applied to large weights and configurations in more general manifolds.  As evidence of this, we present the more delicate \Cref{thm:open surfaces}; the reader may view \Cref{prop:general calculations} as a warm-up calculation, though the result was not previously recorded in the literature.
\end{remark}

While our computation will be notationally more cumbersome, it relies on the same strategy as   the preceding section \ref{warmupsection}. First, we need a higher height variant of the atomic Lie algebras from \Cref{atomicalgebra}, which will make use of the Hecke power ring introduced in  \Cref{Heckepowerring}: 
\begin{definition}[Atomic Hecke Lie algebras at height $h$]\label{higheratomicalgebra}
Given positive integers $n, w$ and an integer $a$, we define  the atomic Hecke Lie algebra $\mathfrak{h} = \mathfrak{h}_{a,s}^{(w)}$ as follows:
\begin{enumerate} 
\item The underlying  graded abelian group of  $\mathfrak{h}$ is given by 
$$\mathfrak{h}_\ast = \bigoplus_{0\leq \ell \leq h} (\mathcal{H}_u^{\Lie})_{n+a}^{n+\ast}(p^\ell) .$$
We shall denote the element corresponding to some $\phi \in (\mathcal{H}_u^{\Lie})_{n+a}^{n+i}(p^\ell)$ by the \mbox{symbol $x_{\phi}$;} it will live  in {homological degree $i$} and {weight $p^\ell w$.} 

\item Hecke operations in  $  (\mathcal{H}_u^{\Lie})_i^j(w)$ act on  $(\mathcal{H}_u^{\Lie})_{n+a}^{n+i }(p^k) \subset \mathfrak{h}_i$ 
via $$ (\mathcal{H}_u^{\Lie})_i^j(w) \times  (\mathcal{H}_u^{\Lie})_{n+a}^{n+i}(p^k)  \xrightarrow{\Susp^n \times \id}   (\mathcal{H}_u^{\Lie})_{n+i}^{n+j}(w) \times  (\mathcal{H}_u^{\Lie})_{n+a}^{n+i}(p^k) \rightarrow  (\mathcal{H}_u^{\Lie})_{n+a}^{n+j}(wp^k) \subset \mathfrak{h}_j $$ 
\item All Lie brackets vanish.
\end{enumerate}
\end{definition}
We generalise \Cref{atomiclemma} to higher heights: 
\begin{lemma}[Homology of atomic Hecke Lie algebras at height $h$]\label{atomiclemmahigher}
\mbox{In weights $k\leq pw$,  we have:}  
$$\mbox{For $a$ even: }\HH^{\Lie^{\mathcal{H}_u}}(\mathfrak{h}_{a,n}^{(w)})(k)  
 \cong (\Lambda_{E_\ast}[x_{(a ,1,0,w)}] \oplus  (\mathcal{H}_u^{\Lie})_{n+a}^{n+ \ast}(p)/_{\Susp^n})(k)  $$
$$ \ \  \ \ \ \  \ \ \  \ \ \  \ \ \ \ \ \   \ \   \ \  \ \  \ \ \ \   \ \  \ \ \ \    \ \  \ \   \vspace{-10pt}  \ \   \ \ \  \     \cong  (\Lambda_{E_\ast}[x_{(a ,1,0,w)}] \oplus  \Sigma^{a-1}E^{-\ast} (B\Sigma_{p})/( \tr, e^{\lfloor \frac{n}{2} \rfloor }))(k)$$ 
 
$$ \ \ \ \ \ \ \ \   \mbox{For $a$ odd: }   \   \HH^{\Lie^{\mathcal{H}_u}}(\mathfrak{h}_{a,n}^{(w)})(k) \cong (\Gamma_{E_\ast}[x_{(a,1,0,w)}] \oplus (\mathcal{H}_u^{\Lie})_{n+a}^{n+ \ast}(p)/_{\Susp^n} )(k)  \ \ \ \  \ \   \ \ \vspace{5pt}$$
$$ \ \  \ \ \ \    \  \ \ \ \   \ \  \ \ \ \  \ \ \  \ \ \ \ \ \ \   \ \   \ \  \  \ \  \   \  \  \  \ \  \  \  \  \ \  \cong  (\Gamma_{E_\ast}[x_{(a ,1,0,w)}] \oplus   \Sigma^{a-1}E^{-\ast}(B\Sigma_{p})/(\tr, e^{\lceil \frac{n}{2} \rceil } ))(k).\ $$ 
As above, $(\tr)$ denotes the transfer ideal associated to the trivial subgroup, and $e\in E^0(B\Sigma_p)$ is the Euler class of the reduced complex standard representation. 
\vspace{-3pt} 
\end{lemma} 

\begin{proof}
The symbols $x_{\phi} $ introduced in \Cref{higheratomicalgebra} are $E_\ast$-linear, i.e. satisfy $\lambda \cdot x_{\phi} = x_{\lambda \cdot\phi}$ for all  
$\lambda \in E_\ast$. Write 
  $x_1 = x$ for the canonical generator in degree $a$ and weight $1$. 
We draw $\AR(\mathfrak{h}_{a,n}^{(w)})$, as defined in \Cref{construction:additive resolution}: \vspace{-10pt}

\begin{figure}[H]
\begin{small}
\begin{center}\def\arraystretch{1.5}
\begin{tabular}{c|cccc}
$r\backslash weight$ &$1w$ & $pw$ & $p^2w$&$\cdots$\\
\hline
$0$&$\substack{\left[x\right]_{_{(a,1,0,1w)}}}$ & \ \ \ \ \ $\substack{\left[x_{\alpha }\right]_{_{(a-1,1,0,pw)}}}$ & \ \ \ \ \ $\substack{\left[x_{\beta }\right]_{_{(a-2,1,0,p^2 w)}}}$   &$\cdots$\ \vspace{-9pt}
\\\\

$1$&$\substack{\left[1|x\right]_{_{(a,1,1,1w)}}}$ & \ \ \ \ \  $\substack{\left[ \alpha' | x \right]_{_{(a-1,1,1,pw)}}\\ \left[ 1 | x_\alpha  \right]_{_{(a-1,1,1,pw)}}}$  &
$ \ \ \ \ \ \
\substack{  
\ \ \left[ 1| x_{\beta} \right]_{_{(a-2,1,1,p^2w)}}\\ 
 \ \ \   \left[ \alpha'  | x_{\alpha}\right]_{_{(a-2,1,1,p^2w)}} \\ 
\ \left[ \beta' | x \right]_{_{(a-2,1,1,p^2w)}}}  $  \vspace{-5pt}
\\\\
$2$&$\substack{\left[1|1|x\right]_{_{(a,1,2,1w)}}}$ &  \ \ \ \ \    $\substack{  \left[ \alpha' | 1 | x \right]_{_{(a-1,1,2,pw)}} \\\left[1 | \alpha'  | x \right]_{_{(a-1,1,2,pw)}} \\ \left[ 1 | 1 | x_\alpha \right]_{_{(a-1,1,2,pw)}}}$  &
$\substack{
 \ \ \ \ \   \ \ \ \  \left[ 1|1 | x_{\beta} \right]_{_{(a-2,1,2,p^2w)}}     \\ 
  \ \ \ \ \   \ \ \  \left[ 1| \beta' | x  \right]_{_{(a-2,1,2,p^2w)}} \\ 
 \ \ \ \ \   \ \ \   \left[ \beta' | 1| x\right]_{_{(a-2,1,2,p^2w)}}\\ 
 \ \ \ \ \   \ \ \ \  \  \left[ 1| \alpha''  | x_{\alpha } \right]_{_{(a-2,1,2,p^2w)}}     \\ 
 \ \ \ \ \   \ \ \  \ \ \left[ \alpha'' | 1 | x_{\alpha}  \right]_{_{(a-2,1,2,p^2w)}} \\ 
 \ \ \ \ \   \ \ \ \ \ \left[ \alpha'' | \alpha'  |x\right]_{_{(a-2,1,2,p^2w)}}}$
&$\cdots$\\
\vdots&\vdots&\vdots&\vdots
\end{tabular}\vspace{-20pt}
\end{center} 
\end{small}
\end{figure} 
 Here $\alpha$ runs over some fixed $E_0$-module basis  for $(\mathcal{H}_u^{\Lie})_{n+a}^{n+a-1}(p)$, the element $\alpha'$ runs over some fixed  basis  for $({\mathcal{H}_u^{\Lie}})_{a}^{a-1}(p)$, the element
 $\alpha''$ runs over some fixed   basis  for \mbox{$({\mathcal{H}_u^{\Lie}})_{a-1}^{a-2}(p)$, etc.}
Similarly, $\beta$ runs over some fixed   basis  for $({\mathcal{H}_u^{\Lie}})_{n+a}^{n+a-2}(p^2)$, $\beta' $ runs over some fixed   basis  for $({\mathcal{H}_u^{\Lie}})_{a}^{a-2}(p^2)$,  
$\beta'' $ runs over some fixed  basis  for $({\mathcal{H}_u^{\Lie}})_{a}^{a-2}(p^2)$, etc.
The subscripts again reflect the quadruple grading from \Cref{quadruple}. 
We will again use \Cref{thm:hecke chevalley--eilenberg works} to compute $\HH^{\Lie^{\mathcal{H}_u}}(\mathfrak{h}_{a,n}^{(w)})$ as the homology of the   complex
$\CE(\AR(\mathfrak{h}_{a,n}^{(w)})) =  \Gamma_{K_\ast}(\AR(\mathfrak{h}_{a,n}^{(w)})[1]).$

For $a$  even,   $[x]$ has odd total degree in the chain complex of graded $E_\ast$-modules $\AR_0(\mathfrak{h}_{a,n}^{(w)})[1]$ which implies  that $[x]^i$ vanishes when $i>1$. 
Using the structure maps in \Cref{higheratomicalgebra}, we see that the normalised chain complex of $\CE(\AR(\mathfrak{h}_{a,n}^{(w)}))$ is  therefore given by the complexes 
$$ \ \ \ \ \ \ \ \  \ldots \rightarrow 0\rightarrow 0\rightarrow 0 \rightarrow 0 \rightarrow [1] \ \mbox{ in weight } 0,$$
$$ \ \ \ \ \ \  \ \ \ldots  \rightarrow 0 \rightarrow 0\rightarrow 0 \rightarrow [x]\rightarrow 0  \ \mbox{ in weight } w,$$
$$  \ldots \rightarrow 0 \rightarrow ({\mathcal{H}_u^{\Lie}})_{a}^{\ast}(p) \xrightarrow{\Susp^n} (\mathcal{H}_u^{\Lie})_{n+a}^{n+\ast}(p) \rightarrow 0 \ \mbox{ in weight } pw, \ \ \ \ \ \ \ \ \ \ \ \ \ \ \ \ \ \ $$
and vanishes for all other smaller weights.

For $a$ odd,   $[x]$ has even total degree in the chain complex of graded $E_\ast$-modules $\AR_0(\mathfrak{h}_{a,n}^{(w)})[1]$. We deduce that the normalised chain complex of $\CE(\AR(\mathfrak{h}_{a,n}^{(w)}))$ is  given by
$$\ \  \ \  \ \   \ \  \ \ \ \   \ldots \rightarrow 0\rightarrow 0\rightarrow 0\rightarrow 0 \rightarrow   0    \xrightarrow{}  [x]^i    \rightarrow 0 \rightarrow  \ldots \rightarrow 0 \mbox{ \ \  \ \   in weight } iw  \mbox{ for } i=0, 1, \ldots,  (p-1), \mbox{\    \ \  \ \ \  } \ \ \ \ \   
 \vspace{-3pt}$$
$$   \ \   \ \ \ \  \ \ \ldots  \rightarrow \gamma_p([x]) \rightarrow  0 \rightarrow \ldots  \rightarrow 0 \rightarrow ({\mathcal{H}_u^{\Lie}})_{a}^{\ast}(p) \xrightarrow{\Susp^n} (\mathcal{H}_u^{\Lie})_{n+a}^{n+\ast}(p) \rightarrow 0\ \ \ \ \mbox{ in weight }pw,   \ \ \ \  \ \  \ \  \ \ \ \   \   \ \  \ \ \ \ \ \ \ \  \ \ \  \ \ \ \ $$
and vanishes in all other weights below $pw$. 
This proves the first part of both cases; the remaining claims follow from \Cref{Gamma0}. 
\end{proof} 
With this computation at hand, we can  prove \Cref{prop:general calculations} along the same lines \mbox{as \Cref{warmup}.}
We  recall  the  structure of   free Hecke Lie algebras on a single class from \cite[Section 4.4.2]{brantnerthesis}:
\begin{corollary}[Free Hecke Lie algebras on one generator at height $h$]\label{freeliehigher}\  

If $x_a$ is a generator in even degree $a$, then   
$$\LL(x_a)_\ast = \bigoplus_{0\leq \ell \leq h} (\mathcal{H}_u^{\Lie})_{a}^{\ast}(p^\ell) .$$
Write elements corresponding to $\phi\in (\mathcal{H}_u^{\Lie})_{a}^{\ast}(p^\ell)$ as $x_\phi$. The Lie bracket vanishes, and the Hecke operations are determined by  $\alpha(x_{\phi})=x_{\alpha \phi}$.

If $x_a$ is a generator in odd degree $a$, then  $\LL(x_a)$  is given by
$$\LL(x_a)_\ast = \bigoplus_{0\leq \ell \leq h} (\mathcal{H}_u^{\Lie})_{a}^{\ast}(p^\ell)  \oplus
\bigoplus_{0\leq \ell \leq h} (\mathcal{H}_u^{\Lie})_{2a}^{\ast}(p^\ell).$$
Write elements corresponding to $\phi\in (\mathcal{H}_u^{\Lie})_{a}^{\ast}(p^\ell)$ on the left as $x_\phi$ and to 
$\phi\in (\mathcal{H}_u^{\Lie})_{2a}^{\ast}(p^\ell)$
on the right as $\widetilde{x}_{\phi}$.
The Lie bracket satisfies $[x_\lambda,x_\mu] = \widetilde{x}_{\lambda \mu}$ for all scalars $\lambda, \mu \in E_\ast \cong  (\mathcal{H}_u^{\Lie})_{a}^{\ast}(1)$, and vanishes otherwise.
The Hecke operations \mbox{are determined by} $\alpha(x_\phi)=x_{\alpha \phi }$ and  $\alpha(\widetilde{x}_\phi)=\widetilde{x}_{\alpha \phi }$.
\end{corollary} 

\vspace{-5pt}
 
\begin{proof}[Proof of \Cref{prop:general calculations}] We will follow the same strategy as in the proof of \Cref{warmup}.\vspace{5pt}\\ 
 \Cref{prop:hecke desuspension} shows that  $\mathfrak{g}(\RR^n, S^k) $ is obtained by desuspending
 $\LL(x_{n+k-1})$  exactly $n$ times in internal degree,  
and then acting through $n$-fold suspended Hecke operations.
As in the proof of \Cref{warmup},  \Cref{freeliehigher} then   implies that  $\mathfrak{g}(\RR^n, S^k)$ is given as indicated below:
\begin{figure}[H]
\begin{small}
\begin{center}\def\arraystretch{1.5}
\begin{tabular}{c|cccc} 
$k\backslash n$ & \ \ even & odd & \\
\hline 
even \ \ \ \  &$ \ \ \ \  \mathfrak{h}^{(1)}_{k-1,n} \oplus  \mathfrak{h}^{(2)}_{n+2k-2,\lfloor \frac{n}{2}\rfloor} $ & $\mathfrak{h}^{(1)}_{k-1, n} $  
\\
odd \ \ \ \   &$ \ \ \ \  \mathfrak{h}^{(1)}_{k-1, n}$ &  $\mathfrak{h}^{(1)}_{k-1,n} \oplus  \mathfrak{g}_{n+2k-2,n}^{(2)}  $
\end{tabular}\vspace{-5pt}
\end{center}
\end{small} 
\end{figure}\ 
We read off $\HH^{\Lie^{\mathcal{H}_u}}( \mathfrak{g}(\RR^n, S^k))(p) $ from 
  \Cref{atomiclemma} and   the isomorphism $\HH^{\Lie^{\mathcal{H}_u}}(\mathfrak{h} \oplus \mathfrak{h}')(w) \cong (\HH^{\Lie^{\mathcal{H}_u}}(\mathfrak{h}) \otimes \HH^{\Lie^{\mathcal{H}_u}}(  \mathfrak{h}'))(w)$, which holds for all  involved Hecke Lie algebras $\mathfrak{h}, \mathfrak{h}'$ \mbox{in weights $w\leq p$.}

\begin{figure}[H]
\begin{small}
\begin{center}\def\arraystretch{1.5}
\begin{tabular}{c|cccc} 
$k\backslash n$ & \ \ even & odd & \\
\hline
even \ \ \ \  &
$  \ \ \ \ \substack{ \  \\ \ \\ \Sigma^{p(k-1)}E_\ast \\  \\  
\Sigma^{k-2}E^{-\ast}(B\Sigma_{p})/(\tr, e^{\lceil \frac{n}{2} \rceil }) \\ \\
\Sigma^{p(k-1) + n}E_\ast}
$  \ \ \ \ \ &  $\substack{ \Sigma^{p(k-1)}E_\ast \\ \\  \Sigma^{k-2}E^{-\ast}(B\Sigma_{p})/(\tr, e^{\lceil \frac{n}{2} \rceil })} $  
\\\\
odd \ \ \ \   &$ \ \ \ \  \substack{ \Sigma^{k-2}E^{-\ast}(B\Sigma_{p})/(\tr, e^{\lfloor \frac{n}{2} \rfloor }) }$ &  $\substack{
\Sigma^{(k-1)+ \frac{p-1}{2} (n+2k-2)}E_\ast \\ \\
 \Sigma^{k-2}E^{-\ast}(B\Sigma_{p})/( \tr, e^{\lfloor \frac{n}{2} \rfloor })
}$
\end{tabular}
\end{center}
\end{small}
\caption{The Hecke Lie algebra homology $\HH^{\Lie^{\mathcal{H}_u}}( \mathfrak{g}(\RR^n, S^k))(p) $.}
\end{figure}

Hence, the $\mathrm{E}_2$-term $\mathrm{E}^2_{s, t} = \HH^{\Lie^{\mathcal{H}_u}}_{s+1}( \mathfrak{g}(\RR^n, S^k))_{t-1}$ of the spectral sequence at weight $p$ \mbox{is  given by:}   
 $$\ \ \ \ \ \mathrm{E}^2_{0, \ast} (p) \  \ \ \ = \ \  \ \ \  \HH^{\Lie^{\mathcal{H}_u}}_{1}( \mathfrak{g}(\RR^n, S^k))(p))_{\ast-1}  \    \ \ \ = \ \ \ \ \  \begin{cases}
  \Sigma^{k-1}E^{-\ast}(B\Sigma_{p})/( \tr, e^{\lceil \frac{n}{2} \rceil})& \mbox{for } n \mbox{ even, } k \mbox{ even}\\
\Sigma^{k-1}E^{-\ast}(B\Sigma_{p})/( \tr, e^{\lfloor \frac{n}{2} \rfloor })& \mbox{for } n \mbox{  even, } k \mbox{  odd}\\
 \Sigma^{k-1}E^{-\ast}(B\Sigma_{p})/( \tr, e^{\lceil \frac{n}{2} \rceil})& \mbox{for } n \mbox{  odd, } k \mbox{  even}\\
\Sigma^{k-1}E^{-\ast}(B\Sigma_{p})/( \tr, e^{\lfloor \frac{n}{2} \rfloor })& \mbox{for } n \mbox{ odd, } k \mbox{   odd}
\end{cases} \ \ \ \ \ \  \ \ \ \ \ \  \ \ \ \ \ \  \ \ \ \ \ \ $$
  \vspace{-10pt}
 
$$\mathrm{E}^2_{\frac{p-1}{2}, \ast-\frac{p-1}{2}} (p)= \HH^{\Lie^{\mathcal{H}_u}}_{\frac{p+1}{2}}( \mathfrak{g}(\RR^n, S^k))(p))_{\ast-\frac{p+1}{2}}  =  \begin{cases}
0 & \mbox{for } n \mbox{   even, } k \mbox{   even}\\
  0 & \mbox{for } n \mbox{ even, } k \mbox{  odd}\\
0  & \mbox{for } n \mbox{  odd, } k \mbox{  even}\\
 \Sigma^{(2k+n-1)(\frac{p-1}{2})+k}E_\ast     & \mbox{for } n \mbox{  odd, } k \mbox{  odd}
\end{cases} $$
  \ \vspace{-10pt} 

$$\ \ \mathrm{E}^2_{p-2, \ast-p+2} (p)= \HH^{\Lie^{\mathcal{H}_u}}_{p-1}( \mathfrak{g}(\RR^n, S^k))(p))_{\ast-p+1}   \ =  \begin{cases}
 \Sigma^{kp +n -1} E_\ast     & \mbox{for } n \mbox{   even, } k \mbox{   even}\\ 
  0 & \mbox{for } n \mbox{ even, } k \mbox{  odd}\\
0  & \mbox{for } n \mbox{  odd, } k \mbox{  even}\\
0 & \mbox{for } n \mbox{  odd, } k \mbox{  odd}
\end{cases} \ \ \ \ \ \ \ \  \ \ \ \ \  $$
 \ \vspace{-10pt}

$$\mathrm{E}^2_{p-1, \ast-p+1} (p) \   =  \    \HH^{\Lie^{\mathcal{H}_u}}_{p}( \mathfrak{g}(\RR^n, S^k))(p))_{\ast-p}   \    =  \    \begin{cases}
 \Sigma^{kp}E_\ast   & \mbox{for } n \mbox{   even, } k \mbox{   even}\\
  0 & \mbox{for } n \mbox{ even, } k \mbox{  odd}\\
\Sigma^{kp}E_\ast  & \mbox{for } n \mbox{  odd, } k \mbox{  even}\\
0 & \mbox{for } n \mbox{  odd, } k \mbox{  odd}
\end{cases}  \ \ \  \ \ \  \ \ \  \ \ \  \ \ \  \ \  $$
 \ \vspace{5pt} \\
The modules $\mathrm{E}^2_{s,t} (p)$ in our spectral sequence vanish for all other values of $s$ and $t$.\\ \\
We will now compute the  $\mathrm{E}^{\infty}$-page of the spectral sequence  
and solve extension problems.

Case 1: $n$ even, $k$ odd. The spectral sequence is concentrated on the $(s=0)$-line.

Case 2: $n$ odd, $k>0$ even. The spectral sequence is concentrated at $(s=0)$ and $(s=p-1)$, so $E^2_{s,t}(p)  \cong E^3_{s,t}(p)  \cong   \ldots \cong E^{p-1}_{s,t}(p) $. As before, the next differential  $d_{p-1}$ is again nonzero.
Indeed, let $\mathrm{E}^{r}_{s,t}(p,n,k)$ denote  the spectral sequence associated to associated to the pair $(n,k)$. For   $n=1$, the James splitting shows that $E^\wedge_\ast(\Omega S^{k+1})$ is torsion-free.  Hence 
the generator $[y_1]$ of $\mathrm{E}^2_{0, \ast} (p,1,n)=  \Sigma^{k-1}E^{-\ast}(B\Sigma_{p})/(\tr, e) = \Sigma^{k-1} E_\ast/p$ must lie in the image of $d^{p-1}$.

For higher  $n$, we  obtain a map of spectral sequences $\phi: \mathrm{E}^r_{s,t}(p,1,k)\rightarrow\mathrm{E}^r_{s,t}(p,n,k) $ 
from the weight $p$ sequence attached to $(1,k)$ to the sequence attached  to $(n,k)$. 
For $\mathrm{E}^1$, this map is determined by applying the Hecke Chevalley--Eilenberg  complex to 
the  map   $$\mathfrak{h}_{{k-1}, 1}^{(1)} = \mathfrak{h}(\RR^1, S^k) \ \xrightarrow{ \ \ \ \phi \ \  \ } \ \mathfrak{g}(\RR^n, S^k) = \mathfrak{h}_{{k-1},n}^{(1)}.$$

We can describe this map explicitly. Writing 
$x_1 \in \mathfrak{g}(\RR^1, S^k)_{k-1}$ and $x_n \in \mathfrak{g}(\RR^n, S^k)_{k-1}$
for the canonical generators, we note that $\phi(x_1) = x_2$.
More generally, if $x_{1,\alpha}\in   \mathfrak{g}(\RR^1, S^k)_{i}  = (\mathfrak{h}^{(1)}_{k-1,n})_i$ corresponds to some  class $\alpha \in (\mathcal{H}_u^{\Lie})_{1+(k-1)}^{1+i}(p)$, then $\phi (x_{1,\alpha}) = x_{2,\Susp^{n-1}(\alpha)} $ corresponds to $\Susp^{n-1}(\alpha) \in (\mathcal{H}_u^{\Lie})_{n+(k-1)}^{n+i}(p)$.
This shows that the induced map on the $(s=0)$--line of $\mathrm{E}^2$-pages is given by 
the natural  map 
$$  \mathrm{E}^2_{0,\ast}(p,1,k) = \Sigma^{k-1}E^{-\ast}(B\Sigma_{p})/( \tr, e^{1})\xrightarrow{ \ \  \ e^{\frac{n}{2}} \ \ \ \ }  \Sigma^{k-1}E^{-\ast}(B\Sigma_{p})/( \tr, e^{  \frac{n+1}{2}  }) = \mathrm{E}^2_{0,\ast}(p,n,k)$$

Since $ \mathrm{E}^{p-1}_{0,\ast}(p,1,k) \cong \mathrm{E}^{2}_{0,\ast}(p,1,k) $ is killed by $d^{p-1}$ and the map 
$\mathrm{E}^{p-1}_{0,\ast}(p,1,k) \rightarrow \mathrm{E}^{p-1}_{0,\ast}(p,n,k)$ respects the differentials, we deduce that
$\mathrm{E}^{p}_{0,\ast}(p,n,k)$ is given by $\Sigma^{k-1}E^{-\ast}(B\Sigma_{p})/(\tr, e^{  \frac{n-1}{2}  })$. The kernel of $d^{p-1}$ on the summand  $\langle \gamma_p([x])\rangle $ is spanned by $[x]^p$. As there is no space for   further differentials and all extension problems are trivial, this finishes the proof for $n$ odd, $k>0$ even.
We then extend to all $n$ and $k$ using the Serre fibration \cite{Serre:HSEF} and complex orientability of $E$-theory following the same strategy as in the proof of \Cref{warmup}.

\end{proof}

\subsection{A height 2 example} \label{sec:height 2}
We end by emphasising that the above formulas are explicit and workable in examples beyond $p$-adic $K$-theory.  To this end, consider the height $2$ Morava $E$-theory at $p=3$ investigated by Yifei Zhu \cite{zhu20power} arising from the moduli of elliptic curves with a choice of a point of exact order $4$.

This cohomology theory $E$ has coefficient ring $E_*=\mathbb{Z}_9[[h]][u^{\pm1}],$ where $\mathbb{Z}_9=W(\mathbb{F}_9)$, $h$ is a class in degree $0$, and $u$ is a class in degree $2$.  Zhu calculates \cite[Corollary 2.6]{zhu20power} the ring $E^0(B\Sigma_3)/\text{tr}$ to be given by
$$\mathbb{Z}_9[[h]][\alpha]/(\alpha^4-6\alpha^2+(h-9)\alpha-3).$$
The element $\alpha$ is closely related to the Euler class $e$. Rather than spelling out this relationship in detail, we simply note that there is an isomorphism of $E_0$-modules  
$$E_0[e]/(f(e),e^n) \cong E_0[\alpha]/(\alpha^4-6\alpha^2+(h-9)\alpha-3,\alpha^n),$$  which can be constructed by  observing that both sides compute the cokernel of the dual of the $2n$-fold suspension map
$$\Gamma^0(p) \longrightarrow \Gamma^{-2n}(p).$$
Indeed, Propositions \ref{suspensiondiagram} and \ref{explicit!} imply that $E_0[e]/(f(e),e^n)$ computes this cokernel, whereas one can read off the formula for the double suspension in terms of $\alpha$ from Zhu's Cartan formula \cite[Proposition 3.6(v)]{zhu20power} applied to \cite[Example 3.7]{zhu20power}.

We  
 suspect it will often be the case, when working with explicit $E$-theories of algebro-geometric origin, that the double suspension operation is presented more easily in a basis for $E^0(B\Sigma_p)/(\tr)$ other than that given by the Euler class $e$.  In any case, it is the cokernel of the double suspension operation that we are really after.
For example, \mbox{we have the following calculations.}
\begin{itemize}
\item Consider the free $\EE_1$-$E$-algebra on a generator $x$ in degree $0$.  The weight $3$ component of this algebra is a free $E_*$-module, generated by  $x^3$.  This is reflected in the calculation
\begin{eqnarray*}
E_*^\wedge(B_3(\mathbb{R}^1)) &\cong& E_* \oplus \Sigma^{-1} E_*[\alpha]/(\alpha^4-6\alpha^2+(h-9)\alpha-3,\alpha^0) \cong E_* 
\end{eqnarray*}  
which is of course clear since $B_3(\mathbb{R}^1)$ is contractible.\vspace{3pt}
\item On the other hand, we calculate
 \begin{eqnarray*}\ \ \ \ \ \ \ \ \ \ \ \ \ \ \ 
E_*^\wedge(B_3(\mathbb{R}^3)) &\cong& E_* \oplus \Sigma^{-1} E_*[\alpha]/(\alpha^4-6\alpha^2+(h-9)\alpha-3,\alpha^1)\cong  E_* \oplus \Sigma^{-1}E_*/3  
\end{eqnarray*}  
We can thus see that, in the free $\mathbb{E}_3$-$E$-algebra on a generator $x$ in degree $0$, there is more than the $E_*$-multiples of $x^p$ in weight $3$.  There is additionally a copy of $\Sigma^{-1}E_*/3$ that constitutes a kind of Dyer--Lashof operation on $x$.
\item We calculate
\begin{eqnarray*}
E_*^\wedge(B_3(\mathbb{R}^5)) &\cong& E_* \oplus \Sigma^{-1} E_*[\alpha]/(\alpha^4-6\alpha^2+(h-9)\alpha-3,\alpha^2) \\
&\cong&  E_* \oplus \Sigma^{-1} E_*\langle 1, \alpha  \rangle/{\langle {  3-(h-9)\alpha, 3\alpha\rangle}}
\end{eqnarray*}
In an $\mathbb{E}_5$-algebra, the weight $3$ Dyer--Lashof operations on $x$ now consist of an $E_*$-module on two generators with two relations.

\item We can easily continue with such calculations; for example
\begin{eqnarray*}
E_*^\wedge(B_3(\mathbb{R}^{11})) &\cong& E_* \oplus \Sigma^{-1} E_*[\alpha]/(\alpha^4-6\alpha^2+(h-9)\alpha-3,\alpha^5) 
\end{eqnarray*}
We obtain  the weight $3$ Dyer--Lashof operations in the free $\mathbb{E}_{11}$-$E$-algebra on a \mbox{degree $0$ class.} 
\end{itemize}

\subsection{Morava K-theory}\label{K-theory section}
One of the main advantages of our approach to iterated loop spaces of spheres, in contrast with the approaches of Yamaguchi \cite{yamaguchi1988moravak} (for $2$-fold loop spaces) and Tamaki \cite{tamaki2002fiber} (for $3$-fold and $4$-fold loop spaces),  is that we do not require a K\"unneth formula, and that our computation does not increase in difficulty with the number $n$ of loops. Thus, we are able to compute Morava $E$-theory for all $n$ simultaneously, rather than computing Morava $K$-theory by induction on $n$. Nevertheless, since $E$-theory is a refinement of  $K$-theory, it is to be expected that our $E$-theory calculations yield Morava $K$-theory calculations as byproducts. 

  In this section we will explain how to conclude facts about Morava $K$-theory from our work, which may be of interest as it settles the computational challenge raised by  Ravenel's conjecture \cite[Conjecture 3]{ravenel1998we} at weight $p$. 
 As usual, the height one situation provides valuable intuition:  
 
\begin{example}
For  $K=K_p^{\wedge}$  the $p$-adic complex $K$-theory spectrum, we have computed that, for $n>0$ even and $k$ odd, there is an isomorphism
$$K^{\wedge}_* \left(\Conf_p(\RR^n)_+ \otimes_{\Sigma_p} (S^k)^{\otimes p} \right) \ \  \cong \ \ \Sigma^{k-1}K_\ast/ p^{ \frac{n}{2}}.
$$
Since  $\Conf_p(\RR^n)_+ \otimes_{\Sigma_p} (S^k)^{\otimes p}$ is a  finite spectrum, \mbox{we in fact did not have to $K(1)$-localise; hence}  
$$K_* \left(\Conf_p(\RR^n)_+ \otimes_{\Sigma_p} (S^k)^{\otimes p} \right) \ \  \cong \ \ \Sigma^{k-1}K_\ast/ p^{ \frac{n}{2}}.
$$ 
This implies that there exists a cofibre sequence of $K$-module spectra
$$
\Sigma^{k-1}K \stackrel{p^{\frac{n}{2}}}{\longrightarrow} \Sigma^{k-1}K \longrightarrow K \otimes \Sigma^{\infty} \left(\Conf_p(\RR^n)_+ \otimes_{\Sigma_p} (S^k)^{\otimes p} \right).
$$
\begin{notation}  The Morava $K$-theory spectrum $K(1)$  is defined as the mod $p$ \mbox{$K$-theory} spectrum $K/p$ .  Note in particular our convention that $K(1)$ is a $2$-periodic theory, rather than the $(2p-2)$-periodic Adams summand. 
\end{notation}
  The map $p^{\frac{n}{2}}$  becomes $0$ after tensoring down to $K/p$, and so we obtain a cofibre sequence
$$
\Sigma^{k-1}K(1) \stackrel{0}{\longrightarrow} \Sigma^{k-1}K(1) \longrightarrow K(1) \otimes  \left(\Sigma^{\infty} \Conf_p(\RR^n)_+ \otimes_{\Sigma_p} (S^k)^{\otimes p} \right).
$$
We have therefore established the following special case of \cite[Theorem 4a]{langsetmo1993k}:
\begin{corollary} For $n$ even and $p$ odd, there are equivalences
$$K(1) \otimes  \left(\Sigma^{\infty} \Conf_p(\RR^n)_+ \otimes_{\Sigma_p} (S^k)^{\otimes p} \right) \simeq \Sigma^{k-1} K(1) \vee \Sigma^k K(1) \simeq K(1) \vee \Sigma K(1).$$\end{corollary}
\begin{remark}  
The $K(1)$-homology contains far less information that the \mbox{$p$-complete} $K$-theory. In particular, it is unable to distinguish 
 the spectra $\Conf_p(\RR^n)_+ \otimes_{\Sigma_p} (S^k)^{\otimes p}$ \mbox{for different even  $n$.}\end{remark}
\end{example}

We now turn to general heights $h$, as well as general parities of $n$ and $k$:

\begin{theorem}\label{Ktheoryatheightp}
Let $K(h)$ be the $2$-periodic Morava $K$-theory associated to a  Morava $E$-theory of height $h$ 
over a perfect field $k$ of characteristic $p>2$.    

The  $K(h)$-module \label{moravaKatp}
$K(h) \otimes  \left(\Sigma^{\infty} \Conf_p(\RR^n)_+ \otimes_{\Sigma_p} (S^k)^{\otimes p} \right)$  
is equivalent to 
$$     \begin{cases} 
    \ \ \    \displaystyle   \left(\Sigma^{k-1} K(h) \oplus \Sigma^{k} K(h) \right)^{ \oplus{\mathrm{min}\left(\frac{p^h-1}{p-1} ,\frac{n}{2}-1\right)}}     \oplus  \Sigma^{kp} K(h)   \oplus \Sigma^{pk+n-1} K(h)  \vspace{6pt} & \mbox{for } n \mbox{   even, } k \mbox{   even}\\
    \ \ \ \displaystyle \left(\Sigma^{k-1} K(h) \oplus  \Sigma^{k} K(h) \right)^{\oplus \mathrm{min}\left(\frac{p^h-1}{p-1} ,\frac{n}{2}\right)}  & \mbox{for } n \mbox{ even, } k \mbox{  odd} \vspace{6pt}\\
    \ \ \  \vspace{6pt}  \displaystyle \left(\Sigma^{k-1} K(h) \oplus \Sigma^{k} K(h) \right)^{\oplus {\mathrm{min}\left(\frac{p^h-1}{p-1} ,\frac{n-1}{2}\right)} }
 \oplus \Sigma^{kp} K(h) 
& \mbox{for } n \mbox{  odd, } k \mbox{  even}\\
    \ \ \  \displaystyle   \vspace{6pt}  \left(\Sigma^{k-1} K(h) \oplus \Sigma^{k} K(h) \right)^{\oplus \mathrm{min}\left(\frac{p^h-1}{p-1} ,\frac{n-1}{2}\right)}\oplus \Sigma^{k+ (2k+n-1)(\frac{p-1}{2})} K(h)  & \mbox{for } n  \mbox{  odd, } k \mbox{  odd}
\end{cases} $$ 

\end{theorem}

\begin{proof}
Let $E$ denote the form of Morava $E$-theory in question.  Our strategy will be to first understand the $E$-module spectrum
$E \otimes  \left(\Sigma^{\infty} \Conf_p(\RR^n)_+ \otimes_{\Sigma_p} (S^k)^{\otimes p} \right),$ 
and then tensor \mbox{down to $K(h)$.}

Choosing a $p$-typical coordinate, we may, as in Proposition \ref{explicit!}, write
$E^*(B\Sigma_p)/(\mathrm{tr}) \cong E_*[e]/f(e)$
for a monic   polynomial $f(e)=e^{\frac{p^h-1}{p-1}}+\cdots+p$  over $E_0$ with \mbox{$f(-e^{p-1})=\frac{[p](e)}{e}$ 
in $E^*(B\Sigma_p)$.}

Given $m \ge 0$,  
 let $A_m$ be  
 the free $E$-module of rank $m$ with \mbox{homotopy groups $\pi_*(A_m) \cong E_*[e]/e^{m}$.} 
Write $\phi_m:A_m \to A_m$ for the unique $E$-module map such that
$\pi_*(\phi_m):E_*[e]/e^{m} \to E_*[e]/e^{m}$  
is given by multiplication by $f(e)$.
Note that $\pi_*(\phi_m)$ is injective, as $f(e)$ having constant term $p$ implies that the corresponding matrix has non-vanishing determinant.
Letting $C_m$ denote the cofibre of $\phi_m$, it follows that 
$\pi_*(C_m) \cong E_*[e]/(f(e),e^m).$

By Theorem \ref{prop:general calculations}, the $E$-module spectrum   $E \otimes  \left(\Sigma^{\infty} \Conf_p(\RR^n)_+ \otimes_{\Sigma_p} (S^k)^{\otimes p} \right)$ is given by  
$$   \begin{cases}
  \ \ \ \      \Sigma^{kp} E   \oplus \Sigma^{pk+n-1} E  \oplus \Sigma^{k-1} C_{\frac{n}{2}-1}      & \mbox{for } n \mbox{   even, } k \vspace{3pt} \mbox{   even}\\
  \ \ \ \  \Sigma^{k-1} C_{\frac{n}{2}}  & \mbox{for } n \mbox{ even, } k \vspace{3pt} \mbox{  odd}\\
  \ \ \ \  \Sigma^{kp}E  \oplus \Sigma^{k-1} C_{\frac{n-1}{2}} & \mbox{for } n \mbox{  odd, } k\vspace{3pt} \vspace{3pt}  \mbox{  even}\\
  \ \ \ \  \Sigma^{k+ (2k+n-1)(\frac{p-1}{2})} E \oplus \Sigma^{k-1} C_{\frac{n-1}{2}} & \mbox{for } n \mbox{  odd, } k \mbox{  odd}
\end{cases}$$ 
Indeed, consider for example the case where $n$ and $k$ are both even.  We have calculated that
$$E^\wedge_\ast \left(\Conf_p(\RR^n)_+ \otimes_{\Sigma_p} (S^k)^{\otimes p} \right) \cong \Sigma^{kp}E_\ast \oplus \Sigma^{k-1} E^\ast(B\Sigma_p)/( \tr, e^{\frac{n}{2} -1} )  \oplus \Sigma^{pk+n-1} E_*.$$
Since $\left(\Conf_p(\RR^n)_+ \otimes_{\Sigma_p} (S^k)^{\otimes p} \right)$ is a finite $CW$-complex, the left hand side of the above is just
$$ \pi_*\left(E \otimes \left(\Conf_p(\RR^n)_+ \otimes_{\Sigma_p} (S^k)^{\otimes p} \right)\right).$$
There is a  quotient map of $E_*$-modules 
$$\Sigma^{kp}E_\ast \oplus \Sigma^{k-1} E^\ast(B\Sigma_p)/(e^{\frac{n}{2}-1} )  \oplus \Sigma^{pk+n-1} E_* \to     \Sigma^{kp}E_\ast \oplus \Sigma^{k-1} E^\ast(B\Sigma_p)/( \tr, e^{\frac{n}{2}-1} )  \oplus \Sigma^{pk+n-1} E_*,$$
Since the domain is a free $E_*$-module this map can be refined to a map of $E$-module  spectra
$$\Sigma^{kp}E \oplus \Sigma^{k-1} A_{\frac{n}{2}-1} \oplus \Sigma^{pk+n-1} E \to E \otimes  \left(\Sigma^{\infty} \Conf_p(\RR^n)_+ \otimes_{\Sigma_p} (S^k)^{\otimes p} \right).$$
The following composition is zero on the level of homotopy groups: 
$$
\begin{tikzcd}[column sep = huge]
\Sigma^{k-1} A_{\frac{n}{2}-1} \arrow{r}{\Sigma^{k-1} \phi_{\frac{n}{2}-1}} & \Sigma^{k-1} A_{\frac{n}{2}-1} \arrow{r} & E \otimes  \left(\Sigma^{\infty} \Conf_p(\RR^n)_+ \otimes_{\Sigma_p} (S^k)^{\otimes p} \right)
\end{tikzcd}
$$
Hence, it vanishes  as a map of  $E$-modules since the domain is free.  We see that the  map 
$\displaystyle \Sigma^{k-1} A_{\frac{n}{2}-1}   \to E \otimes  \left(\Sigma^{\infty} \Conf_p(\RR^n)_+ \otimes_{\Sigma_p} (S^k)^{\otimes p} \right)$
extends over $ \Sigma^{k-1} C_{\frac{n}{2}-1} $.  
 The resulting map
$$\Sigma^{kp}E \oplus \Sigma^{k-1} C_{\frac{n}{2}-1} \oplus \Sigma^{pk+n-1} E \longrightarrow E \otimes  \left(\Sigma^{\infty} \Conf_p(\RR^n)_+ \otimes_{\Sigma_p} (S^k)^{\otimes p} \right)$$  induces an isomorphism  on homotopy groups.
 
To finish the proof, it  remains  to compute $C_m \otimes_E K(h)$ for a general $m$, which can be done using the defining cofibre sequence
$A_m \stackrel{\phi_m}{\longrightarrow} A_m \longrightarrow C_m.$
Indeed, note that by \Cref{explicit!}  
 $$f(-e^{p-1}) = u\frac{[p](e)}{e} \equiv u'e^{p^h-1} \text{ mod } \mathfrak{m},$$
where $u, u'$  
are units and $\mathfrak{m} = (p, u_1,\ldots,u_{m-1})$ is the maximal ideal in $E_0$.   
Modulo $\mathfrak{m}$,   
 $f(e)$ is a unit times $e^{\frac{p^h-1}{p-1}}$, and we obtain an equivalence 
$C_m \otimes_E K(h) \simeq \bigoplus_{\text{min}(\frac{p^h-1}{p-1} ,m)} K(h) \vee \Sigma K(h).\vspace{-10pt}$ 
\end{proof}
 
As $h\rightarrow \infty$, we recover a classical result of Cohen:
\begin{corollary}  
The $\FF_p$-vector space 
$\displaystyle H_\ast \left(\Conf_p(\RR^{2m})_+ \otimes_{\Sigma_p} (S^1)^{\otimes p};\FF_p \right)$
has \mbox{dimension $2m$.}

\begin{proof} Let $K(h)$ be a Morava $K$-theory  with $K(h)_\ast\cong \FF_p[v_n^{\pm 1}]$. 
For $h\gg 0$ very large, the Atiyah--Hirzebruch  spectral sequence of the finite space $(\Conf_p(\RR^{2m})_+ \otimes_{\Sigma_p} (S^1)^{\otimes p}$ degenerates, which leads to an 
isomorphism of  $K(h)$-modules
$$K(h)_\ast \left(\Conf_p(\RR^{2m})_+ \otimes_{\Sigma_p} (S^1)^{\otimes p} \right) \cong H_\ast \left(\Conf_p(\RR^{2m})_+ \otimes_{\Sigma_p} (S^1)^{\otimes p};\FF_p \right)[\beta^{\pm 1}].\vspace{-10pt}$$ 
\end{proof}

\begin{remark}\label{MSG}
This result agrees with the  work of Cohen  \cite[III]{CohenLadaMay:HILS},  which implies  that $$H_\ast \left(\Conf_p(\RR^{2m})_+ \otimes_{\Sigma_p} (S^1)^{\otimes p};\FF_p \right)$$  has a basis given by $\{\beta^{\epsilon} Q^s\ | \ 1 \leq s \leq m, \epsilon \in \{0,1\}\}$. Cohen's result is stronger, as it also gives the dimensions of individual homology groups, while our method only sees the even and odd degree, respectively. We believe that this  deficiency can be removed by studying $2(p^h-1)$-periodic versions of $K(h)$ using the prime-to-$p$ subgroup of the Morava stabiliser group, which acts on our spectral sequence. We will  not pursue \mbox{this approach here.}
\end{remark}

\end{corollary} 
We briefly compare this with several previously known results mentioned in the introduction.

\subsection*{Height one} 
Langsetmo   
computed the $K(1)$-homology of iterated loop spaces  \cite{langsetmo1993k,MR1397734}, utilising 
the work of  Mahowald--Thompson  \cite{MR1153241} and    \mbox{McClure \cite[Chapter 9]{MR836132}.}
 
Write  $T^{a}(x_1,x_2,\ldots)$ \mbox{for the  truncated polynomial algebra  $\FF_p[x_1, x_2, \ldots]/(x_1^{a}, x_2^{a}, \ldots )$. Then:} 
\begin{theorem}[Langsetmo] \label{Lang1} For $h=1$ and $p$ odd, there are isomorphisms
$$ \ \ \ \ \ \   \  \ \ K(1)_\ast( \ \Omega^{2m} \  S^{2m+1})  \ \cong K(1)_\ast \otimes \Lambda(u_i \ | \ i\leq m) \otimes T^{p^m}(x_i \ | \ i \geq 1  )$$
The $u_i$ have odd and the $x_i$ even degrees, respectively. Elements with index $a$ live in \mbox{weight $p^a$.}
\end{theorem}

For all integers $m\geq 1$, our \Cref{moravaKatp} gives an isomorphism between the $K(1)$-homology of the $p^{th}$ Snaith summand of $\Omega^{2m}S^{2m+1}$
and $ K(1)_\ast \oplus  \Sigma  K(1)_\ast $
These two copies  correspond precisely to the classes $u_1$ and $x_1$ in  Langsetmo's result.

\subsection*{Double loop spaces}

The Morava $K$-theory of double loop spaces was computed at all heights by Yamaguchi (cf. \cite{yamaguchi1988moravak}), using the Atiyah--Hirzebruch spectral sequence:
\begin{theorem}[Yamaguchi]\label{yamaguchi}
Let $p$ be an odd prime. For all heights $h\geq 0$, the Atiyah--Hirzebruch spectral sequence collapses and 
gives rise to an isomorphism
$$ K(h)_\ast(\Omega^2 S^{2m+1}) \cong K(h)_\ast \otimes  \Lambda(u_a \ |  \ a \leq h) \otimes T^{p^h}(x_i \ | \ i\geq 1).$$
The $u_i$ have odd and the $x_i$ even degrees, respectively. Elements with index $e$ live in \mbox{weight $p^e$.}
\end{theorem}
\Cref{moravaKatp} asserts an isomorphism between the $K(h)$-homology of the $p^{th}$ Snaith summand of $\Omega^2 S^{2r+1}$ and  $ K(h)_\ast \oplus \Sigma  K(h)_\ast$; these two copies of $K(h)_\ast$ correspond to $y_1$ and $x_0$, respectively.

\subsection*{Triple loop spaces}
The case of triple loop spaces is  
  more delicate, and the problem has an interesting history in this case.  
For some time, it was 
 believed that the Atiyah--Hirzebruch  spectral sequence would degenerate (cf. \cite{MR928218}), 
which led to a  faulty computation and a surpring, yet incorrect, counterexample (cf. \cite{MR1022687}) to a conjecture of Miller--Snaith  \cite{MR551463} and Mahowald--Ravenel \cite{MR893848} (which later became a theorem of \mbox{Bousfield in \cite[Theorem 14.8]{MR1257059}).}

This state of affairs was resolved by Tamaki, who used a version of the  Eilenberg--Moore  spectral sequence \cite{MR1293304} to climb   from double to triple loop spaces \mbox{(cf. \cite[Proposition 6.1]{tamaki2002fiber}):}

\begin{theorem}[Tamaki]\label{tamaki} 
If $p$ is an odd prime, then the Eilenberg--Moore  spectral sequence degenerates and gives rise to an isomorphism
$$K(h)_\ast(\Omega^3 S^{2m+1}) \cong K(h)_\ast \otimes 
P(u_a \ | \ a \leq h)\otimes 
\Lambda(x_{i,j} \ | \ i+j \leq h, i>0) \otimes 
\bigotimes_{1\leq j \leq h-1}
T^{p^{h-j}}(y_{i,j} \  |  \ i > 0  )$$
where $u_i$ as Snaith weight $p^i$ and the elements $x_{i,j}$ and $y_{i,j}$ have   weight $p^{i+j}$.
\end{theorem}
Our \Cref{moravaKatp}  shows that the $K(h)$-homology of the $p^{th}$ Snaith summand in $\Omega^3 S^{2m+1} $ is
$ K(h)_\ast \oplus \Sigma  K(h)_\ast  \oplus \Sigma  K(h)$; this corresponds to the  
classes $u_0^p$,  $u_1$, and $x_{1,0}$ in Snaith \mbox{weight $p$.}

\subsection*{A remark on a conjecture of Ravenel} 
We will briefly comment on Ravenel's  \cite[Conjecture 3]{ravenel1998we}, which concerns the $K(h)$-homology of iterated loop spaces of spheres. 
Ravenel’s conjecture suggests that for a fixed number of loops $n$,  the weight $p$ part of $K(h)_\ast (\Omega^{n} S^{n+1})$ 
 does not change in size as we vary the chromatic height $h$.

This is true whenever $n\leq 3$, as in this case, the multiplicities 
$\mathrm{min}(\frac{p^h-1}{p-1} ,\frac{n}{2}) $ (for $n$ even) and  $\mathrm{min}(\frac{p^h-1}{p-1} ,\frac{n-1}{2}) $ (for $n$ odd) appearing in  \Cref{Ktheoryatheightp}
are  both equal to $1$ for all heights $h\geq 1$. 

However, this pattern does not persist for general $n$, where Ravenel's conjecture does not seem to be true. Indeed,  
 \Cref{Ktheoryatheightp} shows that for $n\geq 4$, the size of $K(h)_\ast (\Omega^{n} S^{n+1})(p)$ depends on $h$. This can in fact already be observed from Langsetmo's \Cref{Lang1}.

\newpage 

\section{Configurations of $p$ points in surfaces}\label{surfaces} 
In this section, we  
 examine the $E$-theory and $\FF_p$-homology of unordered configuration spaces of points on the 
once-punctured orientable surface $\mathcal{S}_{g,1}$ of genus $g \geq 0$, a framed manifold.  

\subsection{The $E$-theory of $B_p(\mathcal{S}_{g,1}$)}
Let $E$ be an $E$-theory of height $h$ over a perfect field $k$ of characteristic $p>2$.
We   compute the $E$-theory of  $B_p(\mathcal{S}_{g,1}) = \Conf_p(\mathcal{S}_{g,1})/\Sigma_p$ at all  heights:   
\begin{theorem}[$E$-theory, surface case]\label{thm:open surfaces} \label{formulabetti} \ 

\begin{enumerate}[wide, labelwidth=!, labelindent=0pt] 
\item The $E$-cohomology of the unordered configuration space of $p$ points in the \mbox{punctured torus is}  
$$E^\ast(B_p(\dot{T})) \cong  \bigg(\bigoplus_{0 \leq i<p} \Sigma^i E_\ast^{\oplus \lfloor \frac{3i+2}{2}\rfloor} \bigg)
\oplus
 \Sigma^p E_\ast^{\oplus (p+1)} .  $$

\item More generally, the $E$-cohomology of the unordered configuration space of $p$ points in the punctured orientable genus $g$ surface $\mathcal{S}_{g,1}$   is given by  \vspace{3pt}
$$E^\ast(B_p(\mathcal{S}_{g,1})) \cong   \bigoplus_{0 \leq i\leq p} \Sigma^i E_\ast^{\oplus \beta_i}. \vspace{3pt}$$
For $ 0< i< p$, we  \vspace{3pt}have
$$ \beta_i = \hspace{-3pt} \sum_{\substack{0 \leq j \leq g\\ \\  j\equiv i \hspace{-6pt}\mod 2}}\hspace{3pt} \left(  {2g \choose j} \hspace{-3pt} - \hspace{-3pt}{2g \choose j-2} \right) {2g + \frac{i-j}{2} -1 \choose  2g-1}   +
\sum_{\substack{g+1 \leq j \leq 2g+1\\ \\  j\equiv i  \hspace{-6pt}\mod 2}}\hspace{-3pt} \left(  {2g \choose j-1} \hspace{-3pt}- \hspace{-3pt}{2g \choose j+1} \right) {2g + \frac{i-j}{2} -1 \choose  2g-1} 
 \vspace{5pt}$$
For $i=p$, we  \vspace{-3pt}have$$ \beta_p = \sum_{\substack{0 \leq j \leq g\\ \\  j\equiv i  \hspace{-6pt}\mod 2}} \left(  {2g \choose j} - {2g \choose j-2} \right) {2g + \frac{i-j}{2} -1 \choose  2g-1}   .
 $$
\end{enumerate}
\end{theorem}
To apply \Cref{thm:cohomology}, we  
determine the Hecke Lie algebra \mbox{$\mathfrak{g}(\mathcal{S}_{g,1}, S^0) = E_*^{\wedge} (\Free_{\Lie}(S^{-1})^{\mathcal{S}_{g,1}} ).$}  
In the height  one case,  this Hecke Lie algebra is especially simple:
 
\begin{lemma}\label{K-surface-HLA}
At height $1$, the underlying $K_*$-module of $\mathfrak{g}(\mathcal{S}_{g,1};S^0)$ is free on $8g+4$ generators: \begin{center}\def\arraystretch{1.3}
\begin{tabular}{c|cccc}
$i\backslash  weight $ &$1$ & $2$ & $p$&$2p$\\
\hline
$-2$&$0$ & $0$ &$\substack{cy}$&$0$
\\
$-1$&$\substack{cx}$ &$0$  &$\substack{a_1y, \ldots , a_gy \\  b_1y, \ldots , b_gy}$&$\substack{c\widetilde y}$
\\
$0$&$\substack{a_1x, \ldots , a_gx \\  b_1x, \ldots , b_gx }$ & $\substack{c\widetilde x}$  &$0$&$ \substack{a_1\widetilde y, \ldots , a_g\widetilde y \\  b_1\widetilde y, \ldots , b_g\widetilde y }$\\
$1$&$0$&$\substack{a_1\widetilde x, \ldots , a_g\widetilde x \\  b_1\widetilde x, \ldots , b_g\widetilde x }$&$0$&$0$
\end{tabular}
\end{center} 
Here $i$ denotes the internal degree. The Hecke module structure is determined by the equations $$\alpha(ex)=\begin{cases}
ey&\quad \mbox{ for }  e=a_1,\ldots a_{g}, b_1,\ldots , b_g \\
pey&\quad \mbox{ for } e=c
\end{cases}$$ $$\ \ \ \  \ \ \  \alpha(e\widetilde x)=pe\widetilde y,\ \ \ \ \ \ \ \mbox{ for } e=a_1,\ldots,a_{g}, b_1, \ldots, b_g, c.\ \ \ $$
All other Hecke operations vanish. The only nonzero components of the Lie bracket are given by $\left[ a_ix, b_ix\right]=-c\widetilde x$ for $i=1,\ldots , g$.
\end{lemma} 
\begin{proof}

We recall three well-known facts about $\mathcal{S}_{g,1}$, the closed oriented surface of genus $g$:
\begin{enumerate} 
\item The cohomology $\widetilde E^{-*} (\mathcal{S}_{g,1})$ is free on  generators $a_1,\ldots, a_g$, $b_1,\ldots,b_g$ (in homological degree $-1$)  and $c$ (in homological degree $-2$).
\item We have $a_ib_i = -b_i a_i =c$ . The remaining products of the generators vanish.
\item The standard cell decomposition of $\mathcal{S}_{g,1}$ is stably split.
\end{enumerate}
By \Cref{freelieone}, the free Hecke Lie algebra $\LL(K_*(S^{1}))$
on a class in degree $1$ is generated as $K_\ast$-module by four classes
$$x\in \LL(K_*(S^{1}))_{1} \ , \ \ y\in \LL(K_*(S^{1}))_{0} \ , \ \
\widetilde{x}\in \LL(K_*(S^{1}))_{2} \ , \ \ \widetilde{y}\in \LL(K_*(S^{1}))_{1}$$
The Lie bracket satisfies $[x,x]=\widetilde{x}$ and vanishes otherwise; the Hecke operations satisfy $$\alpha(x)=y \ , \ \  \alpha(\widetilde{x})=\widetilde{y}\ , \ \ \alpha(\widetilde{x})=\alpha{\widetilde{y}}=0.$$
Combining facts $(1)$ and $(2)$ above with \Cref{prop:bracket}, we see that  the underlying $E_\ast$-module of 
$\mathfrak{g}(\dot{T};S^0)$ is generated by the $8g+4$ indicated elements, and that the Lie bracket behaves as claimed. The third fact allows us to apply \Cref{Gamma0} and   \Cref{lem:stably split} to  determine the structure
of $\mathfrak{g}(\dot{T};S^0)$ as a Hecke module.
\end{proof} 

We immediately generalise this claim to higher heights:

\begin{lemma}\label{E-surface-HLA}
The underlying $E_*$-module of $\mathfrak{g}(\mathcal{S}_{g,1};S^0)$ is given by
$$ E_\ast \langle  {a_1,\ldots,a_g,  b_1,\ldots,b_g , c}  \rangle \otimes_{E_\ast} \left( \bigoplus_{0\leq \ell \leq h} (\mathcal{H}_u^{\Lie})_{1}^{\ast}(p^\ell)  \oplus
\bigoplus_{0\leq \ell \leq h} (\mathcal{H}_u^{\Lie})_{2}^{\ast}(p^\ell) \right) $$
Write elements corresponding to $\phi\in (\mathcal{H}_u^{\Lie})_{1}^{\ast}(p^\ell)$  or  $\phi\in (\mathcal{H}_u^{\Lie})_{2}^{\ast}(p^\ell)$
as $x_\phi$ or 
 $\widetilde{x}_{\phi}$, respectively. \\
Elements  $e \otimes x_\phi$ with $e\in \{a_1,\ldots,a_g,b_1,\ldots b_g,c\}$ and  $\phi\in (\mathcal{H}_u^{\Lie})_{1}^{\ast}(p^\ell)$ have weight $p^\ell$.\\
Elements  $e \otimes \widetilde{x}_\phi$ with $e\in \{a_1,\ldots,a_g,b_1,\ldots b_g,c\}$ and  $\phi\in (\mathcal{H}_u^{\Lie})_{2}^{\ast}(p^\ell)$ have weight $2p^\ell$.

If $\phi\in (\mathcal{H}_u^{\Lie})_{1}^{\ast} $, then 
$$\alpha(e\otimes x_{\phi})=\begin{cases}
e\otimes x_{\Susp(\alpha)\cdot \phi}&\quad \mbox{ for } e=a_1,\ldots a_{g}, b_1,\ldots , b_g  \mbox{ and }  \alpha \in  (\mathcal{H}_u^{\Lie})_{\ast-1}^{j} \\
e\otimes x_{\Susp^2(\alpha)\cdot \phi}&\quad \mbox{ for } e=c \mbox{ and }  \alpha \in  (\mathcal{H}_u^{\Lie})_{\ast-2}^{j}
\end{cases}$$

If $\phi\in (\mathcal{H}_u^{\Lie})_{2}^{\ast} $, then 
$$\alpha(e\otimes \widetilde{x}_{\phi})=\begin{cases}
e\otimes\widetilde{x}_{\Susp(\alpha)\cdot \phi}&\quad \mbox{ for } e=a_1,\ldots a_{g}, b_1,\ldots , b_g  \mbox{ and }  \alpha \in  (\mathcal{H}_u^{\Lie})_{\ast-1}^{j} \\
e\otimes \widetilde{x}_{\Susp^2(\alpha)\cdot \phi}&\quad \mbox{ for } e=c \mbox{ and }  \alpha \in  (\mathcal{H}_u^{\Lie})_{\ast-2}^{j}
\end{cases}$$
The Lie bracket satisfies $[a_i \otimes x_\lambda,b_i \otimes x_\mu] =- c \otimes  \widetilde{x}_{\lambda \mu}$  for all scalars $\lambda, \mu \in E_\ast \cong  (\mathcal{H}_u^{\Lie})_{1}^{\ast}(1)$ and all $i=1,\ldots g$. It vanishes otherwise.

\end{lemma}
\begin{proof}
We proceed as in the proof of \Cref{K-surface-HLA}:  the structure of $\mathfrak{g}(\mathcal{S}_{g,1};S^0)$ as an $E_\ast$-Lie algebra is implied by 
\Cref{freeliehigher}  together with the   well-known ring structure of $E^\ast(\mathcal{S}_{g,1})$.
Moreover,  \Cref{lem:stably split}  determines the structure of $\mathfrak{g}(\mathcal{S}_{g,1};S^0)$ as a Hecke module.
\end{proof}
With this description of $\mathfrak{g}(\mathcal{S}_{g,1};S^0)$ at hand, we can now prove the main result of this section.
\begin{proof}[Proof of \ref{thm:open surfaces}]
Recall that \Cref{thm:cohomology} gives a  cohomological spectral sequence  \mbox{with signature}
\[{\mathrm{E}_2^{s,t}} \cong H^{s+1}\left( \CE_{{\mathcal{H}_u}}\left(\mathfrak{g}(\mathcal{S}_{g,1};S^0) \right)^\vee \right)_{t+1} \implies \bigoplus_k E^{t-s}(B_k(\mathcal{S}_{g,1};S^0)).\] 
In the interest of readability, we present the argument    at height $h=1$; we will then indicate   how it 
generalises to higher heights in a second step.

Recall the description of $\mathfrak{g}(\mathcal{S}_{g,1};S^0)$ given in \Cref{K-surface-HLA}.
We depict the additive resolution $\AR(\mathfrak{g}(\mathcal{S}_{g,1};S^0))$ (defined in \Cref{construction:additive resolution}) of the relevant Hecke Lie \mbox{algebra up to weight $p$:}
\vspace{-10pt}
\begin{figure}[H]
\begin{small}
\begin{center}\def\arraystretch{1.5}
\begin{tabular}{c|cccc}
$r\backslash weight$ &$1$ & $2$ & $p$&$\cdots$\\
\hline
$0$&$\substack{\ \\ \ \\ \    [cx]_{_{(-1,1,0,1)}}\\  [a_i x]_{_{(0,1,0,1)}} \\
   [b_i x]_{_{(0,1,0,1)}} 
}$ & \ \ \ \ \ $\substack{\ \\ \ \\   [c\widetilde{x}]_{_{(0,1,0,2)}}\\ [a_i \widetilde{x}]_{_{(1,1,0,2)}}\\
   [b_i  \widetilde{x}]_{_{(1,1,0,2)}}
}$ & $\substack{\ \\ \ \\  \ \ \ [cy]_{_{(-2,1,0,p)}}\\   \ \ \ [a_i y]_{_{(-1,1,0,p)}}\\
 \ \ \  [b_i y]_{_{(-1,1,0,p)}} 
}
$  \ \ \ \ \ &$\cdots$\ \vspace{-9pt}
\\\\
$1$&$\substack{\ \\ \ \\   [1|cx]_{_{(-1,1,1,1)}}\\   [1|a_i x]_{_{(0,1,1,1)}}\\
   [1|b_i  x]_{_{(0,1,1,1)}}
}
$ & \ \ \ \ \  $\substack{\ \\ \ \\  [1|c\widetilde{x}]_{_{(0,1,1,2)}}\\  [1|a_i \widetilde{x}]_{_{(1,1,1,2)}} \\
 [1|b_i \widetilde{x}]_{_{(1,1,1,2)}} 
}
 $  &$\substack{ [1|cy]_{_{(-2,1,1,p)}}\\   [1|a_i y]_{_{(-1,1,1,p)}}\\
   [1|b_i y]_{_{(-1,1,1,p)}} \\ \  \\
  \hspace{-3pt} [\alpha| cx]_{_{(-2,1,1,p)}}\\  [\alpha | a_i x]_{_{(-1,1,1,p)}} \\
   [\alpha|b_i x]_{_{(-1,1,1,p)}} 
}$

&$\cdots$  
\\
\vdots&\vdots&\vdots&\vdots
\end{tabular}\vspace{-10pt}
\end{center} 
\end{small}
\end{figure}
Thus, basis elements of weight $p$ in $ \CE_{{\mathcal{H}_u}}\left(\mathfrak{g}(\mathcal{S}_{g,1};S^0) \right)$ fall into two different classes:
\begin{enumerate} 
\item weight $p$ elements in the additive resolution (these belong to  $\Gamma^1_{K_\ast}\left(\AR(\mathfrak{g}(\mathcal{S}_{g,1};S^0) )\right)$)
\item products or divided powers of elements in weight $1$ or $2$ in the additive resolution (these belong to  $\Gamma^{>1}_{K_\ast}\left(\AR(\mathfrak{g}(\mathcal{S}_{g,1};S^0) )\right)$).
\end{enumerate}
The simplicial structure maps and the   differential send elements in each class to linear combinations of  
 elements in the same class. This gives a decomposition of simplicial chain complexes 
$$\CE_{{\mathcal{H}_u}}\left(\mathfrak{g}(\mathcal{S}_{g,1};S^0) \right)(p) = \CE_1(p) \oplus \CE_2(p).$$
The explicit description of the Hecke action on $\mathfrak{g}(\mathcal{S}_{g,1};S^0)$ in 
\Cref{K-surface-HLA} shows that the normalised chain complex of $\CE_1(p)$ is equivalent to  $ \ldots \rightarrow 0 \rightarrow  [\alpha | cx]   \xrightarrow{p}    [cy]\rightarrow    0 $. \mbox{Hence, we have}
$$H^{s+1}(\CE_1(p)^\vee)_{\ast+1} = \begin{cases}  \Sigma K_\ast/p & \mbox{ if } s=1 \\ 0 & \mbox{ else} \end{cases} . $$

The second summand  $\CE_2(p)$  can be understood by following the rational computations in \cite{BoedigheimerCohen:RCCSS},  \cite[Section 6.2]{Knudsen:BNSCSVFH}, and \cite[Section 3.2]{DrummondColeKnudsen:BNCSS}. More precisely, consider
the  ordinary Lie algebra $\mathfrak{h}$ in $\Mod_{K_*}^\NN$ with underlying module 
  generated by $2g+1$ classes $a_1 x, \ldots, a_g x, b_1x,\ldots, b_gx$ (in degree $0$ and weight $1$) and an additional  class   $c \widetilde{x} $ (in degree $0$ and weight $2$).
The only non-vanishing   Lie brackets of generators are given by 
$[a_i x, b_i x] = -c\widetilde{x}$ for $1\leq i \leq g$
Observe that $\CE_2(p)$ is   isomorphic to the weight $p$ component of the Chevalley--Eilenberg complex of   $$\mathfrak{h} \oplus \triv(K_*\langle a_i \widetilde{x}, b_i\widetilde{x}, cx \ | \ 1\leq i \leq g \rangle ),$$
where the right summand denotes a trivial Lie algebra on  $2g+1$ classes $a_1 \widetilde{x}, \ldots, a_g \widetilde{x}, b_1\widetilde{x},\ldots, b_g\widetilde{x}$ (in degree $1$ and weight $2$) and an additional  class   $c  {x} $ (in degree $-1$ and weight $1$). 

The Lie algebra \textit{co}homology of $ \triv(K_*\langle a_i \widetilde{x}, b_i\widetilde{x}, cx \ | \ 1\leq i \leq g \rangle ) $ is therefore given by a polynomial algebra $P_{K_\ast}(\widetilde{\alpha}_1,\ldots,\widetilde{\alpha}_g, \widetilde{\beta}_1, \ldots, \widetilde{\beta}_g,   {\gamma})$, where the classes  $\widetilde{\alpha}_i, \widetilde{\beta}_i$ sit in degree $-2$ and \mbox{weight $2$},  and ${\gamma}$ lives in degree $0$ and weight $1$. Here we have used that dualising divided powers of free $K_*$-modules gives symmetric powers.

Since $\mathfrak{h}$ sits in even degree,
 the Chevalley--Eilenberg complex does not contain any divided powers, and the cohomology $H^*(\CE(\mathfrak{h}^\vee)$ agrees with the homology of the complex
  $$( \Lambda[ {\alpha}_1,\ldots, {\alpha}_g,  {\beta}_1, \ldots,  {\beta}_g,\widetilde{\gamma}],d),$$
with differential satisfying $d(\widetilde{\gamma}) = -2(\alpha_1  {\beta}_1 + \ldots +\alpha_g  {\beta}_g)$ and vanishing otherwise.
Here  ${\alpha}_i, {\beta}_i$ sit  in degree $-1$ and weight $1$,  and  $\widetilde{\gamma}$ sits in degree $-1$ and weight $2$. The cohomology of this complex is computed in \cite[Theorem D]{BoedigheimerCohen:RCCSS},
We can therefore describe the $\mathrm{E}_2$-term in \mbox{weight $p$:}
$$
\mathrm{E}_2^{s,\ast}(p)  \cong \left( H_{-s-1} ( \Lambda[ {\alpha}_1,\ldots, {\alpha}_g,  {\beta}_1, \ldots,  {\beta}_g,\widetilde{\gamma}],d)  \otimes  P_{K_\ast}(\widetilde{\alpha}_1,\ldots,\widetilde{\alpha}_g, \widetilde{\beta}_1, \ldots, \widetilde{\beta}_g,   {\gamma})\right)(p)  \oplus \ \ \Sigma^{1}K_\ast[1].$$

Picking an embedding $\RR^2 \cong D^2 \rightarrow \mathcal{S}_{g,1}$ from an open disc into the punctured surface of \mbox{genus $g$}, we obtain a map of configuration spaces $\Conf_p(\RR^2) \rightarrow \Conf_p( \mathcal{S}_{g,1})$. Stably, this map corresponds to the map of spectral Lie algebras $\Free^{\mathscr{L}}(S^{-1})^{S^2} \rightarrow \Free^{\mathscr{L}}(S^{-1})^{\mathcal{S}_{g,1}} $ induced by the collapse map $\mathcal{S}_{g,1} \rightarrow S^2$.

 We obtain maps of cohomological spectral sequences 
$\mathrm{E}_r^{s,t}(\mathcal{S}_{g,1}) \rightarrow \mathrm{E}_r^{s,t}(\RR^2)$
from the spectral sequence computing $E^\ast(B_p(\mathcal{S}_{g,1}))$  to the corresponding sequence for $B_p(\RR^2)$, which we have studied in great detail in the preceding section.

On the $\mathrm{E}_1$-page, the map $\mathrm{E}_1^{s,t}(\mathcal{S}_{g,1}) \rightarrow \mathrm{E}_1^{s,t}(\RR^2)$ is obtained by applying the cohomological Hecke Chevalley--Eilenberg  complex to the map of Hecke Lie algebras
$$\mathfrak{g}(\RR^2, S^0)   \cong  \widetilde{E}^\ast(S^2) \otimes_{\EE_2} \LL(x_{1}) \longrightarrow \widetilde{E}^\ast(\mathcal{S}_{g,1}) \otimes_{\EE_2} \LL(x_{1}) \cong \mathfrak{g}(\mathcal{S}_{g,1}, S^0).$$
Under this map, the classes $[x]$ and $[y]$ on the left (cf. \Cref{warmup} and its proof) are sent to  the classes $[cx]$ and $[cy]$ on the right.
As in the proof of \Cref{warmup},  the   class dual to $[x]$ kills the torsion class dual to $[y]$ in the spectral sequence $\mathrm{E}_r^{s,t}(\RR^2)$,  which implies that the class dual to  $[cy] $ must have died on the $p^{th}$ page of the spectral sequence for $\mathcal{S}_{g,1}$.

By \Cref{thm:free part}, all nontrivial $d_r$-differentials with $r\geq 2$ must have torsion targets.  This implies that the pages $\mathrm{E}_2^{s,t}(\mathcal{S}_{g,1})$, \ldots,  $\mathrm{E}_{p-1}^{s,t}(\mathcal{S}_{g,1})$ have at most one torsion class $[cy]$, and that  $\mathrm{E}_{r}^{s,t}(\mathcal{S}_{g,1})$ is torsion-free for all $r\geq p$. This implies that the spectral sequence degenerates on the  $\mathrm{E}_{p}$-page, with  $\mathrm{E}_{p}^{s,t}(\mathcal{S}_{g,1})\cong \mathrm{E}_{\infty}^{s,t}(\mathcal{S}_{g,1})$ given by \begin{equation} \left( H_{-s-1} ( \Lambda[ {\alpha}_1,\ldots, {\alpha}_g,  {\beta}_1, \ldots,  {\beta}_g,\widetilde{\gamma}],d)  \otimes  P_{K_\ast}(\widetilde{\alpha}_1,\ldots,\widetilde{\alpha}_g, \widetilde{\beta}_1, \ldots, \widetilde{\beta}_g,   {\gamma})\right)(p). \label{betsi}\end{equation}
Since this is a free $K_\ast$-module, there are no extension problems. The claim for $K$-theory therefore follows by computing 
the dimensions in \eqref{betsi} explicitly, which is done in \cite[Section 3.2]{DrummondColeKnudsen:BNCSS}.

 The argument generalises to $E$-cohomology by starting with \Cref{E-surface-HLA} instead of \Cref{K-surface-HLA}. The important facts needed in the proof are:
\begin{itemize}
\item Up to scaling, Hecke operations of weight $p$ shift degree down by $1$ (cf. \Cref{Heckepowerring})
\item The suspension  $(\mathcal{H}^{\Lie}_u)_{i}^{i-1}(p) \rightarrow (\mathcal{H}^{\Lie}_u)_{i-1}^{i-2}(p)$ is an isomorphism for all $i$ even (cf. \Cref{Gamma0}), and has cokernel $E_\ast/p$ for $i$ odd; this  gives a description of $\CE_1(p)$ as $E_*/p$.
\item The $E$-cohomology of $B_p(\RR^2)$ is torsion-free (cf. \Cref{prop:general calculations}).
\end{itemize}
\end{proof}

 \subsection{The $\FF_p$-homology of $B_p(\mathcal{S}_{g,1}$)}
We close this paper with an application in classical topology which, to the best of our knowledge, is new: 
\begin{proposition}[$\FF_p$-homology, surface case]\label{Fpsurface}
Let $p$ be an odd prime. 
The $\FF_p$-homology groups  of the unordered configuration space of $p$ points in the punctured torus satisfy  
$$H_{even}(B_p(\dot{T}));\FF_p) = \bigoplus_i H_{2i}(B_p(\dot{T});\FF_p )   \cong \bigoplus_{0 \leq i<p \mbox{ \small{even} }} \FF_p^{\oplus  \frac{3i+2}{2} }   . \vspace{3pt}$$
$$\  \ \ \ \ \  \ \ \ \ \  \ \ \ \  \ \ \ H_{odd}(B_p(\dot{T}));\FF_p )   = \bigoplus_i H_{2i+1}(B_p(\dot{T});\FF_p )   \cong   \bigg(  \bigoplus_{0 \leq i<p \mbox{ \small{odd} }} \FF_p^{\oplus  \frac{3i+1}{2} } \bigg)
\oplus
 \FF_p^{\oplus (p+1)} . \vspace{3pt}$$  
More generally, the $\FF_p$-homology of the unordered configuration space of $p$ points in a punctured orientable genus $g$ surface $\mathcal{S}_{g,1}$  satisfies: \vspace{-3pt}

$$ \ \ H_{even}(B_p(\mathcal{S}_{g,1}));\FF_p) := \bigoplus_i H_{2i}(B_p(\mathcal{S}_{g,1}));\FF_p) \cong \bigoplus_{0 \leq i<p \mbox{ \small{even} }} \FF_p^{\oplus \beta_i }   . \vspace{3pt}$$
$$  \  \ \  \ H_{odd}(B_p(\mathcal{S}_{g,1}));\FF_p )   := \bigoplus_i H_{2i+1}(B_p(\mathcal{S}_{g,1}));\FF_p)   \cong    \ \bigoplus_{0 \leq i\leq p \mbox{ \small{odd} }} \FF_p^{\oplus \beta_i }   ,\vspace{3pt}$$
 where the numbers $\beta_i$ are defined as in \Cref{thm:open surfaces}.
\end{proposition}  
\begin{proof} Since $B_p(\mathcal{S}_{g,1})$ is a finite complex, we can pick $h\gg 0$ sufficiently large  that the $K(h)$-based Atiyah--Hirzebruch spectral sequence degenerates; this implies \begin{equation} \label{AHSSdeg} K(h)_\ast(B_p(\mathcal{S}_{g,1})) \cong H_\ast(B_p(\mathcal{S}_{g,1});\FF_p)[\beta^{\pm 1}].\end{equation}  
Here $K(h)$ is  the $2$-periodic Morava $K$-theory attached to the  $E$-theory $E_h$ 
corresponding to the height $h$ Honda formal group law over  the field $\FF_p$.

Writing $\DD$ for the Spanier--Whitehead duality functor in spectra, we can also use the finiteness of $B_p(\mathcal{S}_{g,1})$ to obtain  an equivalence
\begin{equation} K(h)^{B_p(\mathcal{S}_{g,1})_+} \simeq  K(h)\otimes \DD(B_p(\mathcal{S}_{g,1})_+) \simeq K(h) \otimes_{E_h} (E_h\otimes \DD(B_p(\mathcal{S}_{g,1})_+)\label{tensordown} \end{equation}
By \Cref{thm:open surfaces}, we know that the $E$-module $E_h\otimes \DD(B_p(\mathcal{S}_{g,1})_+) \simeq E_h^{B_p(\mathcal{S}_{g,1})_+}$ is a direct sum of 
$ \beta_{ \mbox{ \small{even} }}:= \bigoplus_{i<p \mbox{ \small{even} }}   \beta_i $ many copies of $E$ and $\beta_{ \mbox{ \small{odd} }}:=\bigoplus_{i\leq p \mbox{ \small{odd} }}  \beta_i   $ many copies of $\Sigma E$. Expression \eqref{tensordown} then implies that  $K(h)^\ast(B_p(\mathcal{S}_{g,1}))$ consists of $\beta_{ \mbox{ \small{even} }}$ many  copies of $K(h)_\ast$ and  $\beta_{ \mbox{ \small{odd} }}$ many   copies of $\Sigma K(h)_\ast$. 
Since $K(h)$ is a generalised field, the corresponding claim holds for $K(h)$-based homology, and \eqref{AHSSdeg}  implies the result.
\end{proof}

\begin{proof}[Proof of Theorem \ref{thm:no torsion}]
By the universal coefficient theorem, the $\mathbb{F}_p$-Betti number in degree $i$ is greater than or equal to the rational Betti number in degree $i$: 
$$\dim_{\QQ}H_i(B_p(\mathcal{S}_{g,1});\QQ)\leq \dim_{\FF_p}H_i(B_p(\mathcal{S}_{g,1});\FF_p) .$$
On the other hand, Proposition \ref{Fpsurface} and the known rational calculation show that the sum of the $\mathbb{F}_p$-Betti numbers coincides with the sum of the rational Betti numbers. Since these numbers are all non-negative, it follows that corresponding Betti numbers are in fact equal, and the claim follows by a second invocation of the universal coefficient theorem. 
\end{proof}

\newpage
\bibliography{There.bib}
\bibliographystyle{alpha}

\end{document}